\documentclass[a4paper]{amsart}
\usepackage[margin=2cm]{geometry}
\usepackage[T1]{fontenc}
\usepackage{hyperref}
\usepackage{lmodern}

\usepackage{amssymb,amsmath,amsthm}
\usepackage{mathrsfs}
\usepackage{graphicx}
\usepackage{enumerate,color,xcolor}
\usepackage{stmaryrd}
\usepackage[normalem]{ulem}
\usepackage[inline]{enumitem}

\newtheorem{thm}{Theorem}[section]
\newtheorem{defi}{Definition}[section]
\newtheorem{lem}{Lemma}[section]

\newtheorem{rmk}{Remark}[section]
\newtheorem{cor}{Corollary}[section]
\newtheorem{prop}{Proposition}[section]

\newcommand\ds{\displaystyle}

\makeatletter
\newcommand{\opnorm}{\@ifstar\@opnorms\@opnorm}
\newcommand{\@opnorms}[1]{
  \left|\mkern-1.5mu\left|\mkern-1.5mu\left|
   #1
  \right|\mkern-1.5mu\right|\mkern-1.5mu\right|
}
\newcommand{\@opnorm}[2][]{
  \mathopen{#1|\mkern-1.5mu#1|\mkern-1.5mu#1|}
  #2
  \mathclose{#1|\mkern-1.5mu#1|\mkern-1.5mu#1|}
}
\makeatother

\numberwithin{equation}{section}
\newcommand{\beq}{\begin{equation*}}
\newcommand{\eeq}{\end{equation*}}
\newcommand{\ben}{\begin{equation}}
\newcommand{\een}{\end{equation}}
\newcommand{\beno}{\begin{eqnarray*}}
\newcommand{\eeno}{\end{eqnarray*}}
\let\f=\frac

\newcommand{\N}{{\mathbb N}}
\newcommand{\Z}{{\mathbb Z}}
\newcommand{\R}{{\mathbb R}}

\newcommand{\pa}{\partial}

\definecolor{reviewertwo}{RGB}{0,100,0}

\newcommand{\ba}{\begin{aligned}}
\newcommand{\ea}{\end{aligned}}
\def\na{\nabla}

\def\cD{{\mathcal D}}

\def\cI{{\mathcal I}}
\def\cJ{{\mathcal J}}

\def\cP{{\mathcal P}}

\def\CC{\mathfrak{C}}
\def\RR{\mathfrak{R}}
\def\NN{\mathfrak{N}}

\def\virgp{\raise 2pt\hbox{,}}
\def\cdotpv{\raise 2pt\hbox{;}}

\def\C{\mathop{\mathbb C\kern 0pt}\nolimits}
\def\DD{\mathop{\mathbb D\kern 0pt}\nolimits}
\def\EE{\mathop{{\mathbb E \kern 0pt}}\nolimits}
\def\K{\mathop{\mathbb K\kern 0pt}\nolimits}
\def\N{\mathop{\mathbb N\kern 0pt}\nolimits}
\def\Q{\mathop{\mathbb Q\kern 0pt}\nolimits}
\def\R{\mathop{\mathbb R\kern 0pt}\nolimits}
\def\SS{\mathop{\mathbb S\kern 0pt}\nolimits}
\def\<{\langle}
\def\>{\rangle}

\def\gs{\gtrsim}
\def\ls{\lesssim}
\def\S{\mathbb{S}}

\def\sinc{\operatorname{sinc} }
\newcommand{\br}[1]{\left\langle {#1} \right\rangle}

\newcommand{\rr}[1]{\left( {#1} \right)}

\newcommand{\wei}{{\langle v \rangle}}
\newcommand{\weis}{{\langle v_* \rangle}}

\def\tpn{{\tilde{\cP}_{N}}}

\title[Fourier spectral method for the cutoff Boltzmann equation]{A Fourier Spectral Method for the Cutoff Spatially Homogeneous Boltzmann Equation: Convergence Analysis and Numerical Simulation}
\author{Yanzhi Gui}
\address{Department of Mathematical Sciences, Tsinghua University, Beijing 100084, P.R. China}
\email{guiyz22@mails.tsinghua.edu.cn}
\thanks{Corresponding author: Yanzhi Gui.}
\author{Ling-Bing He}
\address{Department of Mathematical Sciences, Tsinghua University, Beijing 100084, P.R. China}
\email{hlb@tsinghua.edu.cn}
\author{Liu Liu}
\address{Department of Mathematics, The Chinese University of Hong Kong, Hong Kong, China}
\email{lliu@math.cuhk.edu.hk}
\date{}
\begin{document}
\begin{abstract}
This work presents a novel Fourier spectral method for the cutoff spatially homogeneous Boltzmann equation with Maxwellian and hard potentials. We establish finite-time error estimates in a weighted whole-space $L^2$ norm, with explicit dependence on the velocity cutoff $R$ and the number of Fourier modes $N$.
\end{abstract}
\maketitle

\setcounter{tocdepth}{1}
\tableofcontents

\section{Introduction}
This work develops a novel numerical method for the spatially homogeneous Boltzmann equation with Maxwellian and hard potentials, provides rigorous error estimates, and validates the method through simulations. This advances the computational toolkit for a fundamental model in kinetic theory, bridging theoretical analysis and practical computation.

\subsection{A brief review on the Boltzmann equation}

The Boltzmann equation, formulated by Ludwig Boltzmann in the 1870s, governs the time evolution of the particle distribution function \( f(t, v) \), which represents the probability density of particles with velocity \( v \) at time \( t \). In the spatially homogeneous setting (no dependence on position \( x \)), it reads:
\ben\label{1}
\frac{\partial f}{\partial t}(t, v) = Q(f, f)(t, v), \quad v \in \mathbb{R}^d,
\een
where the collision operator \( Q \) is a bilinear integral operator that accounts for binary elastic collisions between particles. It is written in gain-loss form as:
\[
Q(g, f)(v) = \int_{\mathbb{R}^d} \int_{\mathbb{S}^{d-1}} B(|v - v_*|, \cos\theta)\big[ g(v'_*)f(v')  -  g(v_*)f(v) \big] \, d\sigma \, dv_*=:Q^+(g,f)-Q^-(g,f).
\]
Here, \( v \) and \( v_* \) are the velocities of two particles before collision, while \( v' \) and \( v'_* \) are their velocities after collision. For elastic collisions, momentum and kinetic energy are conserved, giving the explicit formulas:
\[
v' = \frac{v+v_*}{2} + \frac{|v-v_*|}{2} \sigma, \quad 
v'_* = \frac{v+v_*}{2} - \frac{|v-v_*|}{2} \sigma, \quad \sigma \in \mathbb{S}^{2},
\]
where the unit vector \( \sigma \) parameterizes all possible post-collision directions. The angle \( \theta \) is the deflection angle, defined by \( \cos\theta = \frac{v-v_*}{|v-v_*|} \cdot \sigma \), and \( B \) is the collision kernel that determines the interaction rate. In this work, we consider the class of kernels:
\[
B(|v-v_*|, \cos\theta) = |v-v_*|^\gamma, \quad \gamma \in [0,1],
\]
which includes Maxwellian potentials (\( \gamma = 0 \)) and hard potentials (\( \gamma > 0 \)) with angular cutoff.

\subsubsection{Macroscopic quantities and conservation laws}

Given a distribution \( f(v) \), its macroscopic moments are defined as:
\begin{equation} \label{cco}
\rho_f(t):= \int_{\R^3} f(t,v) \, dv, \quad
u_f(t) := \frac{1}{\rho_f(t)} \int_{\R^3} v \, f(t,v) \, dv, \quad
T_f(t) := \frac{1}{3 \rho_f(t)} \int_{\R^3} |v-u_f(t)|^2\, f(t,v) \, dv.
\end{equation}
representing density, bulk velocity, and temperature, respectively. For any solution to the Boltzmann equation, these quantities are conserved in time, because the collision operator conserves mass, momentum, and energy. Equivalently, for the collision invariants \( \varphi(v) = 1, v, |v|^2/2 \), we have:
\ben\label{cons}
\int_{\R^3} f(t, v) \, \varphi(v) \, dv = \int_{\R^3} f_0(v) \, \varphi(v) \, dv \quad\text{for the collision invariants}\quad \varphi(v) = 1, v_1, v_2, v_3, \frac{|v|^2}{2},
\een

\subsubsection{Equilibrium and the H-theorem}

Given the macroscopic fields \( (\rho_f, u_f, T_f) \), the local Maxwellian (equilibrium distribution) is defined as:
\ben\label{MaxGen}
\mu_{\rho_f,u_f,T_f}(v) = \frac{\rho_f}{(2\pi T_f)^{3/2}} \, \exp\left({-\frac{|v-u_f|^2}{2T_f}}\right),
\een
 which shares the same mass, momentum, and energy as \( f \). The relative entropy (a measure of distance from equilibrium) is defined by:
\ben\label{DefHt}
H(f|\mu_{\rho_f,u_f,T_f})(t) := \int_{\R^3} \left[ f(t,v) \log \left( \frac{f(t,v)}{\mu_{\rho_f,u_f,T_f}(v)} \right) - f(t,v) + \mu_{\rho_f,u_f,T_f}(v) \right] \, dv.
\een
where \[
D(f) = \frac{1}{4} \iint_{\R^3 \times \R^3} \int_{\S^2} B\left(\left|v-v_*\right|, \sigma\right) \left( f' f'_* - f f_* \right) \log\frac{f' f'_*}{f f_*} \, d\sigma \, dv_{*} \, dv\ge0,
\] is the entropy dissipation rate. This inequality describes the irreversible relaxation of the system toward equilibrium.

\subsubsection{Existing analytical and numerical results}

The rigorous mathematical theory for the spatially homogeneous Boltzmann equation with cutoff hard potentials is well-established. \cite{Carleman1, Carleman2} first proved existence and uniqueness in \( L^1 \cap L^\infty \) for hard-sphere interactions. \cite{Arkeryd1, Arkeryd2} extended the theory to \( L^1 \cap L\log L \) under moment bounds. \cite{MW} (see also \cite{Lu1}) later obtained optimal well-posedness results and sharp conditions for moment propagation. Additional results on moment bounds can be found in \cite{Elmroth, Povzner, Desvillettes}. These works form the classical foundation of the field.

Numerically, deterministic Fourier spectral methods, pioneered by \cite{PP96, PR00}, provide a noise-free alternative to the widely used Direct Simulation Monte-Carlo (DSMC) method by \cite{Bi, Na}. These methods discretize the collision operator over a truncated velocity domain. Subsequent work \cite{MP, FMP} accelerated the algorithm to \( O(M N^d \log N) \) using a Carleman-like representation, though this acceleration is restricted to special kernels such as the hard-sphere model. For general collision models, efficient schemes with complexity \( O(M N^{d+1} \log N) \) were developed in \cite{GHHH, HQ} using low-rank approximations.
In the present paper, the direct \( O(N^{2d}) \) evaluation of \cite{PP96, PR00} is deliberately retained, so that the implementation matches exactly the scheme analyzed in Theorem \ref{thm:mainresult}; see Remark \ref{rmk:cost} in Section \ref{sec:numerical}.

Despite these algorithmic advances, the rigorous error analysis of Fourier spectral methods for the Boltzmann equation remains comparatively scarce: \cite{FM, HQY} establish convergence only toward the periodized problem on the torus; on the complementary, non-periodized route, the localized spectral method of \cite{GT} with its Lagrange-multiplier conservation correction received in \cite{AGT18} a complete convergence theory, with the error in the truncated domain bounded by polynomial rates in the domain size and the number of modes and with long-time convergence to equilibrium. That analysis also treats the domain-size error and the comparison with the whole-space solution through zero extension. In the semi-discrete equation analyzed in \cite{AGT18}, however, the collision term acts on the smooth cutoff $\chi g$ rather than on a finite-mode representative, so the operator being analyzed and the operator being computed do not coincide exactly. A recent work \cite{FGH} on the spatially homogeneous Landau equation demonstrates that rigorous error estimates for the whole-space problem are achievable, providing key motivation for the present work.

\subsection{Notations and function spaces}  We begin with the following notation:
 
 \noindent$\bullet$ We use the notation $a \lesssim b (a \gtrsim b)$ to indicate that there is a uniform constant $C$, which may be different on different lines, such that $a \leq C b (a \geq C b)$. We use the notation $a \sim b$ whenever $a \lesssim b$ and $b \lesssim a$.
 
 \noindent$\bullet$  We denote $C_{a_1, a_2, \cdots, a_n}$ (or $C\left(a_1, a_2, \cdots, a_n\right)$) by a constant depending on the parameters $a_1, a_2, \cdots, a_n$.

  \noindent$\bullet$ $\mathbf{1}_{\Omega}$ is the characteristic function of the set $\Omega$. We use $\br{f, g}:=\int_{\R^3}f(v)\overline{g(v)}\,dv$ to denote the inner product of $f, g$. The Japanese bracket is denoted as $\wei:=\sqrt{1+|v|^2}$.

   \noindent$\bullet$ For a multi-index $\alpha=(\alpha^1,\alpha^2,\alpha^3)\in \Z_{\ge 0}^3$, we denote the partial derivative operator with respect to $\alpha$ as $\pa^\alpha=\pa^{\alpha^1}_1\pa^{\alpha^2}_2\pa^{\alpha^3}_3$. Also, we denote the order $|\alpha|=|\alpha_1|+|\alpha_2|+|\alpha_3|$, and for two multi-indices $\alpha$ and $\beta$, we define $\alpha\le \beta$ by $\alpha^1\le \beta^1,\alpha^2\le \beta^2,\alpha^3\le \beta^3$.
 \smallskip

In the subsequent sections, we present various necessary function spaces used throughout this work. In what follows, we assume that $f=f(v)$. 

\begin{itemize}
\item[(1).] The normalized space $L^1_{1,0,1}$ is defined as
$L_{1,0,1}^1:=\left\{0\le f \in L^1 \mid \rho_f=1, u_f=0, T_f=1\right\}.$
 The $L \log L$ space is defined as  $L \log L :=\left\{f(v)\left|\|f\|_{L \log L}=\int_{\mathbb{R}^3}|f| \log (1+|f|) d v<\infty\right.\right\}.$
\item[(2).] For $p \in [1, \infty)$ and $l \in \mathbb{R}$, we define the weighted $L_l^p$ space as follows:
$L_l^p := \{ f(v) \,  | \, \| f \|_{L_l^p} = ( \int_{\mathbb{R}^3} |f(v)|^p \langle v \rangle^{lp} \, dv )^{1/p} < \infty \}$.
\item[(3).] For real numbers $k$ and $l$, the weighted Sobolev space $H_l^k$ is defined as follows:
 $H_l^k :=  \{ f(v) \mid \|f\|_{H_l^k} =  \| \langle D \rangle^k \langle \cdot \rangle^l f  \|_{L^2} < +\infty  \},$
where $a(D)$ represents a pseudo-differential operator characterized by the symbol $a(\xi)$, given by:
$(a(D) f)(x) := \frac{1}{(2 \pi)^3} \int_{\mathbb{R}^3} \int_{\mathbb{R}^3} e^{i(x-y) \xi} a(\xi) f(y) \, d y \, d \xi.
 $
Similarly, the homogeneous Sobolev space $\dot{H}^k$ is defined as:
 $\dot{H}_l^k :=  \{ f(v) \mid \|f\|_{\dot{H}_l^k} = \| |D|^k \langle \cdot \rangle^l f\,\|_{L^2} < +\infty  \},$
indicating that the relationship $\|f\|_{H^k_l} \sim \|f\|_{L^2_l} + \|f\|_{\dot{H}^k_l}$ holds when $k > 0$.
\end{itemize}

\subsection{Goals, difficulties and Stretegies.} The main purpose of this work is to conduct a thorough numerical investigation of the spatially homogeneous Boltzmann equation for two important classes of interactions: Maxwellian potentials and hard potentials with angular cutoff.  
Mathematically, we focus on:
\smallskip

\noindent{\it Question (I).} Suppose $f=f(t,v)$ is a global smooth solution to \eqref{1} with smooth initial data $f_0$. Given a final time $T>0$ and a tolerance $0<\epsilon\ll1$, can we systematically construct a numerical solution $f^{Num}=f^{Num}(t,v)$ such that 
\beno \sup_{t\in[0,T]}\|f(t,\cdot)-f^{Num}(t,\cdot)\|_{L^2(\R^3)}\ls \epsilon? \eeno 
\noindent{\it Question (II).} How can we design a numerical scheme that strictly mirrors its theoretical formulation while remaining computationally realizable?
\medskip

\subsubsection{Difficulties} The exact solution is natively defined on the unbounded whole space \(\R^3\), whereas the Fourier spectral method mathematically requires a bounded periodic domain (a torus \(\mathcal{D}_L = [-L,L]^3\)).
For the present scheme, we address the following two difficulties.

\noindent $\bullet$ \underline{\it  The gap between whole space and torus.} The domain-size error and comparison with the whole-space solution by zero extension were already treated in \cite{AGT18}. Here the difficulty lies in the periodic weighted stability analysis of our explicitly specified inner and outer filters and localized collision operator, and in closing the error in a weighted whole-space $L^2$ norm on finite time intervals. To state   the periodic problem on the torus $\cD_L=[-L,L]^3$, following   \cite{PP96, PR00}, we have the following definition:
\begin{defi}
 The periodic Boltzmann operator $Q_p$ is formulated as follows:
\ben\label{defQp}
Q_p(g,h):=\int_{[-L,L]^3\times\S^2}B(|z|,\sigma)\mathbf{1}_{|z|\le 2R} (g(v+z^-)h(v+z^+)-g(v+z)h(v))\,dz\,d\sigma=:Q_p^+(g,h)-Q_p^-(g,h),
\een
where
\beq
z=v_*-v,\ \ \ z^+=\f{z+|z|\sigma}{2},\ \ \ z^-=\f{z-|z|\sigma}{2},
\eeq
with $L=\f{3+\sqrt{2}}{2}R$ for any functions $g,h$ defined on the torus $\cD_L$. Here the subscript $p$ indicates periodicity. 
\end{defi}
\begin{rmk} \textnormal{Three remarks are in order:
 \begin{itemize}  
\item[(i).] It is not difficult to check that for functions  $g,h$, compactly supported in the ball $B(0,R)$, the periodic operator $Q_p$ exactly coincides with the original operator $Q$, i.e.,
\[ Q(g,h)=Q_p(g,h). \]
\item[(ii).]   The Fourier-Galerkin spectral method in  \cite{PP96, PR00} is then applied to this periodic operator as:
\ben
\begin{cases}
  \partial_t f_N\,=\,\cP_N Q_p(f_N,f_N), \\
  f_N|_{t=0}=\cP_N(f_0\,\mathbf{1}_{\cD_L}),
\end{cases}
\een
where $\cP_N$ is the orthogonal projection onto the finite-dimensional space $\mathbb{P}_N:=span\{e^{i\f{\pi}{L}k\cdot v}|k\in\cJ_N\}$ in $L^2(\cD_L)$ with $\cJ_N:=\llbracket -N, N-1 \rrbracket^3$. More precisely, for any periodic function $f\in L^2(\cD_L)$, we have  
\ben\label{defcPN}
\cP_Nf=(2L)^{-3}\sum_{k\in\cJ_N}\hat{f}(k)e^{i\f{\pi}{L}k\cdot v}, \quad \text{where} \quad \hat{f}(k):=\int_{\cD_L}f(v)\,\exp\left(-i\f{\pi}{L}k\cdot v\right)\,d v.
\een
We record two facts about this classical projection, which is used in this paper only for the comparison with the method of \cite{PP96, PR00}. First, since $\cJ_N=\llbracket-N,N-1\rrbracket^3$ is not symmetric under $k\mapsto-k$, the projection $\cP_N$ does not in general preserve real-valuedness: for real $f$, the retained mode $k$ with some coordinate $k_j=-N$ has no conjugate partner $-k$ in $\cJ_N$, and precisely these unpaired ``Nyquist-face'' modes generate a nonzero imaginary part. Second, all spectral approximations in our error analysis are performed with the smooth filter $\tpn$ (and $\tpn^{\f12}$) instead; the filter $\Phi$ being even, these are conjugate-symmetric Fourier multipliers.
\item[(iii).]  Standard estimates for  the whole-space collision operator rely heavily on the regular/singular pre-post collisional change of variables \((v, v_*) \to (v', v_*')\). On the periodic torus \(\mathcal{D}_L\), this coordinate transformation is geometrically invalid, breaking the classical analytical framework.
\end{itemize}
}
\end{rmk}

\noindent $\bullet$ \underline{\it  Systematic analysis of the periodic collision operator.} The theories of \cite{FM, HQY} concern the convergence in the number of modes, on a fixed domain, toward the solution of the periodized problem. The theory of \cite{AGT18} addressed the domain-size error and the comparison with the whole-space solution, but there the operator being analyzed and the operator being computed do not coincide. The present work provides the systematic analysis of the periodic collision operator $Q_p$ (Section 2), through which the scheme being analyzed and the scheme being computed coincide exactly (Question (II) of Section 1.3), and the error toward the whole-space solution is closed in the $L^2_k(\R^3)$ norm with the interplay of the two parameters quantified in the unified bound $R^{3/2+\varepsilon}/N+R^{-l}$ (respectively $e^{-c'R^s}$).

\subsubsection{New strategies}\label{sec:newstrategies}  To bridge the gap between $\R^3$ and $\cD_L$, motivated  by \cite{FGH}, we introduce  an artificially truncated equation, where the collision operator is modified by a smooth truncation function $\psi^R\in C^\infty_0(B(0,R))$ that localizes the interactions to a ball of radius \( R \):
\ben
\pa_t u^R=Q(u^R\,\psi^R\,,\,u^R\,\psi^R)\,\psi^R,
\qquad u^R|_{t=0}=f_0\psi^R,
\een
Let $\psi\in C^\infty_0(B(0,1))$ be a fixed radial bump function satisfying
\beq
\mathbf{1}_{\{|v|\le 1/2\}}\le \psi\le \mathbf{1}_{\{|v|\le 1\}},
\eeq
and for any scaling parameter $s>1$, define the dilated function
\beq
\psi^{s}(v):=\psi(v/s)
\eeq
(so that $\psi^s$ is the truncation function of radius $s$; in particular we write $\psi^R$ for $s=R$). Then
\beq
\psi^R\in C^\infty_0(B(0,R))\quad\text{with}\quad \mathbf{1}_{\{|v|\le R/2\}}\le \psi^R\le \mathbf{1}_{\{|v|\le R\}}.
\eeq
Such truncated equation has two advantages:
\begin{itemize}
  \item[(i).] It makes possible to analyze the error between the the solution $u^R$ to the truncated equation  and the original solution $f$ in suitable polynomial weighted norms in $\R^3$ as $ R \to \infty $.
  \item[(ii).]  With the truncated initial datum $f_0\psi^R$, the solution $u^R$ remains supported in $B(0,R)$ and coincides with the solution $u_p^R$ to the analogous truncated equation formulated on the torus:
\ben \pa_t u^R_p=Q_p(u^R_p\,\psi^R\,,\,u^R_p\,\psi^R)\,\psi^R,
\qquad u_p^R|_{t=0}=f_0\psi^R.\een
  Thus, this truncation provides a bridge between the whole-space problem and a localized version that can be treated on the torus.
\end{itemize}

Now  we are in a position propose a {\it refined, direct strategy} that significantly streamlines the proof. More precisely, we directly utilize the truncated whole-space exact solution $f^R := f\psi^R$, distinct from the auxiliary solution $u^R$, as our continuous reference function on the torus $\cD_L$. To  execute this strategy, we introduce two basic tools aimed at capturing the behavior of the numerical solution more accurately near the boundary and within the Fourier mode.

\noindent $\bullet$ \underline{Projection operator.} The smoothing  projection operator $\tpn$ and $\tpn^{\f12}$  on the torus $\cD_L$ are defined as 
\ben\label{deftpn}
\tpn f:=\f{1}{(2L)^3}\sum_{k\in\Z^3}\hat{f}(k)\,\Phi\left(\f{k}{N}\right)\,e^{i\f{\pi}{L}k\cdot v},\ \ \tpn^{\f12} f:=\f{1}{(2L)^3}\sum_{k\in\Z^3}\hat{f}(k)\,\sqrt{\Phi\left(\f{k}{N}\right)}\,e^{i\f{\pi}{L}k\cdot v},
\een
where $\Phi$ is a suitably chosen mode-truncation function satisfying:
\ben
\Phi(-\xi)=\Phi(\xi),\ \ \ \Phi\in C^{\infty}(\R^3), \ \ \ \mathbf{1}_{[-0.9,0.9]^3}(x)\le \Phi(x)\le \mathbf{1}_{[-1,1]^3}(x),\ \ \ \sqrt{\Phi}\in C^{\infty}(\R^3).
\een

\noindent $\bullet$ \underline{Periodic weight function.} For any real number $k$, we define the positive periodic weight function $m_k\in C^\infty(\cD_L)$ by:
\ben\label{eq:mk}
m_k(v)=\begin{cases}
\wei^k, & |v|\le R,\\
\rr{\psi^{2R}(v)\,\wei+(1-\psi^{2R}(v))\, (R+1)}^k, & R<|v|< 2R,\\
(R+1)^k, & |v|\ge 2R.
\end{cases}
\een
The definition is meaningful for every real index $k$, because all three bases are positive. The inequalities $m_k\ge1$, $m_k\ge R^k$ on $\cD_L\backslash B(0,R)$ and $\wei^k\sim m_k\ls R^k$ are used only for $k\ge0$, while the identity $m_{k_1}m_{k_2}=m_{k_1+k_2}$ and the derivative bound of Property (3) below hold for all real $k_1,k_2,k$.

Now we can introduce the numerical solution to the equation \eqref{1}:
\begin{defi}[Numerical solutions to \eqref{1}]   We have
\begin{itemize}
\item The numerical solution $f^R_N$ to \eqref{1} is defined as the exact solution to the following discrete system on the torus:
\ben\label{eq:num}
  \partial_t f^R_N=Q^R_N(f^R_N,f^R_N),  \quad f^R_N|_{t=0}=\tpn(f_0\, \psi^R\,\mathbf{1}_{\mathcal{D}_L}).
 \een
Here $Q_N^R$ is the \textit{truncated periodic collision operator in Fourier space} defined by
\ben\label{defQNR} Q_N^R(g,h):=\tpn^{\f12}\left(\tpn^{\f12}\left(Q_p(G_N,H_N )\right)\, \psi^R\right), \een   
where inputs are defined as $G_N:=\tpn(g\psi^R),H_N:=\tpn(h\psi^R)$, and $Q_p$ is given by \eqref{defQp}. 
\item By zero-extending $f^R_N$ outside $\cD_L$:
\ben\label{eq:fN}
f_N(t,v)\,:=\,
\left\{\begin{array}{ll}
      \ds   f^{R}_N(t,v), & {\rm if }\, v\in \cD_L,\\[0.5em]
         \ds 0, &  {\rm if }\, v\in \R^3\setminus \cD_L,
\end{array}\right.
\een
we call $f_N$ the \textit{global numerical approximation on $\mathbb{R}^3$} to the exact solution of the original equation. 
\end{itemize}
\end{defi}

\begin{rmk}  \textnormal{The outer two layers of projection in \eqref{defQNR} ensure that our theoretical formulation identically matches the actual executable spectral algorithm. The inner projection $\tpn^{1/2} (Q_p)$ enables us to employee the Fourier spectral method to calculate the collision operator, while the outer projection $\tpn^{1/2}$ guarantees that the final output remains confined to the finite space $\mathbb{P}_N$. This directly addresses {\bf Question (II)}. Moreover,  the nested composition is geometrically designed to satisfy $\tpn^{1/2}\circ \tpn^{1/2}=\tpn$. This structure allows us to fully use the coercivity of the collision operator $Q$ during the estimates.
}
\end{rmk}
  \begin{rmk}  \textnormal{Two mathematical tools: a smoothing projection operator $\tpn$ and a smoothly patched periodic weight function $m_k$, are introduced to transfer analytical techniques from the whole space $\R^3$ to the periodic torus $\cD_L$. 
    The ingenuity of $m_k$ lies in the fact that it coincides with the standard polynomial weight $\wei^k$ within the domain $B(0,R)$, while remaining a smooth and strictly periodic function on $\cD_L$. This structure enables us to rigorously estimate the commutator between $\tpn$ and $m_k$ (see Lemma \ref{thm:comm}). Consequently, we can avoid heavy intermediate truncation equations and directly execute polynomial-weighted norm estimates on the torus, generalizing classical whole-space methods for the propagation of moments to the periodic setting.}
\end{rmk}

\subsubsection{Main result} We first state a proposition regarding the exact solution to the equation \eqref{1}.

\begin{prop}\label{thm:fHnl} Suppose that the initial data $0\le f_0\in L^1_{1,0,1}(\R^3)\cap L\log L(\R^3)$  to \eqref{1}
satisfies
\ben\label{TAES}
 f_0\in \bigcap_{\ell\ge 0} H^1_{\ell}.
\een
The equation \eqref{1} admits a unique global solution $f\in L^\infty([0,\infty);H^1_\ell(\R^3))$ with initial data $f_0$ for any positive number $\ell$.
\end{prop}
\begin{rmk} \textnormal{For $0<\gamma\le1$, the proposition follows from Theorem \ref{thm:Hs}, with the auxiliary-weight choices recorded there. For Maxwell molecules $\gamma=0$, we use the polynomial-weight propagation recorded in Section~4 of \cite{AGT} as following by analogous arguments, with all polynomial $H^1$ initial weights as assumed in \eqref{TAES} and finite initial entropy. This attribution is explained after Theorem \ref{thm:AGTsob} in the appendix. The exponential-weight theorem quoted there has its own exponential initial-data condition and is not used to infer polynomial propagation by comparison of weights.}
\end{rmk}
 \begin{rmk}\textnormal{ Indeed, our main theorem stated below requires only that the initial data satisfy a finite polynomial decay property. However, for technical reasons, we are unable to make this quantitative. We therefore impose the rapid decay property \eqref{TAES} on the initial data $f_0$.}
\end{rmk}

We are now positioned to state our main theorem, which provides the first error estimate between the numerical approximation generated by the numerical method and the exact continuous solution of the Boltzmann equation.
\begin{thm}
\label{thm:mainresult} Let $f$ be a solution as in Proposition \ref{thm:fHnl} with initial data $f_0$, and let $f_N$ be defined in \eqref{eq:fN} as the zero-extension of $f_N^R$, where $f_N^R$ satisfies \eqref{eq:num} with initial data $\tpn(f_0\, \psi^R\,\mathbf{1}_{\mathcal{D}_L})$.  Here $L = \frac{3+\sqrt{2}}{2}R$ with $R>0$ is the velocity domain truncation, and  $N$ is the number of Fourier modes. Set $\mathsf{H}:=\|f_0\|_{L^1_2}+\|f_0\|_{L\log L}$.  Then there exists a sufficiently large constant $k=k(\mathsf{H})>0$, depending only on $\mathsf{H}$, such that the following results hold.
\begin{enumerate}
	\item [(i).] For any $T>0$ and some $l>0$, $0<\varepsilon<\f12$, if  
	$R\,\ge\, \RR(T)$  and 
	$N\,\geq\, \NN(R,T)$, where    
	\ben
	\begin{cases}
	\ds\RR(t) := \CC_1\,\exp\rr{\f{\kappa}{l}\,t},\\[0.8em]
	\ds\NN(R,t):=\CC_1\,R^{3/2+\varepsilon}\,\exp(\kappa\,t),
	\end{cases}  \mbox{with}\quad \CC_1=\CC_1(k,l,\varepsilon,f_0), \quad \kappa=\kappa(k,l,\varepsilon,f_0),
	\een
 then there exists a constant $\CC=\CC(k,l,\varepsilon,f_0)$  such that  the error bound holds for  all  $t\,\in\, [0,T]$:
	\ben
	\label{eq:errorEstimate}
	\|f_N(t)-f(t)\|_{L^2_k(\R^3)}\,\le\, \CC\,\rr{\f{R^{3/2+\varepsilon}}{N}+R^{-l}}\,\exp\rr{\kappa\,t}\le \CC.
	\een
	\item [(ii).] Additionally, suppose $f_0$ exhibits exponential decay, satisfying that $f_0e^{c\wei^s}\in L^1$
    with $c>0$ and $0<s\le2$.
    Let $a\in(0,c)$ and $C$ be the constants provided by the propagation of exponential moments (Theorem \ref{thm:ex}, quoted in the appendix), and set $c_0:=\f25 a$. For any $0<c'<2^{-s}c_0$, define
    \ben
    \begin{cases}
    \ds\RR(t) := \rr{\f{\kappa t+\log\CC_2}{c'}}^{\f1s}, \\[0.8em]
    \ds\NN(R,t):=\CC_2\,R^{3/2+\varepsilon}\,\exp(\kappa\,t),  
    \end{cases} \mbox{where}\quad \CC_2=\CC_2(k, \varepsilon, c, c', s,f_0),\quad \kappa=\kappa(k, \varepsilon, c, c', s,f_0).
    \een
    Then there exists a constant $\CC'=\CC'(k, \varepsilon, c, c', s,f_0)$   such that, if $R\,\ge\, \RR(T)$ and $N\,\geq\, \NN(R,T)$, the following error bound holds for all $t\,\in\, [0,T]$:
    \ben
    \label{eq:errorEstimate2}
    \|f_N(t)-f(t)\|_{L^2_k(\R^3)}\,\le\, \CC'\,\rr{\f{R^{3/2+\varepsilon}}{N}+\exp(-c'R^s)}\,\exp\rr{\kappa\,t}\le \CC'.
    \een  
\end{enumerate}
\end{thm}

\begin{rmk}  
  \textnormal{The significance of Theorem \ref{thm:mainresult} lies in cleanly decoupling the computational error into two controllable parts: the \emph{mode-truncation (Fourier projection) error} $\mathcal{O}(R^{3/2+\varepsilon}/N)$ and the \emph{velocity truncation error} $\mathcal{O}(R^{-l})$ or $\mathcal{O}(\exp\rr{-c'R^s})$. The exponential factor $e^{\kappa t}$ reflects the error accumulation in this evolutionary equation. Consequently, for any prescribed precision and any finite simulation time $T>0$, one can configure sufficiently large $R$ and $N$ to obtain a bounded and consistent numerical accuracy uniformly over the target time interval $[0,T]$. This provides a rigorous answer to {\bf Question (I)}. Also, this result shows that for fixed number of Fourier modes $N$, there exists an optimal, intermediate range for $R$ where the balance between the mode-truncation and velocity-truncation limits is achieved.}
 \end{rmk}

\begin{rmk} \textnormal{The primary distinction between $(i)$ and $(ii)$ lies in the behavior of the required truncation function $\RR(t)$ and the velocity truncation remainder. Under the additional assumption of exponential decay for the initial data in $(ii)$, the required growth of the truncation radius $R$ is significantly relaxed from exponential to polynomial in time. Concurrently, the velocity truncation error in the final estimate is improved from a polynomial rate to an exponential decay rate.}
 \end{rmk}

\begin{rmk} \textnormal{We conclude with a comment on the practical implications of Theorem \ref{thm:mainresult}. The conditions on $R$ and $N$ there are sufficient but not sharp, and the constants $\CC,\CC'$ --- as well as the weight index $k=k(\mathsf H)$ and, in Part $(ii)$, the propagation rate $a\in(0,c)$ --- are implicit and not designed to be numerically evaluated. The exponential-in-time factor $e^{\kappa t}$ and the resulting thresholds $\RR(T)$, $\NN(R,T)$ are standard artifacts of the Gronwall-type argument applied to a nonlinear equation; they do not reflect the behaviour encountered in practice. As the experiments of Section \ref{sec:numerical} show, the two error mechanisms predicted by the theorem --- the algebraic mode-truncation contribution $R^{3/2+\varepsilon}/N$ and the velocity-truncation remainder --- are clearly visible as qualitative trends: the error decreases then saturates as $N$ grows at fixed $L$, is dominated by mode truncation for large $L$ at fixed $N$, and reaches an intermediate optimal balance between the two, while spectral accuracy is attained for values of $R$ and $N$ far below the theoretical thresholds.}
\end{rmk}

\begin{rmk} \textnormal{If we fully consider the asymptotic behavior of the exact solution $f$ to the spatially homogeneous Boltzmann equation, one may derive that for any given error $\epsilon \ll 1$, there exists a time $T_\epsilon$ such that for all $t \ge T_\epsilon$,
\[\|f(t)-\mu\|_{L^2}\le \epsilon.\]Consequently, we can improve the error estimates \eqref{eq:errorEstimate} and \eqref{eq:errorEstimate2} to global estimates by setting $T := T_\epsilon$ and taking $f_N(t) := \mu$ for all $t \ge T_\epsilon$.
Furthermore, regarding regularity, one may adapt a methodology similar to that in \cite{FGH} to show that, if the exact solution $f$ is sufficiently smooth, then equivalent higher-order regularity estimates can be obtained for the discrete spectral solution $f_N$.}
\end{rmk}

\begin{rmk}\textnormal{The scheme \eqref{eq:num} does not preserve the collision invariants exactly; as recorded at the end of Section~3, the induced defects in the mass, the momentum, and the energy are of the order of the error bound of Theorem \ref{thm:mainresult}. A Lagrangian correction in the spirit of \cite{GT, AGT18} can be constructed on top of the present scheme. Its convergence analysis would require estimating the correction from the weighted error $\|f-f_N\|_{L^2_k(\R^3)}$ and controlling its contribution to the energy estimates. A complete convergence analysis of the corrected scheme is beyond the scope of the present paper.}
\end{rmk}

\begin{rmk} \textnormal{The standard weighted estimates for the Boltzmann collision operator heavily rely on the regular and singular pre/post-collisional changes of variables $(v, v_*) \to (v', v_*')$, which are invalid on a periodic torus. To bypass this incompatibility, we establish analytical bounds for the periodic collision operator $Q_p$ using its \emph{discrete Fourier convolution structure} (see Lemma \ref{thm:Qpghf}). This innovation allows us to estimate the collision operator within the periodic domain. Remarkably, for functions compactly supported in $B(0,R)$, these derived Fourier-based bounds perfectly match the classical polynomial-weighted estimates of the original collision operator $Q$ on $\R^3$.}
 \end{rmk}

\subsection{Organization of the paper}
The article is organized as follows. Section 2 collects classical estimates for the Boltzmann collision operator in both the whole space and the periodic domain. The main result is proved in Section 3, and numerical simulations are presented in Section 4. Technical lemmas and classical results are provided in the appendix.

\section{Preliminary results}

\subsection{Estimates for Boltzmann operators on $\R^3$} We establish the following estimates for Boltzmann operators with polynomial weights:

\begin{thm}\label{thm:Qfgh}Let $0\le \gamma\le 1$, $k\ge 8$, and let $f,g,h$ be functions. Then the following estimates hold:
\begin{itemize}
	\item [(1)] It holds that
\ben\label{eq:Qfgh}
\ba
\left|\br{Q(f,g),h\wei^{2k}}\right|\le \f{C_1}{\sqrt{k}}\|g\|_{L^1_\gamma}\|f\|_{L^2_{k+\gamma/2}}\|h\|_{L^2_{k+\gamma/2}}+C_2\|f\|_{L^1_{\gamma}}\|g\|_{L^2_{k+\gamma/2}}\|h\|_{L^2_{k+\gamma/2}}+ C_k\|f\|_{L^2_k}\|g\|_{L^2_k}\|h\|_{L^2_{k}},
\ea
\een
where $C_1$ and $C_2$ are absolute constants, and $C_k$ depends only on $k$.
	\item[(2)] It holds that
\ben\label{eq:QfgL2}
\ba
\|Q(f,g)\|_{L^2_k}\le \f{C_1}{\sqrt{k}}\|g\|_{L^1_\gamma}\|f\|_{L^2_{k+\gamma}}+C_2\|f\|_{L^1_{\gamma}}\|g\|_{L^2_{k+\gamma}}+C_k\|f\|_{L^2_k}\|g\|_{L^2_k},
\ea
\een
where $C_1$ and $C_2$ are absolute constants, and $C_k$ depends only on $k$.
\end{itemize}
	
\end{thm}
\begin{thm}\label{thm:Qfgg}
Let $f$ be a nonnegative function satisfying $\|f\|_{L^1}>\delta$ and $\|f\|_{L^1_2}+\|f\|_{L\log L}<H$, and let $g$ be any function. Then for $k\ge 8$,
\ben\label{eq:Qfgg}
\ba
&\operatorname{Re}\br{Q(f, g), g\wei^{2 k}}+K\|g\|_{L^2_{k+\gamma/2}}^2\le \f{C_1}{\sqrt{k}}\|g\|_{L^1_{\gamma}}\|f\|_{L^2_{k+\gamma/2}}\|g\|_{L^2_{k+\gamma/2}}+C_k\|f\|_{L^2_k}\|g\|_{L^2_k}^2,
\ea
\een
where $K$ depends only on $\delta$ and $H$, $C_1$ is an absolute constant, and $C_k$ depends only on $k$.
\end{thm}
The proof is provided in Appendix \ref{sec:Boltzmann}.

\subsection{Preliminary results on the torus $\cD_L$.}
Recall that the torus domain $\cD_L$ is defined by $\cD_L=[-L,L]^3$ with $L=\f{3+\sqrt{2}}{2}R$ and $R>1$. In this subsection, we establish several analytical tools that will be essential for the stability analysis of our numerical scheme. We begin with basic facts from Fourier analysis on the torus, then introduce polynomial weight functions adapted to the periodic setting, and finally derive estimates for the smoothing projection operator and its commutator with weight functions. These results form the foundation for controlling the error between the numerical solution and the exact solution in the subsequent stability analysis.

\subsubsection{Fourier analysis} For completeness, we present some basic facts on Fourier series.

\noindent$\bullet$ For any periodic function $f\in L^2(\cD_L)$,
\ben\label{eq:Ffk}
f=\f{1}{(2L)^3}\sum_{k\in\Z^3}\hat{f}(k)e^{i\f{\pi}{L}k\cdot v}\ \ , \text{where} \quad
 \hat{f}(k):=\int_{\cD_L}f(v)e^{-i\f{\pi}{L}k\cdot v}d v.
\een

\noindent$\bullet$   
For functions $f,g\in L^2(\cD_L)$, 
\ben\label{eq:fgfl2}
\br{f,g}_p:=\int_{\cD_L}f(v)\overline{g(v)}\, d v=(2L)^{-3}\sum_{k\in\Z^3}\hat{f}(k)\overline{\hat{g}(k)}, \ \ \|f\|^2_{L^2}=\br{f,f}_p=(2L)^{-3}\sum_{k\in\Z^3}|\hat{f}(k)|^2.
\een

\noindent$\bullet$ Hausdorff-Young inequality: let $1\le p\le 2$ and $\f{1}{p}+\f{1}{q}=1$. Then for $f\in L^1(\cD_L)$,
\ben\label{eq:HY}
\|\hat{f}\|_{\ell^q}=\Big( \sum_{k\in\Z^3}|\hat{f}(k)|^q\Big)^{1/q}\le (2L)^{3/q}\|f\|_{L^p(\cD_L)}.
\een
\noindent$\bullet$ 
The standard Sobolev norms are defined as follows: for $a\ge 0$,
\ben\label{eq:Hda}
\|f\|_{H^a_{per}}^2:=(2L)^{-3}\sum_{k\in\Z^3}|\hat{f}(k)|^2\left( 1+\f{\pi^2|k|^2}{L^2} \right)^{a},\ \ 
\|f\|_{\dot{H}^a_{per}}^2:=(2L)^{-3}\sum_{k\in\Z^3}|\hat{f}(k)|^2\left(\f{\pi|k|}{L} \right)^{2a}.
\een
 
 \noindent$\bullet$ For the projection operator $\cP_N$ (see \eqref{defcPN}) and $a>s$,
\ben
\|(I-\cP_N)f\|_{\dot{H}_{per}^s}\ls \left( \f{L}{N} \right)^{a-s}\|f\|_{\dot{H}_{per}^a},\\
\|\cP_Nf\|_{\dot{H}^{a}_{per}}\ls \left( \f{N}{L} \right)^{a-s}\|\cP_Nf\|_{\dot{H}_{per}^s}.
\een

\subsubsection{Polynomial weights on torus}
Constants in the weight-derivative and projection estimates may depend on the fixed cutoff $\psi$ and filter $\Phi$.
For any $k\ge 0$, we recall the positive weight function $m_k\in C^\infty(\cD_L)$ defined by \eqref{eq:mk}. The function $m_k$ possesses the following properties:
\begin{itemize}
	\item[(1)] $m_k$ is a smooth, positive, periodic function on $\cD_L$.
	\item[(2)] For $k\ge0$ and any $v\in\cD_L$, $\wei^k\sim m_k(v)\ls R^k$. For $k\ge0$ and any $v\in \cD_L\backslash B(0,R)$, $m_k(v)\ge R^k$.
	\item[(3)] For any real number $k\ge 0$ and any multi-index $\alpha\neq 0$,
\beq
|\partial^\alpha m_k(v)|\ls \wei^{k-|\alpha|}\,\mathbf{1}_{|v|\le 2R}\ls m_{k-|\alpha|}(v).
\eeq
	\item[(4)] For any $k_1,k_2\in\R$, $m_{k_1}\,m_{k_2}=m_{k_1+k_2}$.
	\item[(5)] For any function $g$ compactly supported in $B(0,R)$, $g(v)\,\wei^k=g(v)\,m_k(v)$ for $v\in\cD_L$.
\end{itemize}

\smallskip
\noindent\textit{Proof of Property (3).} Since $m_k\in C^\infty(\cD_L)$, it suffices to verify the upper bound on each of the three regions in \eqref{eq:mk}.

(i) On $\{|v|\le R\}$, $m_k=\wei^k=(1+|v|^2)^{k/2}$ is the composition of the quadratic polynomial $q(v):=1+|v|^2$ with $F(q):=q^{k/2}$. Note that $|\na q|\ls\wei$, $|\na^2 q|\ls 1$, and $\na^\beta q\equiv 0$ for $|\beta|\ge 3$. By the multivariate Fa{\`a} di Bruno formula,
\beq
\pa^\alpha m_k=\sum_{\text{partitions of }\alpha}F^{(j)}(q)\,\prod_{i=1}^{j}\pa^{\beta_i}q,\qquad \sum_{i=1}^{j}|\beta_i|=|\alpha|,
\eeq
where $|F^{(j)}(q)|=\big|\f{k}{2}\rr{\f{k}{2}-1}\cdots\rr{\f{k}{2}-j+1}\big|\,q^{k/2-j}\ls_{k,\alpha}\wei^{\,k-2j}$. Let $b_1$ and $b_2$ denote the numbers of blocks with $|\beta_i|=1$ and $|\beta_i|=2$, respectively (blocks with $|\beta_i|\ge3$ vanish). Then $j=b_1+b_2$, $|\alpha|=b_1+2b_2$, and each term is bounded by $\wei^{\,k-2j+b_1}=\wei^{\,k-|\alpha|}$. Hence $|\pa^\alpha m_k(v)|\ls_{k,\alpha}\wei^{\,k-|\alpha|}$.

(ii) On $\{R\le |v|\le 2R\}$, $m_k=g^k$ with $g:=\psi^{2R}\,\wei+(1-\psi^{2R})\,(R+1)$, and $g$ is a convex combination of $\wei\sim R$ and $R+1$, so $g\sim\wei\sim R$. By the definition of the dilation given in Section \ref{sec:newstrategies}, i.e. $\psi^{s}(v)=\psi(v/s)$ with $\psi\in C^\infty_0(B(0,1))$ the fixed bump function, we have $\psi^{2R}(v)=\psi(v/(2R))$, and hence $|\pa^\beta\psi^{2R}|\ls_\beta R^{-|\beta|}$, whence Leibniz's rule gives $|\pa^\beta g|\ls_\beta R^{1-|\beta|}$. Since $g\ge\min\{\wei,R+1\}\ge\br{R}>0$ and $F(t):=t^k$ satisfies $|F^{(j)}(t)|\ls_{k,j}t^{k-j}$ for $t>0$, the multivariate Fa{\`a} di Bruno formula yields
\beq
|\pa^\alpha m_k(v)|=\big|\pa^\alpha(F\circ g)(v)\big|
\ls_{k,\alpha}\sum_{\substack{1\le j\le|\alpha|\\ \beta_1+\cdots+\beta_j=\alpha\\ |\beta_i|\ge1}}
g^{\,k-j}\prod_{i=1}^{j}\big|\pa^{\beta_i}g\big|
\ls_{k,\alpha}R^{\,k-|\alpha|}\ls\wei^{\,k-|\alpha|},
\eeq
where the last step uses $g\asymp\wei\asymp R$ and $\sum_{i=1}^{j}|\beta_i|=|\alpha|$ in each summand.

(iii) On $\{|v|\ge 2R\}$, $m_k\equiv (R+1)^k$ is constant, so $\pa^\alpha m_k=0$ for $\alpha\neq 0$.

Combining (i)--(iii) gives $|\pa^\alpha m_k(v)|\ls\wei^{\,k-|\alpha|}\,\mathbf{1}_{|v|\le 2R}$. Finally, on the support of this indicator, $m_{k-|\alpha|}(v)$ equals $\wei^{\,k-|\alpha|}$ on $\{|v|\le R\}$ and $g^{\,k-|\alpha|}\sim\wei^{\,k-|\alpha|}$ on $\{R\le|v|\le 2R\}$, so the last comparison in (3) holds as well.

\subsubsection{Smoothing projection operator}
We recall the smoothing projection operator $\tilde{\cP}_{N}$ defined by \eqref{deftpn} as follows: for any function $f\in L^2(\cD_L)$,
\ben
\tpn f:=\f{1}{(2L)^3}\sum_{k\in\Z^3}\hat{f}(k)\,\Phi\left(\f{k}{N}\right)\,e^{i\f{\pi}{L}k\cdot v}.
\een
Note that $\tpn$ is not a projection operator, but it possesses the following properties analogous to those of $\cP_N$:
\begin{itemize}
	\item $\tpn f\in\mathbb{P}_N$, hence $\tpn \circ \cP_N=\cP_N\circ\tpn=\tpn$, and $\|\tpn f\|_{L^2}\le \|f\|_{L^2}$.
	\item For $a>s$,
\ben\label{eq:1-P}
\|(I-\tpn)f\|_{\dot{H}_{per}^s}\ls \left( \f{L}{N} \right)^{a-s}\|f\|_{\dot{H}_{per}^a},\\
\|\tpn f\|_{\dot{H}^{a}_{per}}\ls \left( \f{N}{L} \right)^{a-s}\|\tpn f\|_{\dot{H}_{per}^s}.
\een
\end{itemize}
\subsubsection{Estimates for the commutator}We establish the following estimate for the commutator between a weight function and the smoothing projection operator $\tpn$:
\begin{lem}\label{thm:comm}
Let $f\in L^2(\cD_L)$ and let $m\in C^1(\cD_L)$ be a periodic function. When $N>L$, the following estimates for $\tpn (fm)-(\tpn f)m$ hold:
\begin{itemize}
	\item[(1)] There exits a constant depends only on the bump function $\Phi$ such that
\ben
\|\tpn (fm)-(\tpn f)\,m\|_{L^2}\ls \f{L}{N}\|f\|_{L^2}\|\nabla m\|_{L^\infty}.
\een
\item[(2)] When $m=m_k$ is the weight function defined by \eqref{eq:mk} with any real number $k\ge 1$,
\ben
\|\tpn (fm_k)-(\tpn f)\,m_k\|_{L^2}\ls \f{L}{N}\|fm_{k-1}\|_{L^2},
\een
where the constant depends only on $k$, the fixed cutoff $\psi$, and the filter $\Phi$.
\end{itemize}
\end{lem}
\begin{proof}\underline{Step 1. Kernel representation of $\tpn$.}
Consider the Fourier kernel of the operator $\tpn$. For $x\in\cD_L$, by substituting the definition of the Fourier coefficients $\hat{f}(k)$, we obtain
\beq
\ba
\tpn f (x)&=\f{1}{(2L)^3}\sum_{n\in\Z^3}\hat{f}(n)\,\Phi\left(\f{n}{N}\right)\,e^{i\f{\pi}{L}n\cdot x}=\f{1}{(2L)^3}\sum_{n\in\Z^3}\int_{\cD_L}f(v)\,e^{-i\f{\pi}{L}n\cdot v}\,\Phi\left(\f{n}{N}\right)\,e^{i\f{\pi}{L}n\cdot x}\,d v\\
&=\f{1}{(2L)^3}\int_{\cD_L}f(v)\sum_{n\in\Z^3}\Phi\left(\f{n}{N}\right)\,e^{i\f{\pi}{L}n\cdot (x-v)}\,d v=\int_{\cD_L}f(v)\,D_N(x-v)\,d v,
\ea
\eeq
where we define the Fourier kernel of $\tpn$ by 
\ben
D_N(v):=\f{1}{(2L)^3}\sum_{k\in\Z^3}\Phi\left(\f{k}{N}\right)\,e^{i\f{\pi}{L}k\cdot v}.
\een
Thus $\tpn f=f*D_N$, i.e., the smoothing projection operator can be represented as a convolution with the kernel $D_N$.  (Remark: For the projection operator $\cP_N$, the multiplier is $\mathbf{1}_{\cJ_N}(k)$ and the corresponding Dirichlet kernel is $D_N^{\mathrm{hard}}(v)=\f{1}{(2L)^3}\sum_{k\in\cJ_N}e^{i\f{\pi}{L}k\cdot v}$.)

\smallskip

\underline{Step 2. Estimates for the kernel $D_N$.}
We first apply the Poisson summation formula to rewrite $D_N(v)$. Define the inverse Fourier transform of $\Phi$ on $\R^3$ by
\beq
\phi(x)=\Phi^{\vee}(x)=\int_{\R^3}\Phi(\xi)\,e^{2\pi i\xi\cdot x}\,d\xi.
\eeq
Since $\Phi$ is a Schwartz function, $\phi$ is also a Schwartz function. By the Poisson summation formula,
\beq
D_N(v)=\f{1}{(2L)^3}\sum_{n\in\Z^3}\Phi\left(\f{2L}{N}\cdot\f{n}{2L}\right)\,e^{2\pi in\cdot \f{v}{2L}}=\sum_{n\in\Z^3}\left( \Phi\left( \f{2L}{N}\ \cdot \right) \right)^{\vee}(v+2Ln)=\sum_{n\in\Z^3}\phi_{2L/N}\left( (v+2Ln) \right),
\eeq
where we have used the scaling property of the Fourier transform. Here we denote $\phi_{t}(x)=t^{-3}\phi(t^{-1}x)=(\Phi(t\ \cdot\,))^{\vee} (x)$ for any $t>0$. Consequently,
\beq
\int_{\cD_L}D_N(v)\,dv=\int_{\R^3}\phi_{2L/N}(v)\,dv=\int_{\R^3}\phi(v)\,dv=\Phi(0)=1,
\eeq
and 
\beq
\|D_N(v)\|_{L^1(\cD_L)}=\int_{\cD_L}|D_N(v)|\,dv\ls\int_{\R^3}|\phi_{2L/N}(v)|\,dv=\int_{\R^3}|\phi(v)|\,dv=\|\phi\|_{L^1(\R^3)}.
\eeq
For any $s>0$, consider the function $|v|^sD_N(v)$ defined on the torus. For $v\in \cD_L$ and $n\in\Z^3$, when $n\neq 0$ we have $|v+2Ln|\ge L\ge |v|/\sqrt{3}$ (here we regard $v+2Ln\in\R^3$, so $v+2Ln$ lies outside the cube $\cD_L$ except when $n=0$). Thus
\beq
\ba
&\int_{\cD_L}|v|^s\,|D_N(v)|\,dv\ls\sum_{n\in\Z^3}\int_{\cD_L}|v|^s\,|\phi_{2L/N}(v+2Ln)|\,dv\ls \sum_{n\in\Z^3}\int_{\cD_L}|v+2Ln|^s\,|\phi_{2L/N}(v+2Ln)|\,dv\\
&=\int _{\R^3}|v|^s\,|\phi_{2L/N}(v)|\,dv=\left(\f{2L}{N}\right)^s\int_{\R^3}|v|^s\,|\phi(v)|\,dv=\left(\f{2L}{N}\right)^s\|\,|\cdot|^s\,\phi(\cdot)\,\|_{L^1(\R^3)}.
\ea
\eeq
Since $\phi$ is a Schwartz function, $\|\,|\cdot|^s\,\phi(\cdot)\,\|_{L^1(\R^3)}<\infty$ is a finite constant depending only on $s$ and $\Phi$. Therefore,
\ben\label{eq:DNes}
\|\,|\cdot|^s D_N(\cdot)\,\|_{L^1(\cD_L)}\ls \left(\f{L}{N}\right)^s.
\een

\smallskip

\underline{Step 3. Estimates for the commutator.}
In the convolution estimates below, the displacement $x-v$ is understood through its periodic representative in $[-L,L]^3$: $D_N$ is $2L$-periodic and $f,m$ are extended periodically, so joining $v$ to the corresponding lift of $x$ leaves all function values unchanged, and the mean-value and Taylor estimates below use this periodic displacement, whose kernel moments are precisely those in \eqref{eq:DNes}.
For any $x\in \cD_L$,
\beq
(\tpn (fm)-(\tpn f)\,m)(x)=\int_{\cD_L}f(v)\,(m(v)-m(x))\,D_N(x-v)\,d v.
\eeq

\underline{\it{Point (1).}} For any $m\in C^1(\cD_L)$, the mean value theorem implies $|m(v)-m(x)|\le \|\nabla m\|_{L^\infty}|v-x|$ (this also holds for $x,v$ on the torus when $m$ is a periodic function with continuous periodic derivative). Hence,
\beq
|(\tpn (fm)-(\tpn f)\,m)(x)|\le \|\nabla m\|_{L^\infty}\int_{\cD_L}|f(v)|\,|(x-v)\,D_N(x-v)|\,d v.
\eeq
By Young's convolution inequality and \eqref{eq:DNes},
\beq
\|\tpn (fm)-(\tpn f)\,m\|_{L^2}\le \|\nabla m\|_{L^\infty}\|f\|_{L^2}\|\,|\cdot|\,D_N(\cdot)\,\|_{L^1(\cD_L)}\ls \f{L}{N}\|f\|_{L^2}\|\nabla m\|_{L^\infty}.
\eeq

\underline{\it{Point (2).}} When $m=m_k$ for any real number $k\ge 1$, we set $n:=\lfloor k\rfloor+1\in\N$, so that $n\ge2$ and $n-1\le k<n$. We apply the multivariate Taylor expansion of order $n$ with remainder: for any $x,v\in\cD_L$, \beq
m_k(x)-m_k(v)=\sum_{i=1}^{n-1}\sum_{|\alpha|=i}\f{\partial^\alpha m_k(v)}{\alpha!}(x-v)^\alpha+\sum_{|\alpha|=n}\f{\partial^\alpha m_k(\xi)}{\alpha!}(x-v)^\alpha, \eeq
where $\xi=v+\theta(x-v)$ for some $\theta\in(0,1)$, and $\alpha$ denotes a multi-index.
By Property (3) of $m_k$, which holds for any real $k\ge0$ and any multi-index $\alpha$, we have $|\partial^\alpha m_k(v)|\ls m_{k-|\alpha|}(v)$ for $|\alpha|\le n-1\le k$, together with $m_{k-|\alpha|}\le m_{k-1}$ for $|\alpha|\ge1$ (using $m_j\ge1$ and $m_{k-1}=m_{k-|\alpha|}m_{|\alpha|-1}$), as well as $\|\partial^\alpha m_k\|_{L^\infty}\ls\sup_{|v|\le2R}\wei^{\,k-n}\le 1$ when $|\alpha|=n>k$ (here $k-n<0$ and $\wei\ge1$, so the supremum is finite and bounded by one), with constants depending only on $k$, we obtain
\beq
\ba
&|m_k(x)-m_k(v)|\ls \sum_{i=1}^{n-1}|\na^i m_k(v)|\,|x-v|^{i}+\|\na^n m_k\|_{L^\infty}|x-v|^n\\ &\ls m_{k-1}(v)\sum_{i=1}^{n-1}|x-v|^i+|x-v|^n\ls m_{k-1}(v)(|x-v|+|x-v|^n),
\ea
\eeq
where the implicit constant depends only on $k$. Using the same method as in Point (1), we derive
\beq
|(\tpn (fm_k)-(\tpn f)\,m_k)(x)|\ls \int_{\cD_L}|f(v)|\,m_{k-1}(v)\,(|x-v|+|x-v|^n)\,|D_N(x-v)|\,d v.
\eeq
By Young's convolution inequality, \eqref{eq:DNes}, and the condition $N>L$,
\beq
\|\tpn (fm_k)-(\tpn f)\,m_k\|_{L^2}\ls \|f\,m_{k-1}\|_{L^2}\left(\|\,|\cdot|\,D_N(\cdot)\,\|_{L^1(\cD_L)}+\|\,|\cdot|^{n}\,D_N(\cdot)\,\|_{L^1(\cD_L)}\right)\ls \f{L}{N}\|f\,m_{k-1}\|_{L^2},
\eeq
with some constant depending only on $k$ and the bump function $\Phi$ (note that $n=\lfloor k\rfloor+1\le k+1$, so the constant indeed depends only on $k$).
\end{proof}
The commutator estimate in Lemma \ref{thm:comm} provides a fundamental tool for controlling the interaction between the smoothing projection operator and weight functions. As a direct consequence, we obtain the following corollary which will be essential for estimating the weighted norms of the truncated collision operator:
\begin{cor}\label{thm:gnmk}Let $g\in L^2(\cD_L)$ and let $m_k$ be defined by \eqref{eq:mk}. Then the following estimates hold:
\begin{itemize}
	\item [(1)]When $N>L$ with $k\ge 1$,
	\ben\label{eq:gnmk}
	\|\tpn(g)\,m_k\|_{L^2}\ls \|g\,m_k\|_{L^2},
	\een
	where the constant depends only on $k$.
	\item [(2)]When $k\ge 1$ and $N>L$,
	\ben\label{eq:gymk}
	\|\tpn(g\,\psi^R\,m_k)-\tpn(g)\,\psi^R\,m_k\|_{L^2}\ls \f{L}{N}\|g\,m_{k-1}\|_{L^2},
	\een
	where the constant depends only on $k$.
\end{itemize}

\end{cor}
\begin{proof}\underline{\it{Point (1).}} 
Applying Lemma \ref{thm:comm} (2) yields
\beq
\ba
& \|\tpn(g)\,m_k\|_{L^2}\le \|\tpn(g)\,m_k-\tpn(g\,m_k)\|_{L^2}+\|\tpn(g\,m_k)\|_{L^2}\ls \f{L}{N}\|g\,m_{k-1}\|_{L^2}+\|g\, m_k\|_{L^2}\ls \rr{\f{L}{N}+1}\|g\, m_k\|_{L^2}. 
\ea
\eeq
Since $L\le N$, we obtain
\ben
\|\tpn(g)\,m_k\|_{L^2}\ls \|g\,m_k\|_{L^2},
\een
where the constant depends only on $k$. This completes the proof of Point (1).
\smallskip

\underline{\it{Point (2).}} Since
\beq
\|\tpn(g\,\psi^R\,m_k)-\tpn(g)\,\psi^R\,m_k\|_{L^2}\le \|\tpn(g\,\psi^R\,m_k)-\tpn(g\,m_k)\,\psi^R\|_{L^2}+\|\tpn(g\,m_k)\,\psi^R-\tpn(g)\,\psi^R\,m_k\|_{L^2},
\eeq
applying Lemma \ref{thm:comm} (1) to the first term, and noting that $\|\na \psi^R\|_{L^\infty}\ls\f{1}{R}$ and $L/R=\f{3+\sqrt{2}}{2}$ is a constant, we obtain
\beq
\|\tpn(g\,\psi^R\,m_k)-\tpn(g\,m_k)\,\psi^R\|_{L^2}\ls \f{L}{N}\|g\,m_k\|_{L^2}\|\na \psi^R\|_{L^\infty}\ls \f{1}{N}\|g\,m_k\|_{L^2}\ls  \f{L}{N}\|g\,m_{k-1}\|_{L^2}.
\eeq
Here we use $m_k=m_{k-1}m_1\ls L\,m_{k-1}$. For the second term, applying Lemma \ref{thm:comm} (2) and $0\le \psi^R\le 1$ yields
\beq
\|\tpn(g\,m_k)\,\psi^R-\tpn(g)\,\psi^R\,m_k\|_{L^2}\le \|\tpn(g\,m_k)-\tpn(g)\,m_k\|_{L^2}\ls \f{L}{N}\|g\,m_{k-1}\|_{L^2}.
\eeq
This completes the proof of Point (2).
\end{proof}
\medskip
The estimates in Lemma \ref{thm:comm} and Corollary \ref{thm:gnmk} provide the necessary control, and will be used extensively in the subsequent stability analysis, particularly when estimating the error terms arising from the discretization of the Boltzmann collision operator. The key observation is that the smoothing projection operator preserves weighted norms up to a small error of order $L/N$, which can be made arbitrarily small by choosing a sufficiently large truncation parameter $N$.
\subsection{Estimates for periodic Boltzmann collision operator} From \cite{PR00}, we have the following proposition:
\begin{prop}\label{prop:betalm} Let $l,m\in\mathbb{Z}^3$. Then 
\ben Q_p^+(e^{i\f{\pi}{L}l\cdot v},e^{i\f{\pi}{L}m\cdot v})= \beta(l,m)e^{i\f{\pi}{L}(l+m)\cdot v},\ \ \ Q_p^-(e^{i\f{\pi}{L}l\cdot v},e^{i\f{\pi}{L}m\cdot v})= \beta(l,l)e^{i\f{\pi}{L}(l+m)\cdot v},
\een 
where 
\ben\label{eq:defbeta} \beta(l,m)=
(2R)^{3+\gamma}(4\pi)^2\int_0^1r^{2+\gamma}\sinc(\pi \lambda r|l-m|)\,\sinc(\pi \lambda r|l+m|)\,dr.
\een
Here $\lambda=R/L=\f{2}{3+\sqrt{2}}$, and the sinc function is defined by $\sinc(x)=\f{\sin x}{x}$ for $x\neq 0$ and $\sinc(0)=1$.
\end{prop}

As a direct corollary, we have 

\begin{lem}\label{thm:betalm}  For any $g,h,f\in L^2(\cD_L)$,
\beq
Q_p(g,h)=(2L)^{-6}\sum_{l\in\Z^3}\sum_{m\in\Z^3}\hat{g}(l)\hat{h}(m)Q_p(e^{i\f{\pi}{L}l\cdot v},e^{i\f{\pi}{L}m\cdot v})=(2L)^{-6}\sum_{l\in\Z^3}\sum_{m\in\Z^3}\hat{g}(l)\hat{h}(m)(\beta(l,m)-\beta(l,l))e^{i\f{\pi}{L}(l+m)\cdot v}.
\eeq
In particular,
\ben\label{eq:Qffkcon}
\hat{Q}^+_p(g,h)(k)=(2L)^{-3}\sum_{l+m=k}\hat{g}(l)\hat{h}(m)\beta(l,m),\ \  \hat{Q}^-_p(g,h)(k)=(2L)^{-3}\sum_{l+m=k}\hat{g}(l)\hat{h}(m)\beta(l,l).
\een
\end{lem}
We also state the following estimate for $\beta(l,m)$:

\begin{lem}\label{thm:betalmest} For the $\beta(l,m)$ defined in \eqref{eq:defbeta}, the following estimates hold:
\begin{itemize}
	\item[(1)] For any $l,m\in\Z^3$,
\ben\label{eq:betalm}
|\beta(l,m)|\ls R^{3+\gamma}\br{l+m}^{-1}\br{l-m}^{-1}\br{|l+m|-|l-m|}^{-1}.
\een
\item[(2)] For any $0\le\delta<1$ and $m\in\Z^3$,
\ben\label{eq:betam}
\sum_{l\in\Z^3}|l|^\delta |\beta(l,m)|^2\ls R^{2(3+\gamma)}\br{m}^{-2+\delta}.
\een
For any $0\le\delta<1$ and $l\in\Z^3$,
\ben\label{eq:betal}
\sum_{m\in\Z^3}|m|^\delta|\beta(l,m)|^2\ls R^{2(3+\gamma)}\br{l}^{-2+\delta}.
\een
\end{itemize}
	
\end{lem}
The proof is given in Section \ref{sec:proofbeta}.
\medskip

We now establish estimates for the periodic Boltzmann operator. Note that our results hold for general periodic functions $g,h\in L^2(\cD_L)$; therefore, we cannot directly apply the results for the $\R^3$ case.
\begin{lem}\label{thm:Qpghf}
 If $g,h\in L^2(\cD_L)$, for $0\le s< \f12, \f65< p\le 2$ such that $s+\f3p<\f52$ the following upper bound holds:
\ben\label{eq:Qpgh}
\|Q_p^+(g,h)\|_{\dot{H}^s}\ls L^{3-\f3p-s+\gamma} \|g\|_{L^2}\|h\|_{L^p},\ \ \ \|Q_p^+(g,h)\|_{\dot{H}^s}\ls L^{3-\f3p-s+\gamma} \|g\|_{L^p}\|h\|_{L^2}.
\een
 \end{lem}
\begin{proof}[Proof of of Lemma \ref{thm:Qpghf}] By \eqref{eq:Qffkcon} and \eqref{eq:Hda}, for functions $g,h\in L^2(\cD_L)$,
\beq
\ba
\|Q^+_p(g,h)\|_{\dot{H}^s}^2&=(2L)^{-3}\sum_{k\in\Z^3}|\hat{Q}^+_p(g,h)(k)|^2\rr{\f{\pi|k|}{L}}^{2s}\ls(2L)^{-9-2s}\sum_{k\in\Z^3}|k|^{2s}\left|\sum_{l+m=k}\hat{g}(l)\,\hat{h}(m)\,\beta(l,m)\right|^2
\\ &\ls (2L)^{-9-2s}\sum_{k\in\Z^3}\left|\sum_{l+m=k}\rr{|l|^{s}+|m|^{s}}|\hat{g}(l)|\,|\hat{h}(m)|\,|\beta(l,m)|\right|^2\ls L^{-9-2s}(\cI_1+\cI_2),
\ea
\eeq
where $\cI_1,\cI_2$ are defined by
\beno
	\ds \cI_1:=\sum_{k\in\Z^3}\left|\sum_{l+m=k}|l|^{s}\,|\hat{g}(l)|\,|\hat{h}(m)|\,|\beta(l,m)|\right|^2,\,
	\ds \cI_2:=\sum_{k\in\Z^3}\left|\sum_{l+m=k}|m|^{s}\,|\hat{g}(l)|\,|\hat{h}(m)|\,|\beta(l,m)|\right|^2.
\eeno
Here we have used $|l+m|^{2s}\le\rr{|l|^{s}+|m|^{s}}^2$ for $0\le s<\f12$, followed by $(a+b)^2\le 2a^2+2b^2$. \smallskip

For $\cI_1$, Lemma \ref{thm:betalmest} (2) and Cauchy's inequality give
\beq
\ba
\cI_1&\le \sum_{k\in\Z^3}\rr{\sum_{l+m=k}|l|^{2s}\,|\hat{h}(m)|^{2}\,|\beta(l,m)|^{2}}\rr{\sum_{l\in\Z^3}|\hat{g}(l)|^2}\\
&=\rr{\sum_{l\in\Z^3}|\hat{g}(l)|^2}\sum_{m\in\Z^3}\rr{|\hat{h}(m)|^{2}\sum_{l\in\Z^3}|l|^{2s}\,|\beta(l,m)|^{2}}\ls L^{3}\|g\|_{L^2}^2\sum_{m\in\Z^3}|\hat{h}(m)|^{2}\rr{R^{2(3+\gamma)}\br{m}^{-2+2s}}.
\ea
\eeq
Let $q=\f{p}{p-1}$ and $q'=\f{p}{2-p}$. Then $\f1p+\f1q=1$ and $\f2q+\f{1}{q'}=1$. By H\"older's inequality and the Hausdorff-Young inequality \eqref{eq:HY},
\beq
\sum_{m\in\Z^3}|\hat{h}(m)|^{2}\br{m}^{-2+2s}\le \rr{\sum_{m\in\Z^3}|\hat{h}(m)|^q}^{\f2q}\rr{\sum_{m\in\Z^3}\rr{\br{m}^{-2+2s}}^{q'}}^{\f{1}{q'}}\ls \|\hat{h}\|_{\ell^q}^2\ls L^{6/q}\|h\|_{L^p}^2.
\eeq
Here we use that when $s+\f3p<\f52$, $(-2+2s)\,q'<-3$, and thus $\sum_{m\in\Z^3}\br{m}^{(-2+2s)\,q'}\ls 1$. Hence,
\beq
\cI_1\ls L^3\|g\|_{L^2}^2\rr{L^{6/q}\|h\|_{L^p}^2}\rr{R^{2(3+\gamma)}}\ls L^{9+6/q+2\gamma}\|g\|_{L^2}^2\|h\|_{L^p}^2.
\eeq
\smallskip

For $\cI_2$, a similar argument yields
\beq
\ba
\cI_2&\ls  \sum_{k\in\Z^3}\rr{\sum_{l+m=k}|m|^{2s}\,|\hat{h}(m)|^{2}\,|\beta(l,m)|^{2}}\rr{\sum_{l\in\Z^3}|\hat{g}(l)|^2}=\rr{\sum_{l\in\Z^3}|\hat{g}(l)|^2}\sum_{m\in\Z^3}\rr{|m|^{2s}|\hat{h}(m)|^{2}\sum_{l\in\Z^3}\,|\beta(l,m)|^{2}}\\& \ls L^{3}\|g\|_{L^2}^2\sum_{m\in\Z^3}|m|^{2s}|\hat{h}(m)|^{2}\rr{R^{2(3+\gamma)}\br{m}^{-2}}\ls L^{9+6/q+2\gamma}\|g\|_{L^2}^2\|h\|_{L^p}^2.
\ea
\eeq

Therefore, for $Q^+_p$,
\beq
\|Q^+_p(g,h)\|_{\dot{H}^s}^2\ls L^{-9-2s}(\cI_1+\cI_2)\ls L^{-9-2s}\rr{L^{9+6/q+2\gamma}\|g\|_{L^2}^2\|h\|_{L^p}^2}=L^{2(3-\f3p-s+\gamma)} \|g\|_{L^2}^2\|h\|_{L^p}^2.
\eeq

By the symmetry of $Q^+_p$ in its inputs, the same estimate with $g$ and $h$ exchanged follows, completing the proof.
\end{proof}
Taking $s=0$ and $p=\f{6}{5-2\varepsilon}$ for some $0<\varepsilon\le 1$, we obtain the following corollary for the $L^2$ norm of $Q_p^+$:
\begin{cor}
Let $g,h\in L^2(\cD_L)$. For $0<\varepsilon\le 1$ and $p=\f{6}{5-2\varepsilon}$,
\ben\label{eq:QpghL2}
\|Q_p^+(g,h)\|_{L^2}\ls L^{\f12+\gamma+\varepsilon} \|g\|_{L^2}\|h\|_{L^p},\ \ \ \|Q_p^+(g,h)\|_{L^2}\ls L^{\f12+\gamma+\varepsilon}\|g\|_{L^p}\|h\|_{L^2}.
\een
\end{cor}

Note that this result holds for general periodic functions $g,h\in L^2(\cD_L)$ and does not require compact support of $g,h$. In this periodic domain, the regular/singular change of variables $v\to v'$ and $v_* \to v'$ (Lemma \ref{thm:ch}) cannot be applied when estimating $Q^+$.
\medskip 

Using this result, we establish the following estimate for the $m_k$-weighted norm of the truncated periodic collision operator:
\begin{lem}\label{thm:Qpmk}
Let $g,h\in L^2(\cD_L)$ and $k\ge 8$. We assume $R\ge 1$. Then
\ben
	\|Q_p(g,h)\,m_k\|_{L^2}\ls\|g\, m_k\|_{L^2}\|h\,m_{k+1/2+\gamma+\varepsilon}\|_{L^2}+\|h\, m_k\|_{L^2}\|g\,m_{k+1/2+\gamma+\varepsilon}\|_{L^2},
	\een
where the constant depends only on $k$ and $\varepsilon$.
\end{lem}
\begin{proof}
We split the functions $g,h$ by truncation: define $g_i,g_o\in L^2(\cD_L)$ by
\beq
g_i=g\,\mathbf{1}_{|v|< R},\ \ \ g_o=g\,\mathbf{1}_{|v|\ge  R}.
\eeq
The subscripts $i$ and $o$ denote the inner and outer parts of the function, respectively. Similarly, define $h_i,h_o\in L^2(\cD_L)$. Then
 $g=g_i+g_o,\ \ \ h=h_i+h_o.$
 Thus, $Q_p(g,h)$ can be written as three parts:
\beq
Q_p(g,h)=Q_p(g_i,h_i)+Q_p(g_i,h_o)+Q_p(g_o,h).
\eeq
For the first part, since $g_i,h_i$ are both compactly supported on $B(0,R)$,
 $Q_p(g_i,h_i)=Q(g_i,h_i),$
 where we do not distinguish between $g_i,h_i$ and their zero extensions on $\R^3$. By \eqref{eq:QfgL2},
\beq
\ba
\|Q_p(g_i,h_i)\,m_k\|_{L^2(\cD_L)}&\ls \|Q(g_i,h_i)\wei^{k}\|_{L^2(\R^3)}\ls \|g_i\|_{L^2_{k}}\|h_i\|_{L^2_{k+\gamma}}+\|g_i\|_{L^2_{k+\gamma}}\|h_i\|_{L^2_{k}}\\
&=\|g_i\,m_{k+\gamma}\|_{L^2(\cD_L)}\|h_i\,m_{k}\|_{L^2(\cD_L)}+\|g_i\,m_{k}\|_{L^2(\cD_L)}\|h_i\,m_{k+\gamma}\|_{L^2(\cD_L)}.
\ea
\eeq
For the other two parts, the property of $m_k$ gives $R^k\|g_o\|_{L^2}\ls \|g_o\,m_k\|_{L^2}$ and $R^k\|h_o\|_{L^2}\ls \|h_o\,m_k\|_{L^2}$. First, consider the $Q^+$ part. By \eqref{eq:QpghL2}, for any $0<\varepsilon\le 1$ and $p=\f{6}{5-2\varepsilon}$,
\beq
\ba
\|Q^+_p(g_i,h_o)\,m_k\|_{L^2}&\ls R^{k}\|Q^+_p(g_i,h_o)\|_{L^2}\ls R^{k+1/2+\gamma+\varepsilon}\|g_i\|_{L^p}\|h_o\|_{L^2}\\
&\ls \|g_i\|_{L^p}\|h_o\,m_{k+1/2+\gamma+\varepsilon}\|_{L^2}.
\ea
\eeq
Since $6/5<p\le 2$, we have $\|g_i\|_{L^p}\le \|g_i\, m_{2}\|_{L^2}\le \|g_i\, m_k\|_{L^2}$, and therefore
\beq
\|Q_p^+(g_i,h_o)\,m_k\|_{L^2}\ls \|g_i\, m_k\|_{L^2}\|h_o\,m_{k+1/2+\gamma+\varepsilon}\|_{L^2}.
\eeq
Similarly we have $\|Q^+_p(g_o,h)\,m_k\|_{L^2}\ls \|g_o\,m_{k+1/2+\gamma+\varepsilon}\|_{L^2}\|h\,m_k\|_{L^2}$. 

For the $Q^-$ part, by definition we know $Q^-(g,h)\, m_k$ is still a convolution:
\beq
Q^-(g,h)\, m_k(v)=\int_{\cD_L\times \S^2}|z|^\gamma \mathbf{1}_{|z|\le 2R}\,g(v+z)\,h(v)\,m_k(v)\,d\sigma\,dz.
\eeq
We then apply Young's convolution inequality to obtain
\beq
\|Q^-(g,h)\, m_k(v)\|_{L^2}\ls R^{\gamma}\|g\|_{L^1}\|h\,m_k\|_{L^2}.
\eeq
This inequality holds for general periodic functions $g,h\in L^2(\cD_L)$. Consequently, for the latter two terms we obtain
\beq
\|Q^-(g_i,h_o)\,m_k\|_{L^2}\ls R^{\gamma}\|g_i\|_{L^1}\|h_o\,m_k\|_{L^2}\ls \|g_i\, m_2\|_{L^2}\|h_o\,m_{k+\gamma}\|_{L^2}\ls \|g_i\, m_k\|_{L^2}\|h_o\,m_{k+\gamma}\|_{L^2},
\eeq
\beq
\|Q^-(g_o,h)\,m_k\|_{L^2}\ls R^{\gamma}\|g_o\|_{L^1}\|h\,m_k\|_{L^2}\ls R^\gamma\|g_o\, m_2\|_{L^2}\|h\,m_{k+\gamma}\|_{L^2}\ls \|g_o\, m_k\|_{L^2}\|h\,m_{k+\gamma}\|_{L^2},
\eeq
here we use that $k\ge 8>2+\gamma$.
By definition, $\|g\,m_k\|_{L^2}^2=\|g_i\,m_k\|_{L^2}^2+\|g_o\,m_k\|_{L^2}^2$, so that $\|g_i\,m_k\|_{L^2}+\|g_o\,m_k\|_{L^2}\le\sqrt2\,\|g\,m_k\|_{L^2}$, and summing the three parts yields
\beq
\ba
&\|Q_p(g,h)\,m_k\|_{L^2}\ls \|g_i\,m_{k+\gamma}\|_{L^2}\|h_i\,m_{k}\|_{L^2}+\|g_i\,m_{k}\|_{L^2}\|h_i\,m_{k+\gamma}\|_{L^2}+\|g_i\, m_k\|_{L^2}\|h_o\,m_{k+1/2+\gamma+\varepsilon}\|_{L^2}\\&+\|g_o\,m_{k+1/2+\gamma+\varepsilon}\|_{L^2}\|h\,m_k\|_{L^2} 
+\|g_i\, m_k\|_{L^2}\|h_o\,m_{k+\gamma}\|_{L^2}+\|g_o\, m_k\|_{L^2}\|h\,m_{k+\gamma}\|_{L^2}\\
 &\ls \|g\, m_k\|_{L^2}\|h\,m_{k+1/2+\gamma+\varepsilon}\|_{L^2}+\|h\, m_k\|_{L^2}\|g\,m_{k+1/2+\gamma+\varepsilon}\|_{L^2}.
\ea
\eeq
This completes the proof.
\end{proof}

\section{Proof of Theorem \ref{thm:mainresult}}
\label{sec:3}
\setcounter{equation}{0}
\setcounter{figure}{0}
\setcounter{table}{0}

Now we are in a position to prove Theorem \ref{thm:mainresult}.

 \begin{proof}[Proof of Theorem \ref{thm:mainresult}: Part $(i)$]  We shall directly compare the numerical solution $f_N^R$  via equation \eqref{eq:num} with the truncated exact solution $f^R:=f\psi^R$ (the exact solution multiplied by the truncation function), where it solves
\ben
\begin{cases}
  \partial_t f^R\,=\,Q(f,\,f)\psi^R, \\
  f^R|_{t=0}\,=\, f_0\psi^R.
\end{cases}
\een
Define the error function $g:=f_N^R-f^R$ on $\cD_L$, then we have
\ben\label{eq:eqg}
\begin{cases}
  \partial_t g\,=\,Q^R_N\left(f_N^R,\,f_N^R\right)-Q\left(f,\,f\right) \psi^R,
  \\
  g|_{t=0}\,=\,(\tpn-I)(f_0\, \psi^R\,\mathbf{1}_{\mathcal{D}_L}).
\end{cases}
\een

For some positive integer $k>8$, we apply the energy method to $g$ in the weighted space $L^2(\cD_L;m_k)$, where the weight function $m_k$ is defined by \eqref{eq:mk}. We obtain
\ben
\f{1}{2}\f{d}{dt}\|g\, m_k\|_{L^2(\cD_L)}^2=\cI_1+\cI_2+\cI_3,
\een
where we decompose the right-hand side into three terms:
\beq
\begin{cases}
\cI_1\,:=\br{Q^R_N\left(f_N^R\,,\,f_N^R\right)\,,\,g\,m_k^2}-\br{Q_p(\tpn(f_N^R)\,\psi^R\,,\,\tpn(f_N^R)\,\psi^R)\,,\,\tpn(g)\,m_k^2\,\psi^R},\\[0.5em]
\cI_2\,:=\br{Q_p(\tpn(f_N^R)\,\psi^R\,,\,\tpn(f_N^R)\,\psi^R)-Q_p(\tpn(f^R)\,\psi^R\,,\,\tpn(f^R)\,\psi^R)\,,\,\tpn(g)\,m_k^2\,\psi^R},\\[0.5em]
\cI_3\,:=\br{Q_p(\tpn(f^R)\,\psi^R\,,\,\tpn(f^R)\,\psi^R)\,,\,\tpn(g)\,m_k^2\,\psi^R}-\br{Q(f,f)\,\psi^R\,,\,g\,m_k^2}.
\end{cases}
\eeq
The three terms separate different tasks. The term $\cI_1$ collects the discretization defects caused by the placement of the smooth Fourier multipliers, the cutoff, and the weight; $\cI_2$ is the stability contribution, in which the reference-profile term supplies the coercive contribution used in the closure; and $\cI_3$ compares the filtered, localized reference profile with the exact collision term and isolates the velocity-tail remainder. We first establish the bounds \eqref{eq:I1}, \eqref{eq:I2}, and \eqref{eq:I3} below, then combine them into the differential inequality \eqref{eq:odeg3}, close the estimate by a bootstrap and Gr\"onwall's inequality, and finally pass from the torus to the whole-space error.

\underline{Estimate for $\cI_1$.} For simplicity, let $F_N:=\tpn(f_N^R\,\psi^R)$.
The four subterms below account for moving the weight and the cutoff through the smooth multipliers ($\cI_{11}$, $\cI_{12}$) and for replacing the localized, filtered input by $F_N$, one collision argument at a time ($\cI_{13}$, $\cI_{14}$), the common small factor being $L/N$.
 We then split $\cI_1$ as $\cI_1=\cI_{11} +\cI_{12} +\cI_{13} + \cI_{14}$, where
\beq
\begin{cases}
\cI_{11}\,:=\br{Q^R_N\left(f_N^R,\,f_N^R\right)\,,\, g\,m_k^2}-\br{\tpn\rr{Q_p(F_N,F_N)\,\psi^R\,m_k},\,g\,m_k}\\[0.5em]
\cI_{12}\,:=\br{Q_p(F_N,F_N)\,\psi^R\,m_k,\,\tpn(g\,m_k)}-\br{Q_p(F_N,F_N)\,\psi^R\,,\,\tpn(g)\,m_k^2},\\[0.5em]
\cI_{13}\,:=\br{Q_p(F_N,F_N)\,\psi^R-Q_p(\tpn(f_N^R)\,\psi^R,F_N)\,\psi^R\,\,,\,\tpn(g)\,m_k^2},\\[0.5em]
\cI_{14}\,:=\br{Q_p(\tpn(f_N^R)\,\psi^R,F_N)\,\psi^R-Q_p(\tpn(f_N^R)\,\psi^R,\tpn(f_N^R)\,\psi^R)\,\psi^R\,,\,\tpn(g)\,m_k^2}.
\end{cases}
\eeq

The subterm $\cI_{11}$ contains the two nested half-projections of $Q^R_N$; in its estimate below, the commutator of the outer half-filter with the weight ($\cJ_1$) is separated from that of the inner half-filter with the product $\psi^R m_k$ ($\cJ_2$).
For $\cI_{11}$, using the definition of $Q_N^R$ in \eqref{defQNR}, we can write
\beq
\cI_{11}=\br{\tpn^{\f12}\left(\tpn^{\f12}\left(Q_p\left(F_N,F_N \right)\right)\, \psi^R\right)m_k-\tpn\rr{Q_p(F_N,F_N)\,\psi^R\,m_k},\,g\,m_k}
\eeq
considering the first argument of the inner product, we decompose it as
\beq
\left\|\tpn^{\f12}\left(\tpn^{\f12}\left(Q_p\left(F_N,F_N \right)\right)\, \psi^R\right)m_k-\tpn\rr{Q_p(F_N,F_N)\,\psi^R\,m_k}\right\|_{L^2}\le \cJ_1+\cJ_2,
\eeq
with
\ben\label{eq:J1}
\begin{cases}
\cJ_1:=\left\|\tpn^{\f12}\left(\tpn^{\f12}\left(Q_p\left(F_N,F_N \right)\right)\, \psi^R\right)m_k-\tpn^{\f12}\left(\tpn^{\f12}\left(Q_p\left(F_N,F_N \right)\right)\, \psi^R\,m_k\right)\right\|_{L^2},\\[0.5em]
\cJ_2:=\left\|\tpn^{\f12}\left(\tpn^{\f12}\left(Q_p\left(F_N,F_N \right)\right)\, \psi^R\,m_k\right)-\tpn\rr{Q_p(F_N,F_N)\,\psi^R\,m_k}\right\|_{L^2},
\end{cases}
\een
For $\cJ_1$, noting that $\tpn^{\f12}$ is also a smoothing operator, we apply Lemma \ref{thm:comm} (2) to obtain
\beq
\cJ_1\ls \f{L}{N}\| \tpn^{\f12}\left(Q_p\left(F_N,F_N \right)\right)\, \psi^R\,m_{k-1}\|_{L^2}.
\eeq
Since we assume $N>L$, we may apply Corollary \ref{thm:gnmk} (1) to obtain
\beq
\cJ_1\ls \f{L}{N}\| \tpn^{\f12}\left(Q_p\left(F_N,F_N \right)\right)\,m_{k-1}\|_{L^2}\ls \f{L}{N}\|Q_p\left(F_N,F_N \right)m_{k-1}\|_{L^2}.
\eeq
Applying Lemma \ref{thm:Qpmk} yields
\beq
\cJ_1\ls \f{L}{N}\|F_N\,m_{k}\|_{L^2}\|F_N\,m_{k-1/2+\gamma+\varepsilon}\|_{L^2}.
\eeq

For $\cJ_2$, since $\tpn^{\f12}\circ\tpn^{\f12}=\tpn$, Corollary \ref{thm:gnmk} (2) gives 
have
\beq
\ba
&\cJ_2=\left\|\tpn^{\f12}\left(\tpn^{\f12}\left(Q_p\left(F_N,F_N \right)\right)\, \psi^R\,m_k-\tpn^{\f12}\rr{Q_p(F_N,F_N)\,\psi^R\,m_k}\right)\right\|_{L^2}\\
&\le \left\|\tpn^{\f12}\left(Q_p\left(F_N,F_N \right)\right)\, \psi^R\,m_k-\tpn^{\f12}\rr{Q_p(F_N,F_N)\,\psi^R\,m_k}\right\|_{L^2}
 \ls \f{L}{N}\|Q_p(F_N,F_N)\,m_{k-1}\|_{L^2}.
\ea
\eeq
Applying Lemma \ref{thm:Qpmk} again, for any $\varepsilon>0$, we obtain
\beq
\cJ_2\ls \f{L}{N}\|F_N\,m_{k}\|_{L^2}\|F_N\,m_{k-1/2+\gamma+\varepsilon}\|_{L^2}.
\eeq
Combining the estimates for $\cJ_1$ and $\cJ_2$ yields
\beq
\cI_{11}\ls \f{L}{N}\|F_N\,m_{k}\|_{L^2}\|F_N\,m_{k-1/2+\gamma+\varepsilon}\|_{L^2}\|g\,m_k\|_{L^2}\ls \f{R^{1+\varepsilon}}{N}\|F_N\,m_{k}\|_{L^2}\|F_N\,m_{k+\gamma/2}\|_{L^2}\|g\,m_k\|_{L^2},
\eeq
where the last inequality uses $L\ls R$ together with the comparison $m_{k-1/2+\gamma+\varepsilon}\ls R^{\varepsilon}m_{k+\gamma/2}$, which holds for $\gamma\le1$. \smallskip

For $\cI_{12}$, we write
\beq
\cI_{12}=\br{Q_p(F_N,F_N)\,\psi^R\,m_{k-1},\tpn(g\,m_k)m_1-\tpn(g)\,m_{k+1}}.
\eeq
For the first argument, applying Lemma \ref{thm:Qpmk} as before yields
\beq
\|Q_p(F_N,F_N)\,\psi^R\,m_{k-1}\|_{L^2}\ls \|F_N\,m_{k}\|_{L^2}\|F_N\,m_{k-1/2+\gamma+\varepsilon}\|_{L^2}.
\eeq
For the second argument, applying Lemma \ref{thm:comm} (2) twice gives
\beq
\ba
&\|\tpn(g\,m_k)m_1-\tpn(g)\,m_{k+1}\|_{L^2}
 \\&\le \|\tpn(g\,m_k)\,m_1-\tpn(g\,m_{k+1})\|_{L^2}+\|\tpn(g\,m_{k+1})-\tpn(g)\,m_{k+1}\|_{L^2}
 \ls \f{L}{N}\|g\, m_k\|_{L^2}.
\ea
\eeq
This yields the estimate for $\cI_{12}$:
\beq
\cI_{12}\ls \f{L}{N}\|F_N\,m_{k}\|_{L^2}\|F_N\,m_{k-1/2+\gamma+\varepsilon}\|_{L^2}\|g\,m_k\|_{L^2}.
\eeq

For $\cI_{13}$ and $\cI_{14}$, we first estimate the weighted difference between $F_N=\tpn(f_N^R\,\psi^R)$ and $\tpn(f_N^R)\,\psi^R$:
\beq
\left\|\rr{F_N-\tpn(f_N^R)\,\psi^R}\,m_k\right\|_{L^2}\le \|F_N\,m_k-\tpn(f_N^R\,\psi^R\,m_k)\|_{L^2}+\|\tpn(f_N^R\,\psi^R\,m_k)-\tpn(f_N^R)\,\psi^R\,m_k\|_{L^2}.
\eeq
Applying Lemma \ref{thm:comm} (2) to the first term yields
\beq
\|F_N\,m_k-\tpn(f_N^R\,\psi^R\,m_k)\|_{L^2}\ls \f{L}{N}\|f_N^R\,\psi^R\,m_{k-1}\|_{L^2}\le  \f{L}{N}\|f_N^R\,m_{k-1}\|_{L^2},
\eeq
and using Corollary \ref{thm:gnmk} for the second term gives
\beq
\|\tpn(f_N^R\,\psi^R\,m_k)-\tpn(f_N^R)\,\psi^R\,m_k\|_{L^2}\ls \f{L}{N}\|f_N^R\,m_{k-1}\|_{L^2}.
\eeq
Hence
\ben\label{eq:fnrmk}
\left\|\rr{F_N-\tpn(f_N^R)\,\psi^R}\,m_k\right\|_{L^2}\ls\f{L}{N}\|f_N^R\,m_{k-1}\|_{L^2}.
\een
From this result, similarly to the proof of Corollary \ref{thm:gnmk} (1), when $L<N$ we also have
\ben\label{eq:Fnmk}
\|F_N\, m_k\|_{L^2}\ls \|f_N^R\,m_k\|_{L^2}
\een
Returning to $\cI_{13}$ and $\cI_{14}$, Lemma \ref{thm:Qpmk} implies
\beq
\ba
&\cI_{13}=\br{Q_p(F_N-\tpn(f_N^R)\,\psi^R,F_N)\,m_{k-\gamma/2}\,,\,\tpn(g)\,\psi^R\,m_{k+\gamma/2}}\\& \ls \|Q_p(F_N-\tpn(f_N^R)\,\psi^R,F_N)\,m_{k-\gamma/2}\|_{L^2}\|\tpn(g)\,\psi^R\,m_{k+\gamma/2}\|_{L^2}.
\ea
\eeq
Applying Lemma \ref{thm:Qpmk} again to $Q_p$ and using \eqref{eq:fnrmk}, we obtain 
\beq
\ba
&\|Q_p(F_N-\tpn(f_N^R)\,\psi^R,F_N)\,m_{k-\gamma/2}\|_{L^2}\ls \left\|\rr{F_N-\tpn(f_N^R)\,\psi^R}\,m_{k-\gamma/2}\right\|_{L^2}\|F_N\,m_{k+1/2+\gamma/2+\varepsilon}\|_{L^2}\\ &
+\left\|\rr{F_N-\tpn(f_N^R)\,\psi^R}\,m_{k+1/2+\gamma/2+\varepsilon}\right\|_{L^2}\|F_N\,m_{k-\gamma/2}\|_{L^2}
\\ &\ls \f{L}{N}\|f_N^R\,m_{k-\gamma/2-1}\|_{L^2}\|F_N\,m_{k+1/2+\gamma/2+\varepsilon}\|_{L^2}+\f{L}{N}\|f_N^R\,m_{k-1/2+\gamma/2+\varepsilon}\|_{L^2}\|F_N\,m_{k-\gamma/2}\|_{L^2}
\\ &\ls \f{L}{N}R^{1/2+\varepsilon}\|f_N^R\,m_k\|_{L^2}\|F_N\,m_{k+\gamma/2}\|_{L^2}\ls \f{R^{\f32+\varepsilon}}{N}\|f_N^R\,m_k\|_{L^2}\|F_N\,m_{k+\gamma/2}\|_{L^2}.
\ea
\eeq
Here we use \eqref{eq:Fnmk}. Thus we obtain the bound for $\cI_{13}$:
\beq
\cI_{13}\ls \f{R^{\f32+\varepsilon}}{N}\|f_N^R\,m_k\|_{L^2}\|F_N\,m_{k+\gamma/2}\|_{L^2}\|\tpn(g)\,\psi^R\,m_{k+\gamma/2}\|_{L^2}.
\eeq
By the same reasoning, we obtain the same bound for $\cI_{14}$:
\beq
\cI_{14}\ls  \f{R^{\f32+\varepsilon}}{N}\|f_N^R\,m_k\|_{L^2}\|F_N\,m_{k+\gamma/2}\|_{L^2}\|\tpn(g)\,\psi^R\,m_{k+\gamma/2}\|_{L^2}.
\eeq

Combining these terms and using \eqref{eq:Fnmk}, we obtain the estimate for $\cI_1$:
\beq
\ba
&\cI_1=\cI_{11}+\cI_{12}+\cI_{13}+\cI_{14}\\ &\ls \f{R^{1+\varepsilon}}{N}\|f_N^R\,m_{k}\|_{L^2}\|F_N\,m_{k+\gamma/2}\|_{L^2}\|g\,m_k\|_{L^2}+\f{R^{\f32+\varepsilon}}{N}\|f_N^R\,m_k\|_{L^2}\|F_N\,m_{k+\gamma/2}\|_{L^2}\|\tpn(g)\,\psi^R\,m_{k+\gamma/2}\|_{L^2}
\\ &\ls \f{R^{\f32+\varepsilon}}{N}\|f_N^R\,m_k\|_{L^2}\|F_N\,m_{k+\gamma/2}\|_{L^2}\rr{\|g\,m_k\|_{L^2}+\|\tpn(g)\,\psi^R\,m_{k+\gamma/2}\|_{L^2}}.
\ea
\eeq
We introduce the notation $g_N=\tpn(g)$. Since $\psi^R$ is compactly supported in $B(0,R)$, we may write $\|\tpn(g)\,\psi^R\,m_{k+\gamma/2}\|_{L^2}=\|g_N\,\psi^R\|_{L^2_{k+\gamma/2}(\R^3)}$.

Furthermore, since $F_N=\tpn(f_N^R\,\psi^R)$ and $f_N^R=f^R+g$, using \eqref{eq:fnrmk} and $L\le N$ we obtain
\beq
\ba
\|F_N\,m_{k+\gamma/2}\|_{L^2}\ls \|\tpn(f_N^R)\,\psi^R\,m_{k+\gamma/2}\|_{L^2}+\f{L}{N}\|f_N^R\,m_{k+\gamma/2-1}\|_{L^2}\\ \ls \|\tpn(f^R)\,\psi^R\,m_{k+\gamma/2}\|_{L^2}+\|\tpn(g)\,\psi^R\,m_{k+\gamma/2}\|_{L^2}+\|f_N^R\,m_{k}\|_{L^2}.
\ea
\eeq
Using $\psi^R\le 1$ and Corollary \ref{thm:gnmk} (1), we have 
\beq
\|\tpn(f^R)\,\psi^R\,m_{k+\gamma/2}\|_{L^2}\le \|\tpn(f^R)\,m_{k+\gamma/2}\|_{L^2}\ls \|f^R\,m_{k+\gamma/2}\|_{L^2}=\|f\,\psi^R\|_{L^2_{k+\gamma/2}(\R^3)}\le \|f\|_{L^2_{k+1}(\R^3)}.
\eeq
Similarly, we have $\|f_N^R\,m_{k}\|_{L^2}\le \|f\|_{L^2_k(\R^3)}+\|g\,m_k\|_{L^2}$. Combining these results, we obtain the estimate for $\cI_1$:
\ben\label{eq:I1}
\ba
&\cI_1\ls \f{R^{\f32+\varepsilon}}{N}\rr{\|f\|_{L^2_k}+\|g\,m_k\|_{L^2}}\rr{\|f\|_{L^2_{k+1}}+\|g_N\,\psi^R\|_{L^2_{k+\gamma/2}}+\|g\,m_k\|_{L^2}}\rr{\|g_N\,\psi^R\|_{L^2_{k+\gamma/2}}+\|g\,m_k\|_{L^2}}\\
&\ls \f{R^{\f32+\varepsilon}}{N}\|f\|_{L^2_{k+1}}^2\rr{\|g_N\,\psi^R\|_{L^2_{k+\gamma/2}}+\|g\,m_k\|_{L^2}}+\f{R^{\f32+\varepsilon}}{N}\rr{\|f\|_{L^2_{k+1}}+\|g\,m_k\|_{L^2}}\rr{\|g_N\,\psi^R\|_{L^2_{k+\gamma/2}}^2+\|g\,m_k\|_{L^2}^2}.
\ea
\een

\smallskip
\underline{Estimate for $\cI_2$.} Since $\psi^R$ is compactly supported in $B(0,R)$, we have $Q_p(\tpn(f_N^R)\psi^R,\tpn(f_N^R)\psi^R)=Q(\tpn(f_N^R)\psi^R,\tpn(f_N^R)\psi^R)$ on $[-L,L]^3\subset\R^3$. (Here we regard $\tpn(f_N^R)$ as its zero extension to $\R^3$, so $\|\tpn(f_N^R)\|_{L^2(\cD_L)}=\|\tpn(f_N^R)\|_{L^2(\R^3)}$, although $\tpn(f_N^R)$ is not continuous.)
We expand the bilinear difference around the localized reference profile: the reference-profile term $\cI_{21}$ supplies the coercive contribution used in the closure, $\cI_{22}$ collects the projection defect of the reference profile, and $\cI_{23}$ the remaining contribution involving the error.
 Therefore, with $g=f_N^R-f^R$, we split $\cI_2$ as
\beq
\ba
\cI_2&=\br{Q(\tpn(f_N^R)\,\psi^R,\tpn(f_N^R)\,\psi^R)-Q(\tpn(f^R)\,\psi^R,\tpn(f^R)\,\psi^R)\,,\,\tpn(g)\,m_k^2\,\psi^R}=\cI_{21}+\cI_{22}+\cI_{23},
\ea
\eeq
where, introducing the notation $g_N=\tpn(g)$,
\beq
\begin{cases}
\cI_{21}\,:=\br{Q(f^R\,\psi^R,g_N\,\psi^R)\,,\,g_N\,\psi^R\,\wei^{2k}},\\[0.5em]
\cI_{22}\,:=\br{Q((\tpn-I)(f^R)\,\psi^R,g_N\,\psi^R)\,,\,g_N\,\psi^R\,\wei^{2k}},\\[0.5em]
\cI_{23}\,:=\br{Q(g_N\,\psi^R,\tpn(f_N^R)\,\psi^R)\,,\,g_N\,\psi^R\,\wei^{2k}},
\end{cases}
\eeq
since $m_k^2=\wei^{2k}$ on $B(0,R)$. 

\smallskip
For $\cI_{21}$, applying Theorem \ref{thm:Qfgg} and noting that $f^R$ is non-negative, we obtain
\beq
\ba
&\br{Q(f^R\,\psi^R,g_N\,\psi^R)\,,\,g_N\,\psi^R\,\wei^{2k}}+K\|g_N\,\psi^R\|_{L^2_{k+\gamma/2}}^2\\&\le \f{C_1}{\sqrt{k}}\|g_N\,\psi^R\|_{L^1_{\gamma}}\|f^R\,\psi^R\|_{L^2_{k+\gamma/2}}\|g_N\,\psi^R\|_{L^2_{k+\gamma/2}}+C_k\|f^R\,\psi^R\|_{L^2_k}\|g_N\,\psi^R\|_{L^2_{k}}^2,
\ea
\eeq
where the constant $K$ depends only on a lower bound of $\|f^R\psi^R\|_{L^1}$ and an upper bound of $\|f^R\psi^R\|_{L^1_2} +\|f^R\psi^R\|_{L\log L}$. Since the exact solution $f\ge0$ satisfies the conservation laws and the H-theorem, we have
\ben
\int_{\R^3}f(t,v)\,\varphi(v)\,dv=\int_{\R^3}f_0(v)\,\varphi(v)\,dv,\qquad
\varphi(v)=1,\,v,\,\f{|v|^2}{2},
\een
and
\beq
\int_{\R^3}f(t,v)\log f(t,v)\,dv\le \int_{\R^3}f_0\log f_0\,dv.
\eeq
As $f_0\in L^1_{1,0,1}(\R^3)$, this gives $\|f\|_{L^1}=1$ and $\|f\|_{L^1_2}=4$.
Moreover, the solution being an H-solution, the well-known entropy-moment bound
\beq
\|f(t)\|_{L\log L}\ \lesssim\ \int_{\R^3}f\log f\,dv+\bigl(1+\|f\|_{L^1_2}\bigr)
\eeq
holds uniformly in time. Since $f_0\in L\log L(\R^3)$ and $\mathsf H=4+\|f_0\|_{L\log L}$,
\beq
\int_{\R^3}f_0\log f_0\,dv\ \le\ \|f_0\|_{L\log L}\ =\ \mathsf H-4\ <\ \mathsf H,
\eeq
and combining the entropy-moment bound with the H-theorem inequality above and $\|f\|_{L^1_2}=4$ yields
\beq
\|f\|_{L\log L}\ \le\ C\bigl(1+\mathsf H\bigr)=:H_*.
\eeq
We now verify the hypotheses of Theorem \ref{thm:Qfgg} for the truncated function $f^R\psi^R=f\,(\psi^R)^2$. By the definition of $\psi^R$ we have $\psi^R\ge \mathbf{1}_{\{|v|\le R/2\}}$, hence $1-(\psi^R)^2\le \mathbf{1}_{\{|v|>R/2\}}$; since $\wei^2\ge R^2/4$ on $\{|v|>R/2\}$, we obtain, for $R>4\sqrt2$,
\beq
\|f\,(1-(\psi^R)^2)\|_{L^1}\le \|f\,\mathbf{1}_{|v|>R/2}\|_{L^1}\le \f{4}{R^2}\|f\,\mathbf{1}_{|v|>R/2}\|_{L^1_2}\le \f{4}{R^2}\cdot 4=\f{16}{R^2}<\f{1}{2},
\eeq
which yields the lower bound
\beq
\|f^R\psi^R\|_{L^1}=\|f\|_{L^1}-\|f\,(1-(\psi^R)^2)\|_{L^1}>\f{1}{2}.
\eeq
On the other hand, since $0\le(\psi^R)^2\le\psi^R\le 1$ and $t\mapsto t\log(1+t)$ is increasing on $[0,\infty)$,
\beq
\|f^R\psi^R\|_{L^1_2}\le \|f\|_{L^1_2}=4,\qquad
\|f^R\psi^R\|_{L\log L}\le \|f\|_{L\log L}\le H_*,
\eeq
and hence
\beq
\|f^R\psi^R\|_{L^1_2}+\|f^R\psi^R\|_{L\log L}\le 4+H_*<5+H_*.
\eeq
Thus all the hypotheses of Theorem \ref{thm:Qfgg} are satisfied for $f^R\psi^R=f\,(\psi^R)^2$, with $\delta=\f12$ and with the upper bound $5+H_*$; in particular, the coercivity constant depends only on $\mathsf H$. \smallskip
For $\cI_{22}$, we return to the torus and use an estimation similar to that for $\cI_{13}$. Applying Lemma \ref{thm:Qpmk} yields
\beq
\ba
\cI_{22}&=\br{Q_p((\tpn-I)(f^R)\,\psi^R,g_N\,\psi^R)\,m_{k-\gamma/2},\,g_N\,\psi^R\,m_{k+\gamma/2}}\\
&\ls \|(\tpn-I)(f^R)\,\psi^R\,m_{k+1/2+\gamma/2+\varepsilon}\|_{L^2}\|g_N\,\psi^R\,m_{k-\gamma/2}\|_{L^2}\|g_N\,\psi^R\,m_{k+\gamma/2}\|_{L^2}\\ &+\|(\tpn-I)(f^R)\,\psi^R\,m_{k-\gamma/2}\|_{L^2}\|g_N\,\psi^R\,m_{k+1/2+\gamma/2+\varepsilon}\|_{L^2}\|g_N\,\psi^R\,m_{k+\gamma/2}\|_{L^2}.
\ea
\eeq
To estimate the projection error $\|(\tpn-I)(f^R)\, \psi^R\, m_k\|_{L^2}$ for $k\ge1$, we decompose it as follows:
\beq
\|(\tpn-I)(f^R)\, \psi^R\, m_k\|_{L^2}\le \|(\tpn-I)(f^R)\, m_k\|_{L^2}\le \|(\tpn-I)(f^R\,m_k)\|_{L^2}+\|(\tpn f^R)\,m_k-\tpn (f^R\,m_k)\|_{L^2}.
\eeq
For the first term, since $f^R=f\,\psi^R\in H^1(\cD_L)$, the weighted function $f^R\,m_k$ inherits the same regularity. Therefore, by the spectral approximation property, $\|(\tpn-I)(f^R\,m_k)\|_{L^2}\ls \f{R}{N}\|f^R\, m_k\|_{H^1(\cD_L)}$. As $f^R=f\,\psi^R$ is supported in $B(0,R)$, we have $\|f^R\, m_k\|_{H^1(\cD_L)}=\|f^R\, m_k\|_{H^1(\R^3)}\ls \|f\|_{H^1_k(\R^3)}$. For the second term, Lemma \ref{thm:comm} (2) gives
\beq
\|(\tpn f^R)\,m_k-\tpn (f^R\,m_k)\|_{L^2}\ls \f{R}{N}\|f^R\,m_{k-1}\|_{L^2}\le \f{R}{N}\|f\|_{L^2_{k-1}(\R^3)}.
\eeq
Combining these estimates yields
\ben\label{eq:tpnIfR}
\|(\tpn-I)(f^R)\, m_k\|_{L^2}\ls \f{R}{N}\|f\|_{H^1_k(\R^3)}.
\een
Substituting this bound into the estimate for $\cI_{22}$ gives
\beq
\ba
\cI_{22}&\ls \f{R}{N}\|f\|_{H^1_{k+1/2+\gamma/2+\varepsilon}(\R^3)}\|g_N\,\psi^R\,m_{k+1/2+\gamma/2+\varepsilon}\|_{L^2}\|g_N\,\psi^R\,m_{k+\gamma/2}\|_{L^2}
\\ &\ls \f{R^{\f32+\varepsilon}}{N}\|f\|_{H^1_{k+2}(\R^3)}\|g_N\,\psi^R\,m_{k+\gamma/2}\|_{L^2}^2.
\ea
\eeq

\smallskip
The term $\cI_{23}$ involves the bilinear collision operator with the filtered error $g_N=\tpn(g)$. Applying Theorem \ref{thm:Qfgh} yields
\beq
\ba
&\br{Q(g_N\,\psi^R,\tpn(f_N^R)\,\psi^R)\,,\,g_N\,\psi^R\,\wei^{2k}}\\ &\le \f{C_1}{\sqrt{k}}\|\tpn(f_N^R)\psi^R\|_{L^1_\gamma}\|g_N\psi^R\|_{L^2_{k+\gamma/2}}^2+C_2\|g_N\psi^R\|_{L^1_{\gamma}}\|\tpn(f_N^R)\psi^R\|_{L^2_{k+\gamma/2}}\|g_N\psi^R\|_{L^2_{k+\gamma/2}}\\ &+C_k\|\tpn(f_N^R)\psi^R\|_{L^2_k}\|g_N\psi^R\|_{L^2_{k}}^2,
\ea
\eeq
where we use the fact that $g=f_N^R-f^R$ and hence $\tpn(f_N^R)=\tpn(f^R)+g_N$. 

\smallskip
We first record a low-weight estimate for the smooth filter. Let $d(x,v)$ denote the Euclidean length of the periodic displacement between $x$ and $v$. Since $\wei^{\gamma}\le\weis^{\gamma}(1+d(x,v))^{\gamma}$ for $0\le\gamma\le1$, the kernel representation $\tpn h=h*D_N$ and \eqref{eq:DNes} give, for $N>L$,
\beq
\|\tpn h\|_{L^1_\gamma(\cD_L)}\le C_{\gamma,\Phi}\rr{1+\rr{\f LN}^{\gamma}}\|h\|_{L^1_\gamma(\cD_L)}\le C_{\gamma,\Phi}\|h\|_{L^1_\gamma(\cD_L)}.
\eeq
Consequently, using $\tpn(f_N^R)=\tpn(f^R)+\tpn(g)$, $f^R=f\psi^R\le f$, $0\le\psi^R\le1$, and $m_k\ge m_4\asymp\wei^{4}$ with $\wei^{\gamma-4}\in L^2(\R^3)$ for $k\ge8$,
\beq
\|\tpn(f_N^R)\psi^R\|_{L^1_\gamma}\le \|\tpn(f_N^R)\|_{L^1_\gamma(\cD_L)}
\le C_{\gamma,\Phi,\psi}\rr{\|f\|_{L^1_\gamma}+\|g\,m_k\|_{L^2}},
\eeq
where the constant is independent of $k$ and of $R,N$; the additional $\|g\,m_k\|_{L^2}$ contribution is absorbed by the corresponding term below.
We now combine the three terms $\cI_{21}$, $\cI_{22}$, and $\cI_{23}$. For $k\ge 8$, the embedding $L^2_k\subset L^1_\gamma$ holds. Since $g\psi^R$ is compactly supported in $B(0,R)$, we have
\beq
\|g_N\psi^R\|_{L^2_k(\R^3)}\ls \|g_N\,\psi^R\,m_k\|_{L^2(\R^3)}= \|g_N\,\psi^R\,m_k\|_{L^2(\cD_L)}\le \|g_N\,m_k\|_{L^2(\cD_L)}.
\eeq
Applying \eqref{eq:gnmk} gives $\|g_N\psi^R\|_{L^2_k(\R^3)}\ls \|g\,m_k\|_{L^2(\cD_L)}$ and $\|\tpn(f^R)\psi^R\|_{L^2_k}\ls \|f^R\,m_k\|_{L^2(\cD_L)}= \|f\,\psi^R\,\wei^k\|_{L^2(\R^3)}\le \|f\|_{L^2_k}$. Therefore,
\beq
\ba
\cI_2&+K\|g_N\,\psi^R\|_{L^2_{k+\gamma/2}}^2\le \f{C_1}{\sqrt{k}}\|f\|_{L^1_\gamma}\|g_N\,\psi^R\|_{L^2_{k+\gamma/2}}^2+\rr{\f{C_1}{\sqrt{k}}+C_2}\|g\,m_k\|_{L^2}\|g_N\psi^R\|_{L^2_{k+\gamma/2}}^2\\
&+\rr{\f{C_1}{\sqrt{k}}+C_2}\|g\,m_k\|_{L^2}\|f\|_{L^2_{k+\gamma/2}}\|g_N\,\psi^R\|_{L^2_{k+\gamma/2}}+C_k\|g\,m_k\|_{L^2}^3\\ &+C_k\|f\|_{L^2_k}\|g\,m_k\|_{L^2}^2+C_3\f{R^{\f32+\varepsilon}}{N}\|f\|_{H^1_{k+2}(\R^3)}\|g_N\,\psi^R\,m_{k+\gamma/2}\|_{L^2}^2.
\ea
\eeq
Since $\|g\,m_k\|_{L^2}\|f\|_{L^2_{k+\gamma/2}}\|g_N\,\psi^R\|_{L^2_{k+\gamma/2}}\le \epsilon\|g_N\,\psi^R\|_{L^2_{k+\gamma/2}}^2+C_\epsilon\|g\,m_k\|_{L^2}^2\|f\|_{L^2_{k+\gamma/2}}^2$ for some sufficiently small positive $\epsilon$ to be chosen, we obtain
\beq
\ba
\cI_2&+\rr{K-\f{C_1}{\sqrt{k}}\|f\|_{L^1_\gamma}-(C_1+C_2)\|g\,m_k\|_{L^2}-\epsilon-C_3\f{R^{\f32+\varepsilon}}{N}\|f\|_{H^1_{k+2}(\R^3)}}\|g_N\,\psi^R\|_{L^2_{k+\gamma/2}}^2\\ &\le C_k\rr{1+\|f\|_{L^2_{k+\gamma/2}}^2}\|g\,m_k\|_{L^2}^2+C_k\|g\,m_k\|_{L^2}^3.
\ea
\eeq
Here we also use that for $k\ge 8$, $\f{C_1}{\sqrt{k}}+C_2\le C_1+C_2$. We denote $C_0=C_1+C_2$, an absolute constant. 
Since $\|f\|_{L^1_\gamma}\le \|f\|_{L^1_2}=4$, we choose $k$ sufficiently large such that $\f{C_1}{\sqrt{k}}\|f\|_{L^1_\gamma}\le K/4$, and select $\epsilon<K/4$ sufficiently small. This yields
\ben\label{eq:I2}
\ba
\cI_2&+\rr{\f{K}{2}-C_0\|g\,m_k\|_{L^2}-C_3\f{R^{\f32+\varepsilon}}{N}\|f\|_{H^1_{k+2}(\R^3)}}\|g_N\,\psi^R\|_{L^2_{k+\gamma/2}}^2\\ &\le C_k\rr{1+\|f\|_{L^2_{k+\gamma/2}}^2}\|g\,m_k\|_{L^2}^2+C_k\|g\,m_k\|_{L^2}^3.
\ea
\een
Since the constant $K$ depends only on $\delta=\f12$ and on $5+H_*$, hence only on $\mathsf H$, the choice of $k$ also depends only on $\mathsf H$.
\smallskip

We compare the collision of the filtered, localized reference profile with the exact collision term: $\cI_{31}$ is a commutator in the test function, $\cI_{32}$ the projection defect of the weighted collision output, and $\cI_{33}$ the difference of the two collision profiles, in which the velocity-tail remainder is isolated.
\underline{Estimate for $\cI_3$.} We denote $F'_N=\tpn(f^R)\, \psi^R$ and split $\cI_3$ as $\cI_3=\cI_{31}+\cI_{32}+\cI_{33}$, where
\beq
\begin{cases}
\cI_{31}\,:=\br{Q_p(F'_N,F'_N)\,m_k,\, g_N\,m_{k}\,\psi^R-\tpn(g\,m_k\, \psi^R)\,},\\[0.5em]
\cI_{32}\,:=\br{\tpn\rr{Q_p(F'_N,F'_N)\,m_k}-Q_p(F'_N,F'_N)\,m_k\,,\,g\,m_k\,\psi^R},\\[0.5em]
\cI_{33}\,:=\br{Q_p(F'_N,F'_N)\,,\,g\,m_k^2\,\psi^R}-\br{Q(f,f)\,\psi^R\,,\,g\,m_k^2}.\\[0.5em]
\end{cases}
\eeq

For $\cI_{31}$, applying Corollary \ref{thm:gnmk} (2) gives
\beq
\|g_N\,m_{k}\,\psi^R-\tpn(g\,m_k\, \psi^R)\|_{L^2}\ls \f{L}{N}\|g\,m_{k-1}\|_{L^2}.
\eeq
Using Lemma \ref{thm:Qpmk} and noting that $1/2+\gamma+\varepsilon<2$ since $0<\varepsilon<\f12$, we obtain
\beq
\cI_{31}\le \f{L}{N}\|Q_p(F'_N,F'_N)\,m_k\|_{L^2}\|g\,m_{k-1}\|_{L^2}\ls \f{L}{N}\|F'_N\,m_{k+2}\|_{L^2}^2\|g\,m_k\|_{L^2}.
\eeq
Since $\psi^R\le 1$ and $F'_N=\tpn(f^R)\, \psi^R$, Corollary \ref{thm:gnmk} (1) yields
\beq
\|F'_N\,m_{k+2}\|_{L^2}\le \|\tpn(f^R)\,m_{k+2}\|_{L^2}\ls \|f^R\,m_{k+2}\|_{L^2}\le \|f\|_{L^2_{k+2}(\R^3)}.
\eeq
Consequently, the estimate for $\cI_{31}$ is
\beq
\cI_{31}\ls \f{L}{N}\|f\|_{L^2_{k+2}(\R^3)}^2\|g\,m_k\|_{L^2}.
\eeq

\smallskip
For $\cI_{32}$, we use a technique similar to that for $\cI_{22}$ to handle $(I-\tpn)$. First, we have
\beq
\ba
\cI_{32}\le \|(I-\tpn) \rr{Q_p(F'_N,F'_N)\,m_k}\|_{L^2}\|g\,m_k\,\psi^R\|_{L^2}.
\ea
\eeq
Applying \eqref{eq:1-P} yields
\beq
\|(I-\tpn) \rr{Q_p(F'_N,F'_N)\,m_k}\|_{L^2}\ls \f{R}{N}\|Q_p(F'_N,F'_N)\,m_k\|_{\dot{H}^1}\ls \f{R}{N}\|Q_p(F'_N,F'_N)\,m_k\|_{H^1}.
\eeq
Since $F_N'$ is compactly supported in $B(0,R)$, the periodic collision operator coincides with the original one on $[-L,L]^3\subset\R^3$, i.e., $Q_p(F'_N,F'_N)=Q(F'_N,F'_N)$. Therefore,
\beq
\|Q_p(F'_N,F'_N)\,m_k\|_{H^1(\cD_L)}\le C_k\|Q(F_N',F_N')\|_{H^1_k(\R^3)}\ls \|F_N'\|_{H^1_{k+2}(\R^3)}^2\ls \|f\|_{H^1_{k+2}(\R^3)}^2.
\eeq
Indeed, both inputs $F'_N$ are supported in $B(0,R)$, so $Q_p(F'_N,F'_N)$ is supported in $B(0,\sqrt2 R)$, on which $m_k\asymp\wei^k$ and $|\na m_k|\le C_k\wei^{k-1}$ by Properties (2)--(3); the product rule then gives the first inequality. Moreover, $\pa_jQ_p(H,H)=Q_p(\pa_jH,H)+Q_p(H,\pa_jH)$ and \eqref{eq:QfgL2} give $\|Q(H,H)\|_{H^1_k}\ls\|H\|_{H^1_{k+2}}^2$, and the last bound follows from $\pa_jF'_N=\tpn(\pa_jf^R)\psi^R+\tpn(f^R)\pa_j\psi^R$, the weighted boundedness of $\tpn$ (Corollary \ref{thm:gnmk}), and $\|\na\psi^R\|_{L^\infty}\ls R^{-1}$.
Consequently, we obtain the estimate
\beq
\cI_{32}\ls \f{R}{N}\|f\|_{H^1_{k+2}(\R^3)}^2\|g\,m_k\|_{L^2}.
\eeq

\smallskip
For $\cI_{33}$, we rewrite it as an inner product on $\R^3$:
\beq
\cI_{33}=\br{\rr{Q(F'_N,F'_N)-Q(f,f)}\,\wei^k\,,\,g\,\psi^R\,\wei^{k}}\le \|Q(F'_N,F'_N)-Q(f,f)\|_{L^2_k(\R^3)}\|g\,m_k\|_{L^2}.
\eeq
Since $Q$ is a bilinear operator, applying Theorem \ref{thm:Qfgh} (2) yields
\beq
\ba
&\|Q(F'_N,F'_N)-Q(f,f)\|_{L^2_k(\R^3)}\le \|Q(F'_N-f,F'_N)\|_{L^2_k}+\|Q(f,f-F'_N)\|_{L^2_k}\\ 
&\ls \|F'_N-f\|_{L^2_{k+\gamma}}(\|f\|_{L^2_{k+\gamma}}+\|F'_N\|_{L^2_{k+\gamma}}).
\ea
\eeq
By \eqref{eq:gnmk}, we have
\beq
\|F'_N\|_{L^2_{k+\gamma}(\R^3)}=\|\tpn(f^R)\,\psi^R\,m_{k+\gamma}\|_{L^2(\cD_L)}\ls \|f^R\,m_{k+\gamma}\|_{L^2(\cD_L)}=\|f^R\|_{L^2_{k+\gamma}(\R^3)}\le \|f\|_{L^2_{k+\gamma}(\R^3)}.
\eeq
To estimate $F'_N-f$, we decompose the difference as
\beq
\ba
&\|F'_N-f\|_{L^2_{k+\gamma}(\R^3)}\le\|(F'_N-f^R\,\psi^R)m_{k+\gamma}\|_{L^2(\cD_L)}+\|f^R\psi^R-f\|_{L^2_{k+\gamma}(\R^3)}
\\ &=\|(\tpn-I)(f^R)\,\psi^R\,m_{k+\gamma}\|_{L^2(\cD_L)}+\|f\,((\psi^R)^2-1)\|_{L^2_{k+\gamma}(\R^3)}.
\ea
\eeq
Since $|(\psi^R)^2-1|\le \mathbf{1}_{\R^3\backslash B(0,R/2)}$, we have
\ben\label{eq:truner}
\|f\,((\psi^R)^2-1)\|_{L^2_{k+\gamma}(\R^3)}\ls R^{-l}\|f\,((\psi^R)^2-1)\|_{L^2_{k+\gamma+l}(\R^3)}\le R^{-l}\|f\|_{L^2_{k+l+\gamma}(\R^3)}.
\een
Thus the estimate for $F'_N-f$ becomes
\ben\label{eq:FN-f}
\|F'_N-f\|_{L^2_{k+\gamma}(\R^3)}\ls \|(\tpn-I)(f^R)\,m_{k+\gamma}\|_{L^2(\cD_L)}+R^{-l}\|f\|_{L^2_{k+l+\gamma}(\R^3)}.
\een
For the first term, applying \eqref{eq:tpnIfR} yields
\beq
\|(\tpn-I)(f^R)\,m_{k+\gamma}\|_{L^2(\cD_L)}\ls \f{R}{N}\|f\|_{H^1_{k+2}(\R^3)}.
\eeq
Consequently, the estimate for $\cI_{33}$ is
\ben\label{eq:I33}
\cI_{33}\ls \rr{\f{R}{N}\|f\|_{H^1_{k+2}(\R^3)}+R^{-l}\|f\|_{L^2_{k+l+\gamma}(\R^3)}}\|f\|_{L^2_{k+\gamma}}\|g\,m_k\|_{L^2}.
\een

Combining these three parts yields
\ben\label{eq:I3}
\cI_3\ls \rr{\f{R}{N}+R^{-l}}\rr{\|f\|_{H^1_{k+2}(\R^3)}+\|f\|_{L^2_{k+l+\gamma}(\R^3)}}^2\|g\,m_k\|_{L^2}.
\een
\medskip

Having established estimates for the three error terms $\cI_1$, $\cI_2$, and $\cI_3$, we now combine them to obtain the energy inequality for the error function $g$. Summing \eqref{eq:I1}, \eqref{eq:I2}, and \eqref{eq:I3} gives
\ben\label{eq:odeg1}
\ba
&\f{1}{2}\f{d}{dt}\|g\, m_k\|_{L^2}^2+\rr{\f{K}{2}-C_0\|g\,m_k\|_{L^2}-C_3\f{R^{\f32+\varepsilon}}{N}\|f\|_{H^1_{k+2}(\R^3)}}\|g_N\,\psi^R\|_{L^2_{k+\gamma/2}}^2 \\& \ls(1+\|f\|_{L^2_{k+\gamma/2}}^2)\|g\,m_k\|_{L^2}^2+\|g\,m_k\|_{L^2}^3+\f{R^{\f32+\varepsilon}}{N}\|f\|_{L^2_{k+1}}^2\rr{\|g_N\,\psi^R\|_{L^2_{k+\gamma/2}}+\|g\,m_k\|_{L^2}}\\&+\f{R^{\f32+\varepsilon}}{N}\rr{\|f\|_{L^2_{k+1}}+\|g\,m_k\|_{L^2}}\rr{\|g_N\,\psi^R\|_{L^2_{k+\gamma/2}}^2+\|g\,m_k\|_{L^2}^2}.\\
&+\rr{\f{R}{N}+R^{-l}}\rr{\|f\|_{H^1_{k+2}(\R^3)}+\|f\|_{L^2_{k+l+\gamma}(\R^3)}}^2\|g\,m_k\|_{L^2}.
\ea
\een
where the constant depends only on $k,l,\varepsilon$. Then, applying Young's inequality, we estimate the following terms:
\beq
\f{R^{\f32+\varepsilon}}{N}\|f\|_{L^2_{k+1}}^2\rr{\|g_N\,\psi^R\|_{L^2_{k+\gamma/2}}+\|g\,m_k\|_{L^2}}\le \f{K}{4}\|g_N\,\psi^R\|_{L^2_{k+\gamma/2}}^2+\|g\,m_k\|_{L^2}^2+C_{K}\f{R^{3+2\varepsilon}}{N^2}\|f\|_{L^2_{k+1}}^4,
\eeq
\beq
\rr{\f{R}{N}+R^{-l}}\rr{\|f\|_{H^1_{k+2}}+\|f\|_{L^2_{k+l+\gamma}}}^2\|g\,m_k\|_{L^2}\le \|g\,m_k\|_{L^2}^2+\rr{\f{R}{N}+R^{-l}}^2\rr{\|f\|_{H^1_{k+2}}+\|f\|_{L^2_{k+l+\gamma}}}^4.
\eeq
Consequently, \eqref{eq:odeg1} becomes
\beq
\ba
&\f{1}{2}\f{d}{dt}\|g\, m_k\|_{L^2}^2+\rr{\f{K}{4}-C_0\|g\,m_k\|_{L^2}-C_3\f{R^{\f32+\varepsilon}}{N}\rr{\|f\|_{H^1_{k+2}}+\|g\,m_k\|_{L^2}}}\|g_N\,\psi^R\|_{L^2_{k+\gamma/2}}^2 \\& \ls(1+\|f\|_{L^2_{k+\gamma/2}}^2)\|g\,m_k\|_{L^2}^2+\|g\,m_k\|_{L^2}^3+\f{R^{\f32+\varepsilon}}{N}\rr{\|f\|_{L^2_{k+1}}+\|g\,m_k\|_{L^2}}\|g\,m_k\|_{L^2}^2\\
&+\rr{\f{R^{\f32+\varepsilon}}{N}+R^{-l}}^2\rr{\|f\|_{H^1_{k+2}}+\|f\|_{L^2_{k+l+\gamma}}}^4
\ls \rr{1+\|f\|_{L^2_{k+\gamma/2}}^2+\f{R^{\f32+\varepsilon}}{N}\|f\|_{L^2_{k+1}}}\|g\,m_k\|_{L^2}^2\\ &+\rr{1+\f{R^{\f32+\varepsilon}}{N}}\|g\,m_k\|_{L^2}^3+\rr{\f{R^{\f32+\varepsilon}}{N}+R^{-l}}^2\rr{\|f\|_{H^1_{k+2}}+\|f\|_{L^2_{k+l+\gamma}}}^4.
\ea
\eeq
To proceed, we assume that $N$ is sufficiently large such that $N>R^{\f32+\varepsilon}$. Under this condition, we can rewrite \eqref{eq:odeg1} as
\ben\label{eq:odeg2}
\ba
&\f{1}{2}\f{d}{dt}\|g\, m_k\|_{L^2}^2+\rr{\f{K}{4}-C_0\|g\,m_k\|_{L^2}-C_3\f{R^{\f32+\varepsilon}}{N}\rr{\|f\|_{H^1_{k+2}}+\|g\,m_k\|_{L^2}}}\|g_N\,\psi^R\|_{L^2_{k+\gamma/2}}^2\\ &\ls \rr{1+\|f\|_{L^2_{k+1}}^2}\|g\,m_k\|_{L^2}^2+\|g\,m_k\|_{L^2}^3+\rr{\f{R^{\f32+\varepsilon}}{N}+R^{-l}}^2\rr{\|f\|_{H^1_{k+2}}+\|f\|_{L^2_{k+l+\gamma}}}^4.
\ea
\een
We now establish uniform bounds for the exact solution $f$. Since the initial data $f_0$ satisfies Assumption \ref{TAES}, Proposition \ref{thm:fHnl} guarantees that $f\in L^\infty([0,\infty);H^1_{k+2}(\R^3))\cap L^\infty([0,\infty);L^2_{k+l+\gamma}(\R^3))$. Consequently, there exists a constant $M$, depending only on $k,l$ and $f_0$, such that
\ben\label{eq:boundM1}
\sup_{t\ge 0}\rr{\|f\|_{H^1_{k+2}(\R^3)}+\|f\|_{L^2_{k+l+\gamma}(\R^3)}}<M.
\een
Treating $M$ as a constant depending only on $k,l$ and $f_0$, we simplify \eqref{eq:odeg2} to
\ben\label{eq:odeg3}
\ba
\f{1}{2}\f{d}{dt}\|g\, m_k\|_{L^2}^2&+\rr{\f{K}{4}-C_0\|g\,m_k\|_{L^2}-C_3\f{R^{\f32+\varepsilon}}{N}\rr{M+\|g\,m_k\|_{L^2}}}\|g_N\,\psi^R\|_{L^2_{k+\gamma/2}}^2 \\& \ls(1+M^2)\|g\,m_k\|_{L^2}^2+\|g\,m_k\|_{L^2}^3+\rr{\f{R^{\f32+\varepsilon}}{N}+R^{-l}}^2M^4.
\ea
\een
For the initial data $g(0)=(\tpn-I)(f_0\,\psi^R\,\mathbf{1}_{\mathcal{D}_L})$, we decompose
\beq
\ba
\|g(0)\,m_k\|_{L^2}&\le\|(\tpn-I)(f_0\,\psi^R\,m_k)\|_{L^2}+\|\tpn(f_0\,\psi^R)\,m_k-\tpn(f_0\,\psi^R\,m_k)\|_{L^2}\\
&\ls \f{L}{N}\|f_0\,\psi^R\,m_k\|_{H^1(\cD_L)}+\f{L}{N}\|f_0\,\psi^R\,m_{k-1}\|_{L^2}
\ls \f{R}{N}\|f_0\,\psi^R\|_{H^1_k}\ls \f{R}{N}M,
\ea
\eeq
where the first term uses the spectral approximation property of $\tpn$ (as in \eqref{eq:tpnIfR}), the second term uses Lemma \ref{thm:comm}~(2), and the last steps use that $f_0\psi^R$ is supported in $B(0,R)$, on which $m_j\asymp\wei^{\,j}$, together with $L\asymp R$.
Since $M$ is a constant, we can express this estimate as
\beq
\|g|_{t=0}\,m_k\|_{L^2}\le c\,\f{R}{N}.
\eeq
where $c$ is a positive constant depending only on $k,M$. When $N$ is sufficiently large such that 
\ben\label{eq:N}
c\,\f{R}{N}\le \f12\min\rr{\f{K}{8C_0},1},\ \ \ \f{R^{\f32+\varepsilon}}{N}\le \f{K}{8C_3\,(M+1)}
\een
the continuity of $t\to \|g(t)\,m_k\|_{L^2}$, together with the strict inequality $\|g(0)m_k\|_{L^2}\le c\,R/N<\min\{K/(8C_0),1\}$ ensured by \eqref{eq:N}, guarantees the existence of $T_*>0$ defined by
\beq
T_*\,:=\,\sup\left\{t\in[0,T],\qquad
  \sup_{s\in[0,t]}\|g(s)\,m_k\|_{L^2}\,\le\,
  \min\left(\frac{K}{8C_0},1\right) \right\}.
\eeq
On the interval $[0,T_*]$, the coercivity condition
\beq \f{K}{4}-C_0\|g\,m_k\|_{L^2}-C_3\f{R^{\f32+\varepsilon}}{N}\rr{M+\|g\,m_k\|_{L^2}}\ge 0,\eeq
holds, and treating $M$ as a constant, \eqref{eq:odeg3} reduces to
\beq
\f{d}{dt}\|g\, m_k\|_{L^2}^2\le 2\kappa\,\|g\,m_k\|_{L^2}^2+C_4\rr{\f{R^{\f32+\varepsilon}}{N}+R^{-l}}^2,
\eeq
for some constants $\kappa,C_4$ depending only on $k,l,M,\varepsilon$. Applying Gronwall's inequality gives that, for $t\in[0,T_*]$, 
\beq
\ba
\|g(t)\, m_k\|_{L^2}^2\,&\le\,
\rr{\|g|_{t=0}\,m_k\|_{L^2}^2+C_4\rr{\f{R^{\f32+\varepsilon}}{N}+R^{-l}}^2\,\int_0^t e^{-2\kappa s}\,d s}\,e^{2\kappa t}\le\, \CC^2\rr{\f{R^{\f32+\varepsilon}}{N}+R^{-l}}^2\,e^{2\kappa t} ,
\ea
\eeq
where $\CC^2=c^2+\f{C_4}{2\kappa}$. This estimate demonstrates that by choosing $R$ and $N$ sufficiently large such that
\beq
\CC \rr{\f{R^{\f32+\varepsilon}}{N}+R^{-l}}\exp(\kappa T)\le \f12\min\rr{\f{K}{8C_0},1}.
\eeq
A sufficient condition for this bound is $R>\RR(T)$ and $N>\NN(R,T)$, where
\beq
\RR(t)\,:=\,\rr{\f{4\CC\,\exp\left(\kappa\,t\right) }{\min\{K/(8C_0),1\}}}^{1/l},\quad
\NN(R,t):=\f{4\CC \,R^{\f32+\varepsilon}\,\exp\left(\kappa\,t\right) }{\min\{K/(8C_0),1\}}.
\eeq
With these thresholds, the Gr\"onwall bound above is at most half of the bootstrap threshold; if $T_*<T$, the strict inequality at $T_*$ together with the continuation criterion for the finite-dimensional ODE \eqref{eq:num} extends the bound beyond $T_*$, contradicting the definition of $T_*$. Hence $T_*=T$.
These conditions also guarantee $R>4\sqrt{2}$, $L\le N$, and $N>R^{\f32+\varepsilon}$ when the constants are chosen appropriately. Moreover, by taking $\CC$ sufficiently large, condition \eqref{eq:N} is satisfied. Consequently, when $R>\RR(T)$ and $N>\NN(R,T)$, we obtain the estimate:
\beq
\|g(t)\, m_k\|_{L^2}\le\CC\rr{\f{R^{\f32+\varepsilon}}{N}+R^{-l}}\,e^{\kappa t}\le 1,\ \ \forall \, t\in [0,T].
\eeq

Using a classical iteration scheme together with the previous estimate, we establish the well-posedness of \eqref{eq:eqg}, and hence of the numerical equation \eqref{eq:num}.

\smallskip
To complete the proof, we need to extend the error estimate from the torus domain $\cD_L$ to the whole space $\R^3$. The numerical solution $f_N$ is defined as the zero extension of $f_N^R$ outside $\cD_L$. We will estimate the error on two regions: the inner region $B(0,R/2)$ where the numerical solution is well-resolved, and the outer region $\R^3\backslash B(0,R/2)$ where the error is controlled by the decay of the exact solution. We now consider the zero extension $f_N$ defined by \eqref{eq:fN}. Since $m_k=\wei^k$ on $B(0,R)\subset \cD_L$, $f^R=f$ on $B(0,R/2)$, and $g=f_N-f$ on $B(0,R/2)$, we have
\beq
\|f_N-f\|_{L^2_k(B(0,R/2))}=\|g\,m_k\|_{L^2(B(0,R/2))},
\eeq
for all $t\in [0,T]$. For the domain $\R^3\backslash B(0,R/2)$, the decay property of the exact solution implies
\beq
\|f^R\|_{L^2_k(\R^3\backslash B(0,R/2))}\le \|f\|_{L^2_k(\R^3\backslash B(0,R/2))}\le (R/2)^{-l}\|f\|_{L^2_{k+l}(\R^3)}\le (R/2)^{-l}M\ls R^{-l}.
\eeq
On the domain $\R^3\backslash B(0,R/2)$, we have
\beq
\ba
\|f_N-f\|_{L^2_k(\R^3\backslash B(0,R/2))}&\le \|f_N-f^R\|_{L^2_k(\R^3\backslash B(0,R/2))}+\|f^R\|_{L^2_k(\R^3\backslash B(0,R/2))}+\|f\|_{L^2_k(\R^3\backslash B(0,R/2))}
\\ &\le C\|g\,m_k\|_{L^2(\cD_L\backslash B(0,R/2))}+C\,R^{-l},
\ea
\eeq
where $C$ depends only on $M,k,l$. Here we use the facts that $f_N-f^R=0$ on $\R^3\backslash \cD_L$ and $f_N-f^R=g$ on $\cD_L$, together with the inequality $\wei^k\ls L^k\ls m_k$ on $\cD_L\backslash B(0,R/2)$. Combining the estimates on the inner and outer regions yields
\beq
\|f_N-f\|_{L^2_k}\le (C+1)\|g\,m_k\|_{L^2(\cD_L)}+C\,R^{-l}\le (C+1)(\CC+1)\rr{\f{R^{\f32+\varepsilon}}{N}+R^{-l}}\,e^{\kappa t},\ \ \forall \, t\in [0,T].
\eeq
With the new constant $(C+1)(\CC+1)$, this completes the proof of Part $(i)$. \end{proof}
\medskip

\begin{proof}[Proof of Theorem \ref{thm:mainresult}: Part $(ii)$]
Under the assumption that the initial data decays exponentially, Theorem \ref{thm:ex} (the propagation of exponential moments, which applies to hard potentials $0<\gamma\le1$ and to Maxwell molecules $\gamma=0$ alike --- in the latter case through the standing assumption $f_0\in L\log L$ --- so that the present proof covers the full admissible range $0\le\gamma\le1$ of Theorem \ref{thm:mainresult}(ii)) yields constants $a\in(0,c)$ and $C$ such that
\beq
\sup_{t\ge 0}\int_{\R^3}e^{a\wei^s}\,f(t,v)\,dv\le C.
\eeq

Furthermore, Proposition \ref{thm:fHnl} ensures that $\sup_{t\ge 0}\|f(t)\|_{H^1(\R^3)}<\infty$. Using the Sobolev embedding $H^1(\R^3)\hookrightarrow L^6(\R^3)$ alongside H\"older's inequality, we deduce that
\ben
\sup_{t\ge 0}\left\|f(t,\cdot)\,e^{\f25 a\br{\cdot}^s}\right\|_{L^2(\R^3)}<\infty.
\een

Let us denote $c_0=\f25 a$ with $a\in(0,c)$ as above. Without loss of generality, up to enlarging the constant $M$ defined in \eqref{eq:boundM1}, we may assume that such a global-in-time bound is also bounded by $M$.

Since the structural arguments of the proof in Part $(i)$ remain entirely valid, it suffices to refine the estimates concerning the velocity truncation. Specifically, for the estimates of $\cI_1$, $\cI_2$, and $\cI_3$, the only modification occurs in the term $\cI_{33}$, where the difference $F_N'-f$ is evaluated. 
Under the new exponential decay condition, the error bound \eqref{eq:truner} becomes:
\beq
\|f\,((\psi^R)^2-1)\|_{L^2_{k+\gamma}(\R^3)}\ls e^{-c'R^s}\|f\,e^{c_0\br{\cdot}^s}\|_{L^2(\R^3)},
\eeq
where any $0<c'<2^{-s}c_0$ is admissible. Indeed, since $((\psi^R)^2-1)$ is supported in $\{|v|>R/2\}$ and the function $\rho\mapsto c_0\rho^s-(k+\gamma)\log\rho$ is increasing for large $\rho$, it suffices that
\beq
\wei^{k+\gamma}e^{-c_0\wei^s}\le \rr{\f{R}{2}}^{k+\gamma}e^{-c_0(R/2)^s}\le e^{-c'R^s}\qquad\text{for } |v|>\f{R}{2},
\eeq
which holds provided $c'\le 2^{-s}c_0-\f{(k+\gamma)\log(R/2)}{R^s}$; the last term is absorbed by enlarging the constant $\CC_2$ (and hence $\RR(T)$), so that the effective condition is $c'<2^{-s}c_0$.
Consequently, the estimate for $\cI_{33}$ in \eqref{eq:I33} can be refined to:
\beq
\cI_{33}\ls \rr{\f{R}{N}\|f\|_{H^1_{k+2}(\R^3)}+e^{-c'R^s}\|f\,e^{c_0\br{\cdot}^s}\|_{L^2(\R^3)}}\|f\|_{L^2_{k+\gamma}}\|g\,m_k\|_{L^2}.
\eeq

Following identical arguments to those used to derive equations \eqref{eq:odeg1}--\eqref{eq:odeg3}, we obtain the updated differential inequality. The only difference is the substitution of the polynomial term $R^{-l}$ with the exponential term $e^{-c'R^s}$:
\ben\label{eq:odeg4}
\ba
\f{1}{2}\f{d}{dt}\|g\, m_k\|_{L^2}^2&+\rr{\f{K}{4}-C_0\|g\,m_k\|_{L^2}-C_3\f{R^{\f32+\varepsilon}}{N}\rr{M+\|g\,m_k\|_{L^2}}}\|g_N\,\psi^R\|_{L^2_{k+\gamma/2}}^2 \\& \ls(1+M^2)\|g\,m_k\|_{L^2}^2+\|g\,m_k\|_{L^2}^3+\rr{\f{R^{\f32+\varepsilon}}{N}+e^{-c'R^s}}^2M^4.
\ea
\een
We analyze this differential inequality exactly as in Part $(i)$. Let $N$ be sufficiently large, we define $T_*>0$ by 
\beq
T_*\,:=\,\sup\left\{t\in[0,T]\big|
  \sup_{s\in[0,t]}\|g(s)\,m_k\|_{L^2}\,\le\,
  \min\left(\frac{K}{8C_0},1\right) \right\}.
\eeq
On the interval $[0,T_*]$, the coercivity condition holds:
\beq \f{K}{4}-C_0\|g\,m_k\|_{L^2}-C_3\f{R^{\f32+\varepsilon}}{N}\rr{M+\|g\,m_k\|_{L^2}}\ge 0,\eeq
Treating $M$ as a constant, \eqref{eq:odeg4} reduces to the differential inequality:
\beq
\f{d}{dt}\|g\, m_k\|_{L^2}^2\le 2\kappa\,\|g\,m_k\|_{L^2}^2+C_4\rr{\f{R^{\f32+\varepsilon}}{N}+e^{-c'R^s}}^2,
\eeq
where $\kappa$ and $C_4$ are constants depending only on $k, M, \varepsilon, c, c'$, and $s$. Applying Gronwall's inequality yields that, for $t\,\in[0,T_*],$
\beq
\ba
\|g(t)\, m_k\|_{L^2}^2\,&\le\,
\rr{\|g|_{t=0}\,m_k\|_{L^2}^2+C_4\rr{\f{R^{\f32+\varepsilon}}{N}+e^{-c'R^s}}^2\,\int_0^t e^{-2\kappa s}\,d s}\,e^{2\kappa t}\le\, \CC^2\rr{\f{R^{\f32+\varepsilon}}{N}+e^{-c'R^s}}^2\,e^{2\kappa t},
\ea
\eeq
with $\CC^2=c^2+\f{C_4}{2\kappa}$. This estimate demonstrates that by choosing $R$ and $N$ sufficiently large such that
\beq
\CC \rr{\f{R^{\f32+\varepsilon}}{N}+e^{-c'R^s}}\exp(\kappa T)\le \f12\min\rr{\f{K}{8C_0},1}.
\eeq
A sufficient condition to guarantee this upper bound is to impose $R>\RR(T)$ and $N>\NN(R,T)$, where
\beq
\RR(t)\,:=\,\rr{\f{\kappa}{c'}\,t+\f{1}{c'}\,\log\f{4\CC }{\min\{K/(8C_0),1\}}}^{1/s},\quad
\NN(R,t):=\f{4\CC \,R^{\f32+\varepsilon}\,\exp\left(\kappa\,t\right) }{\min\{K/(8C_0),1\}}.
\eeq
With these thresholds, the Gr\"onwall bound above is at most half of the bootstrap threshold; if $T_*<T$, the strict inequality at $T_*$ together with the continuation criterion for the finite-dimensional ODE \eqref{eq:num} extends the bound beyond $T_*$, contradicting the definition of $T_*$. Hence $T_*=T$.
Consequently, assuming the constants are appropriately configured and taking $R>\RR(T)$ and $N>\NN(R,T)$, we secure the uniform bound:
\beq
\|g(t)\, m_k\|_{L^2}\le\CC\rr{\f{R^{\f32+\varepsilon}}{N}+e^{-c'R^s}}\,e^{\kappa t}\le 1,\ \ \forall \, t\in [0,T].
\eeq
Finally, the extension of the error estimate from the torus $\cD_L$ to the whole space $\R^3$ to bound the difference $\|f-f_N\|_{L^2_k(\R^3)}$ follows the exact same line of reasoning as in Part $(i)$. This completes the proof of Part $(ii)$  and thereby concludes the proof of Theorem \ref{thm:mainresult}. \end{proof}

We conclude the section with a direct consequence of the estimates just proved, concerning the conserved quantities. Let $m\in\{0,1,2\}$.
Let $\varphi$ satisfy $|\varphi(v)|\le\wei^m$.
Since the weight index $k=k(\mathsf H)$ is chosen sufficiently large, one has $k>m+\f32$ and hence $\wei^{\,m-k}\in L^2(\R^3)$; the Cauchy--Schwarz inequality therefore yields, for all $t\in[0,T]$,
\beq
\left|\int_{\R^3}\varphi(v)\bigl(f_N-f\bigr)(t,v)\,dv\right|
\le\|\wei^{\,m-k}\|_{L^2(\R^3)}\,\bigl\|(f_N-f)(t)\bigr\|_{L^2_k(\R^3)}
\le C_{k,m}\,\CC\left(\f{R^{3/2+\varepsilon}}{N}+R^{-l}\right)e^{\kappa t},
\eeq
the remainder $R^{-l}$ being replaced by $e^{-c'R^s}$ under the hypotheses of Part $(ii)$. Hence the errors in the conserved quantities of the numerical solution --- the mass (zeroth moment, $\int f_N\,dv$), the momentum (first moment, $\int v\,f_N\,dv$), and the energy (second moment, $\int|v|^2f_N\,dv$) --- are all controlled at the same rate as the weighted $L^2$ error of Theorem \ref{thm:mainresult}, although the scheme preserves none of them exactly. We note that the projections and truncations destroy positivity, so that $f_N$ is in general sign-changing; this entails no ambiguity in the mass: since $f\ge0$, one has pointwise $f_N^{-}\le|f-f_N|$, whence $\|f_N^{-}\|_{L^1}\le\|f_N-f\|_{L^1}$, and likewise $\bigl|\int|f_N|\,dv-\int f\,dv\bigr|\le\|f_N-f\|_{L^1}$.

\section{Numerical simulations}
\label{sec:numerical}
\setcounter{equation}{0}
\setcounter{figure}{0}
\setcounter{table}{0}

In this section, we perform numerical simulations using the Fourier method. We apply a fourth-order Runge-Kutta scheme for time discretization, with $R=\frac{2}{3+\sqrt{2}}L$ in the cube domain $\cD_L=[-L,L]^3$. Both the velocity truncation $\psi^R$ and the mode filter $\Phi$ used in the simulations are explicit flat $C^\infty$ cutoffs satisfying exactly the assumptions of the analysis. They are built from the standard smooth step
$$
\theta(x):=\frac{\eta(x)}{\eta(x)+\eta(1-x)},\qquad \eta(x):=\begin{cases}e^{-1/x},&x>0,\\0,&x\le 0,\end{cases}
$$
which is $C^\infty$ on $\R$, identically equal to $0$ on $(-\infty,0]$ and to $1$ on $[1,+\infty)$, with all derivatives vanishing at the gluing points. The velocity truncation is the radial flat bump
$$
\psi^R(v)=\theta\Big(10\big(1-\f{|v|}{R}\big)\Big),
$$
which equals $1$ on $B(0,0.9R)$ and vanishes for $|v|\ge R$; it is the dilation of the fixed bump $\psi=\theta(10(1-|\cdot|))$ of Section 1, so that $\psi^R\in C^\infty_0(B(0,R))$ and $\mathbf{1}_{\{|v|\le R/2\}}\le\psi^R\le\mathbf{1}_{\{|v|\le R\}}$. The mode filter is the tensor-product squared flat step: with $\sigma(t):=\theta(10(1-|t|))$,
$$
\Phi(x)=\sigma(x_1)^2\,\sigma(x_2)^2\,\sigma(x_3)^2,\qquad \sqrt{\Phi(x)}=\sigma(x_1)\,\sigma(x_2)\,\sigma(x_3),
$$
so that $\Phi$ is even, equals $1$ on $[-0.9,0.9]^3$, is supported in $[-1,1]^3$, and $\sqrt{\Phi}\in C^{\infty}(\R^3)$ holds by construction --- the half-projection $\tpn^{1/2}$ is implemented exactly as the Fourier multiplier of symbol $\sigma(k_1/N)\sigma(k_2/N)\sigma(k_3/N)$, with no numerical square root involved. The flatness of $\theta$ makes all derivatives continuous across the gluing spheres and planes; both cutoffs are evaluated using exponentials and arithmetic operations only, which is numerically stable, and all the simulations below are run with these cutoffs.

Before describing the test cases, we comment on the computational cost of the scheme \eqref{eq:num}--\eqref{defQNR}. Per right-hand-side evaluation, the inputs $G_N=\tpn(g\psi^R)$ and $H_N=\tpn(h\psi^R)$ are obtained by a pointwise multiplication by $\psi^R$ followed by one application of the diagonal Fourier filter $\tpn$, and the outer operators $\tpn^{\f12}$ and the second multiplication by $\psi^R$ are of the same kind; all of these contribute only $O(N^d\log N)$ operations. The dominant cost is the evaluation of the periodic collision operator $Q_p(G_N,H_N)$: by the convolution identity $Q_p^+(e^{i\f{\pi}{L}l\cdot v},e^{i\f{\pi}{L}m\cdot v})=\beta(l,m)\,e^{i\f{\pi}{L}(l+m)\cdot v}$ behind \eqref{eq:defbeta}, its modes are given by a discrete convolution with the precomputed kernel weights $\beta(l,m)$, whose direct evaluation costs $O(N^{2d})$ operations and $O(N^{2d})$ storage --- exactly as in the classical spectral method of \cite{PP96, PR00}. The overall cost per Runge--Kutta step is therefore $O(N^{2d})$ up to constant factors (four right-hand-side evaluations), the smooth truncations and filters of the present scheme adding only lower-order corrections.

\begin{rmk}\label{rmk:cost}\textnormal{We deliberately refrain from the kernel approximations that underlie the fast variants of the method: the algorithms of \cite{MP, FMP} (Carleman-type representation, $O(M N^d\log N)$, restricted to special kernels) and of \cite{GHHH, HQ} (low-rank approximations of the kernel weights, $O(M N^{d+1}\log N)$) replace the exact convolution structure of $Q_p$ by an $M$-term approximation, an additional perturbation that is not covered by the analysis of this paper, whose energy estimates rely on the exact convolution structure. The direct evaluation is therefore the faithful implementation of the theorem just proved. We note, however, that these accelerations can be combined with the present scheme: the operators $\tpn,\tpn^{\f12}$ act diagonally on the Fourier side and the multiplications by $\psi^R$ act pointwise in the physical space, leaving the convolution --- the only object to which the approximation is applied --- untouched, so that a large-scale implementation may evaluate $Q_p$ by any of the fast variants while retaining the filters and truncations analyzed here.}
\end{rmk}

We consider two cases in our simulations.
First, we address the case of Maxwellian molecules where $\gamma=0$,
for which we compare the approximate solution with an explicit
solution. Second, we address the hard sphere case where $\gamma=1$, and observe the behavior of the approximate
solution through relative entropy, high-order moments, and Fischer
information.
\subsection{Maxwellian molecules $\gamma=0$}
We consider the case of Maxwellian molecules $\gamma=0$ in
\eqref{1}, with the collision kernel as a constant:
\beq
B(|v-v_*|,\cos\theta)=\f{1}{4\pi}.
\eeq
For Maxwell molecules the collision kernel is constant, so this normalization differs from the kernel $B\equiv1$ of the analysis only through the time rescaling $t\mapsto t/(4\pi)$, and none of the convergence conclusions is affected.
This solution, known as the BKW (Bobylev-Krook-Wu)~\cite{Bo,KW} solution for Maxwell molecules, takes the form
\beq
    f(t,v) = \frac{1}{2(2\pi \,K(t))^{3/2}} \exp\left(-\frac{|v|^2}{2\,K(t)}\right) \left( \frac{5\,K(t)-3}{K(t)} + \frac{1-K(t)}{K(t)^2}\, |v|^2 \right),
\eeq
where we take $K(t) = 1 - 0.4 \exp(-t/6)$ to ensure that the solution is positive from the initial time $t = 0$. The initial data for this case is given by
\beq
f(0,v)=f_0(v) \,=\;  \frac{1}{2(2\pi\; K_0)^{3/2}}\; \frac{1-K_0}{K_0^2}\,|v|^2 \;\exp\left(-\frac{|v|^2}{2\,K_0}\right),
\eeq
with $K_0=0.6$, so that the constant term $5K_0-3$ in the BKW profile vanishes at $t=0$. We compute the $L^2$ error as
\beq
e(t)=\int_{\mathcal{D}_L} |f_{N}(t,v) - f(t,v)|^2 dv,
\eeq
where $f_N$ is the numerical solution obtained from \eqref{eq:num} at time $t$, and $f$ is the exact BKW solution. 

We conducted a series of numerical calculations by varying the two parameters, the length of the domain $2L$ and the number of modes per direction $2N$. We conduct the numerical simulations in time period $[0,1]$ with time step $\Delta t=0.01$.

First, we present the results when the
number of modes $2N=32,48,56,64$, to illustrate the evolution of the error (on
a logarithmic scale) as a function of $L$ in Figure \ref{fig:1}. 
Using first the weaker error bound of Theorem \ref{thm:mainresult}(i), we describe two competing terms: the mode-truncation (Fourier projection) error $R^{3/2+\varepsilon}/N$ and the truncation error $R^{-l}$, where $R = \frac{2}{3+\sqrt{2}}L$. The mode-truncation error increases with $R$ (and thus with $L$) due to the polynomial factor $R^{3/2+\varepsilon}$, while the truncation error decreases as $R^{-l}$ for sufficiently large $l$.

The numerical results in Figure \ref{fig:1} clearly demonstrate this trade-off. For the smallest number of modes $2N=32$, the mode-truncation error term dominates, and consequently the overall error increases monotonically with $L$.
At the final time, for the largest numbers of modes, $2N=56$ and $2N=64$, the
velocity-truncation contribution is leading at small $L$ and the error
decreases with $L$ up to $L\approx12$, beyond which the mode-truncation error
dominates and the error increases again; the intermediate case $2N=48$
behaves similarly, with its minimum at $L\approx11$.
This behavior is entirely consistent with the theoretical prediction that the optimal choice of $R$ (and hence $L$) depends on $N$, and that there exists a critical value where the two error terms are of comparable magnitude.
We note that this test falls into the exponential-tail regime of Theorem
\ref{thm:mainresult}(ii): the BKW profile decays like $e^{-|v|^2/(2K(t))}$, so
Part (ii) applies with $s=2$ (Maxwell molecules being admissible after the
removal of the restriction $\gamma>0$), and the truncation contribution is
$e^{-c'R^2}$ rather than $R^{-l}$. In $N$, the observed decay is much faster
than the guaranteed worst-case rate $R^{3/2+\varepsilon}/N$. This rapid decay
is consistent with spectral convergence for the analytic BKW profile --- a
polynomial times a Gaussian; the algebraic
rate is what the present theory guarantees uniformly over the assumed class,
not the observed order, as recorded in the remark after Theorem
\ref{thm:mainresult}. The observed trade-off, and the minimizing radii among
the scanned values, are displayed in Figure
\ref{fig:scan}.

\begin{figure}[htbp]
	\centering
\includegraphics[width=\textwidth]{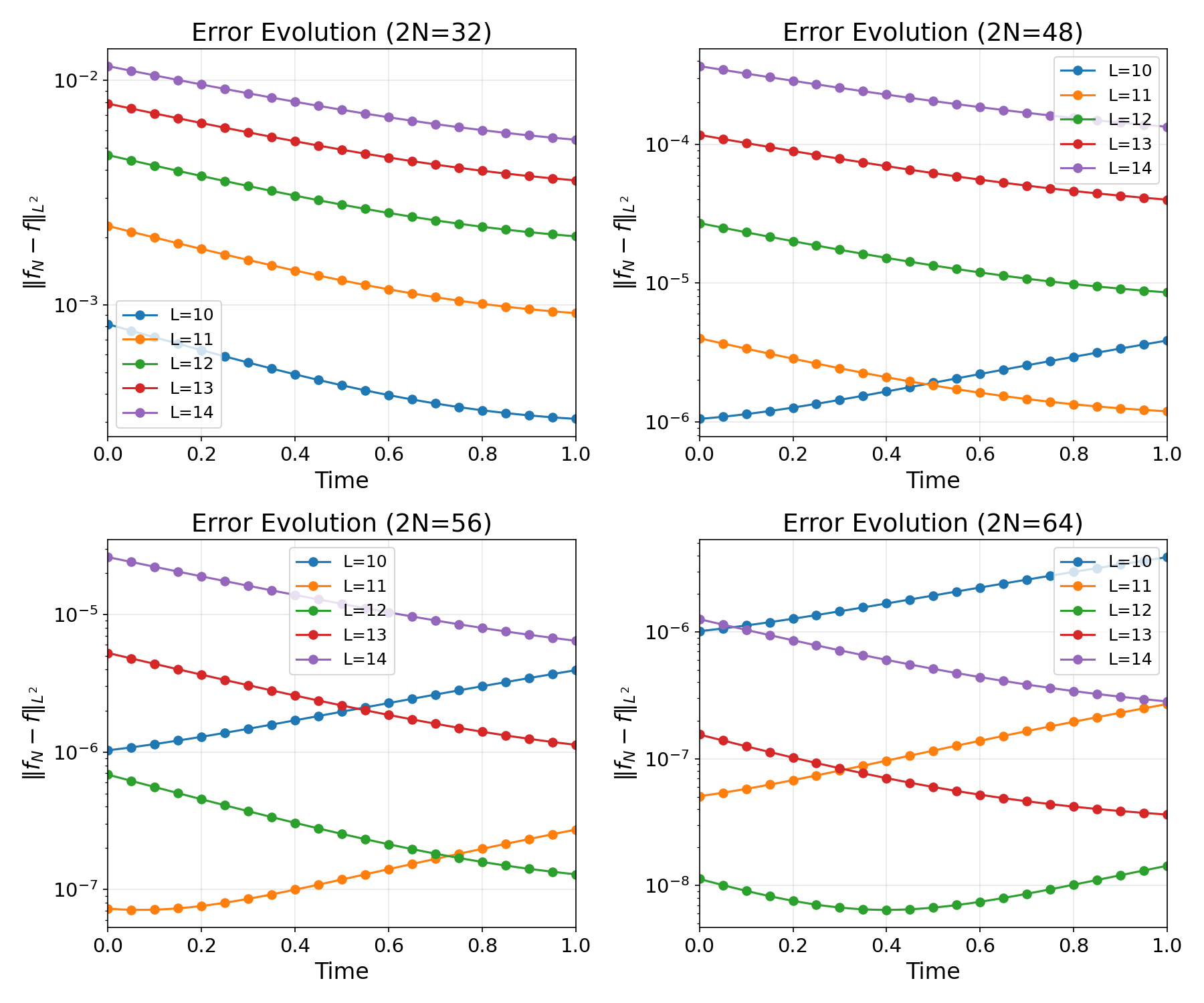}
	\caption{{\bf Maxwellian molecules:} time evolution of the $L^2$ error in log scale for the scheme \eqref{eq:num} with respect to $L$ for $2N=32$, $2N=48$, $2N=56$ and $2N=64$.}
	\label{fig:1}
\end{figure}

In the second part, we consider a sufficiently large domain $\cD_L$
and focus on the error with respect to the number of Fourier modes per
direction $2N$. We then plot the evolution
of the error (still on a logarithmic scale) for $L=10,11,12,13,14$ in Figure \ref{fig:2}.

Across the scanned range, the error decreases steeply and monotonically
in $N$ for $L=12$, $13$ and $14$, with no saturation within $2N\le64$: the
observed decay is much faster than the guaranteed algebraic rate in $N$ and
is consistent with spectral convergence for this analytic datum. At the two
smallest scanned domains, $L=10$ and $L=11$, the curves approach saturation
levels. The corresponding two norm floors are fitted by a single exponential
in $R^2$ in Figure \ref{fig:scan}, consistently with the velocity-truncation
mechanism of Theorem \ref{thm:mainresult}(ii); these observed floors are
empirical saturation levels, not lower bounds supplied by the theorem.

In contrast, for larger domain sizes $L=13$ and $L=14$, the truncation contribution $e^{-c'R^2}$ becomes smaller, and the mode-truncation error $R^{3/2+\varepsilon}/N$ emerges as the leading term. Consequently, the error continues to decrease as $N$ increases across the range of $N$ values considered, consistently with the mode-truncation term of the theoretical estimate. This transition from velocity-truncation-dominated to mode-truncation-dominated regimes is consistent with the two contributions in Theorem \ref{thm:mainresult}.

\begin{figure}[htbp]
	\centering
\includegraphics[width=\textwidth]{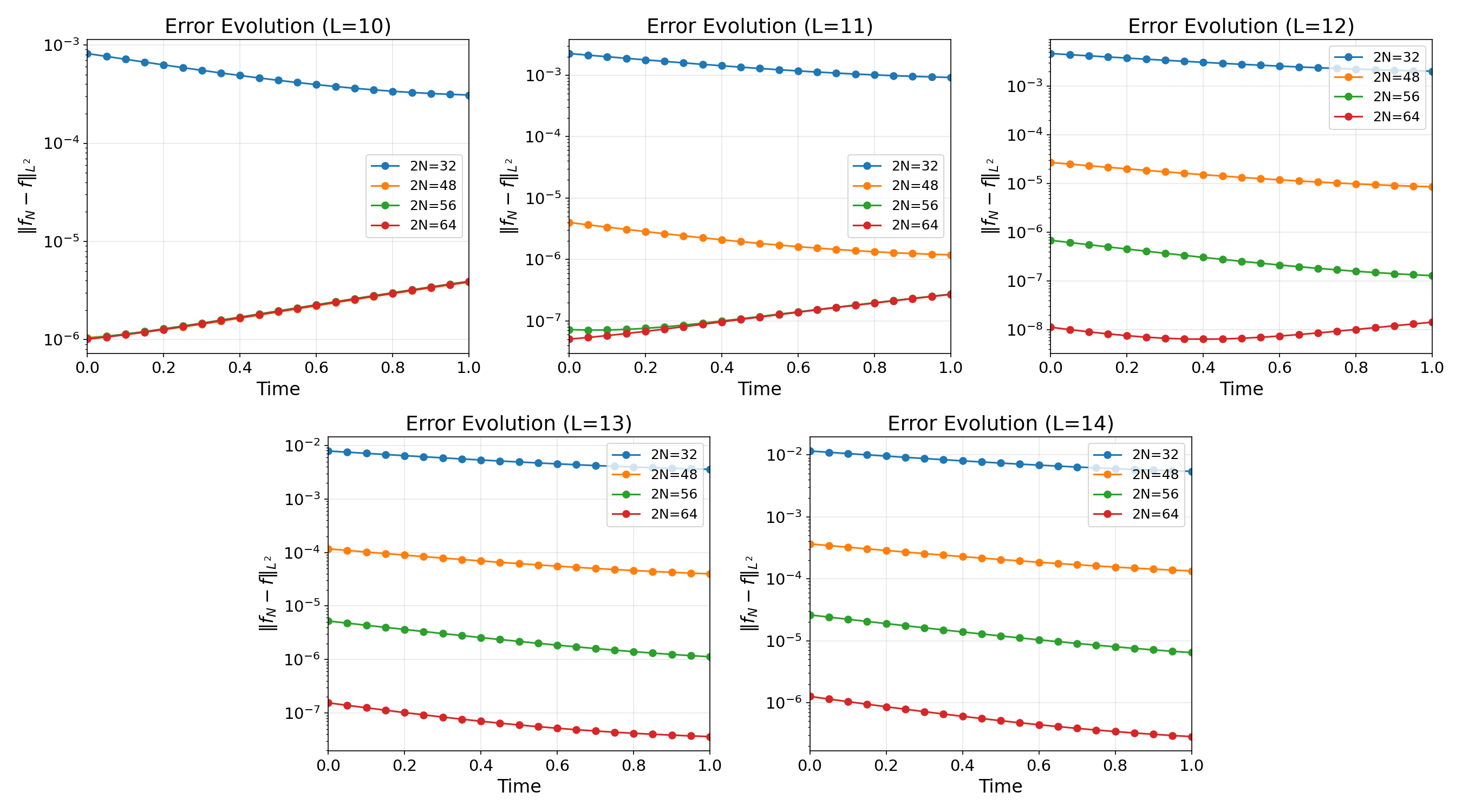}
	\caption{{\bf Maxwellian molecules:} time evolution of the $L^2$
	  error in log scale for
	  the scheme \eqref{eq:num} with respect to $N$ for $L=10$, $L=11$, $L=12$, $L=13$ and $L=14$.}
	\label{fig:2}
\end{figure}
We further compare the whole $(R,N)$ scan with the theoretical form
$R^{3/2+\varepsilon}/N+e^{-c'R^2}$, where $R=\frac{2}{3+\sqrt2}L$.
Figure \ref{fig:scan} displays the final-time norm errors versus $N$ at fixed
$L$ in panel (a), and versus $L$ at fixed $N$ in panel (b). The comparison
curves use the limiting exponent $\varepsilon=0$ for illustration, a fitted
exponential coefficient $c_n$, and a separate normalization at the worst
sampled point of each curve; they do not recover the implicit constants or
the admissible $c'$ of Theorem \ref{thm:mainresult}(ii). Two features are
visible. First, the two measured norm floors at $L=10$ and $L=11$ are fitted
by a single exponential $B e^{-c_nR^2}$, consistently with the
domain-truncation mechanism. Second, away from saturation the equivalent
algebraic orders are roughly $10$--$25$ and vary across the scan, without
increasing at every step. This rapid decay is consistent with spectral
convergence for the analytic BKW profile; the theory guarantees the
worst-case algebraic rate and does not assert the higher observed order.
The calibrated comparison curves lie above their sampled data by
construction, so this comparison illustrates consistency with the theoretical
form rather than independently validating its constants. In panel (b), the
curves for $2N=48,56,64$ are U-shaped, with circled minima identifying the
best radii among those scanned. For $2N=32$, the smallest scanned error is
at the boundary $L=10$. These empirical minima do not quantitatively verify
the minimizers of the displayed comparison bounds, which lie outside the
scanned interval.

\begin{figure}[htbp]
	\centering
	\includegraphics[width=\textwidth]{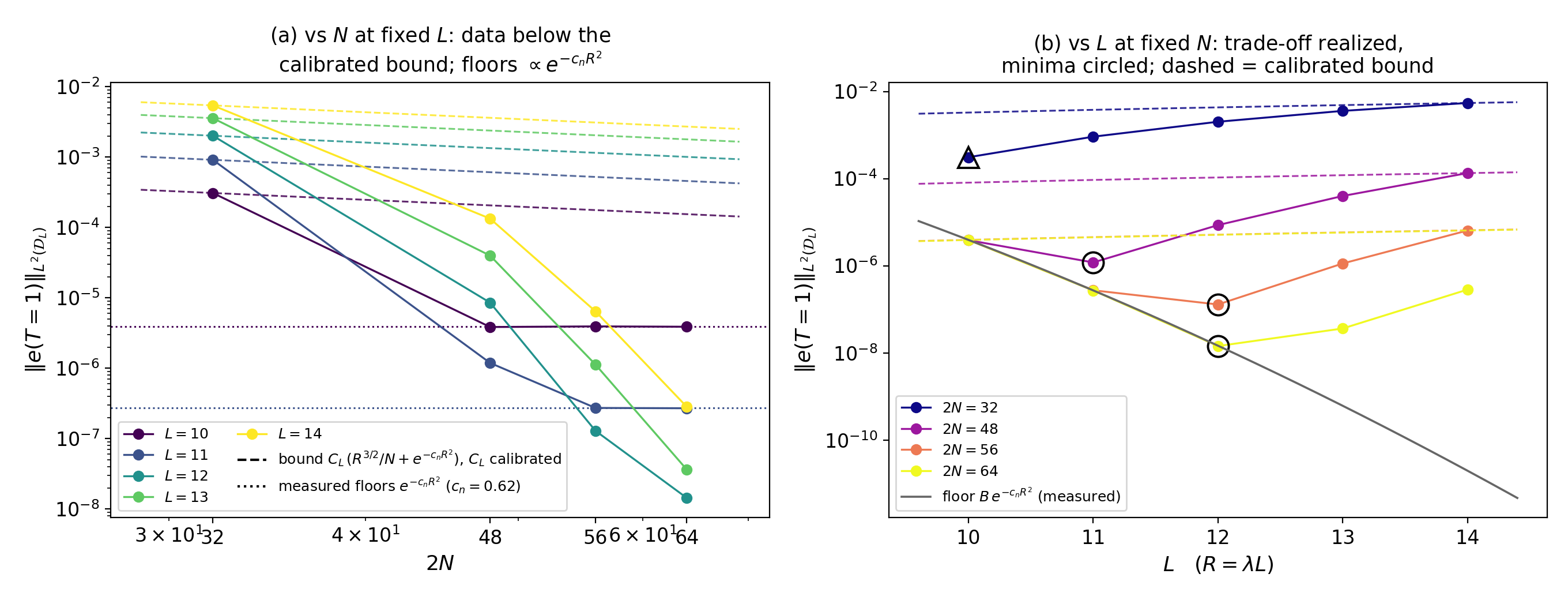}
	\caption{{\bf Maxwellian molecules, the $(R,N)$ scan and the form of Theorem
	\ref{thm:mainresult}(ii):} $\|f_N(T)-f(T)\|_{L^2(\mathcal{D}_L)}$ over the
	scanned $(L,2N)$ pairs for the BKW datum ($T=1$; the norm is the $L^2$ error used in Figures \ref{fig:1}--\ref{fig:2}).
	(a) Versus $2N$ at fixed $L$ (log-log): dotted lines show the two measured
	saturation floors; the nonsaturated decay has equivalent orders roughly
	$10$--$25$, consistently with spectral convergence. Dashed curves are
	$C_L(R^{3/2}/N+e^{-c_nR^2})$, with the illustrative choice
	$\varepsilon=0$ and each $C_L$ calibrated at the worst sampled point.
	(b) Versus $L$ at fixed $N$ (semilog): circles mark the minima over the
	scanned radii for $2N=48,56,64$; the triangle marks the smallest scanned
	error for $2N=32$, attained at $L=10$. Dashed curves use the same
	comparison form, calibrated separately for each fixed $N$; grey shows
	the fit $B e^{-c_nR^2}$ through the two measured floors, extended across
	the displayed interval.}
	\label{fig:scan}
\end{figure}

\subsection{Hard sphere molecules $\gamma=1$}This test is used to compute the time evolution of the numerical solution \eqref{eq:num} for the hard sphere collision kernel with a sum of two Gaussians as the initial condition. This test case has been widely studied in the literature, including in \cite{PR00} and \cite{GHHH}, and serves as a benchmark for validating numerical schemes for the Boltzmann equation. We choose the collision kernel as
\beq
B(|v-v_*|,\cos\theta)=\frac{1}{4\pi}|v-v_*|,
\eeq
and we choose the initial condition as
$$
f(0, v) = \frac{1}{2(2\pi)^{3/2}} \left[ \exp\left(-\frac{|v - u_1|^2}{2}\right) + \exp\left(-\frac{|v - u_2|^2}{2}\right) \right]
$$
with
$u_1 = (0, 0, 2), \ u_2 = (0, 0, -2)$
. The integration time is $T
 =6$ with time step $\Delta t = 0.1$ in the computational domain $\cD_L=[-L,L]^3$ with the number of Fourier modes per direction $2N=64$.

First, we visualize the cross-section of the distribution function at different times. Figure \ref{fig:5} displays $f(0,0,v_3)$ at $t=0, 0.5, 1, 1.5, 3, 6$ with $L=16$, illustrating the gradual relaxation from the initial sum of two Gaussians to the equilibrium Maxwellian profile.
 \begin{figure}[htbp]
	 \centering
	 \includegraphics[width=\textwidth]{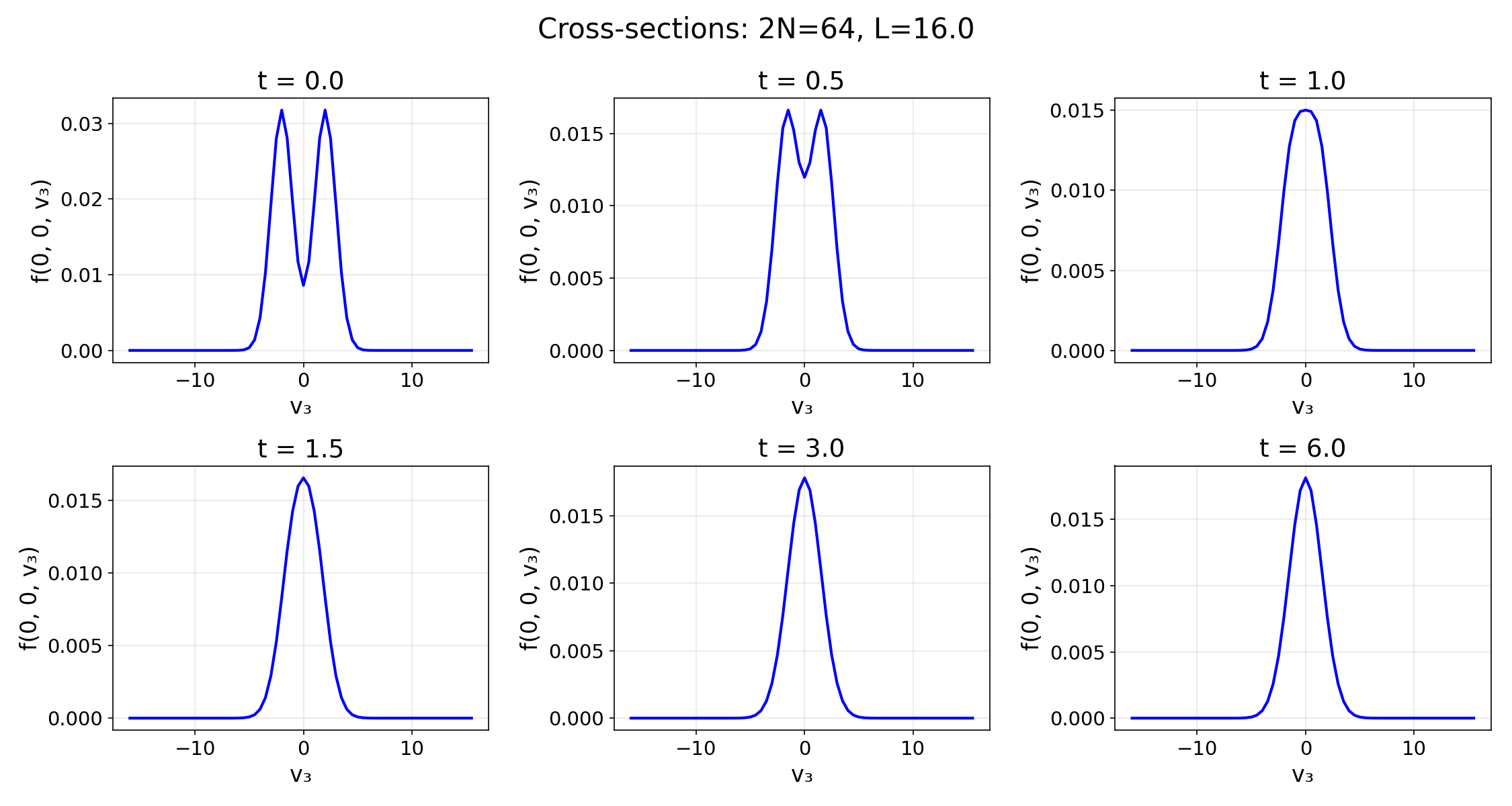}
	 \caption{{\bf Hard sphere molecules:} Cross-section $f(0,0,v_3)$ at different times $t=0, 0.5, 1, 1.5, 3, 6$.}
	 \label{fig:5}
 \end{figure}

Next, we investigate the conservation properties of the equation. It is important to note that the scheme \eqref{eq:num} does not exactly preserve the mass, momentum, and energy of the Boltzmann equation. To quantify this, we perform numerical experiments with $2N=64$ and varying computational domain sizes $L=16, 17, 18$. Figure \ref{fig:mass_energy} shows the time evolution of the mass $\rho=\int_{\R^3}f(v)\,dv$ and energy $E=\int_{\R^3}|v|^2\,f(v)\,dv$ for these three cases. Both quantities exhibit a gradual decrease over time, with the decay rate inversely proportional to the domain size $L$. Specifically, larger values of $L$ result in slower decay, indicating that the conservation properties improve as the computational domain expands.

\begin{figure}[htbp]
	\centering
	\includegraphics[width=0.48\textwidth]{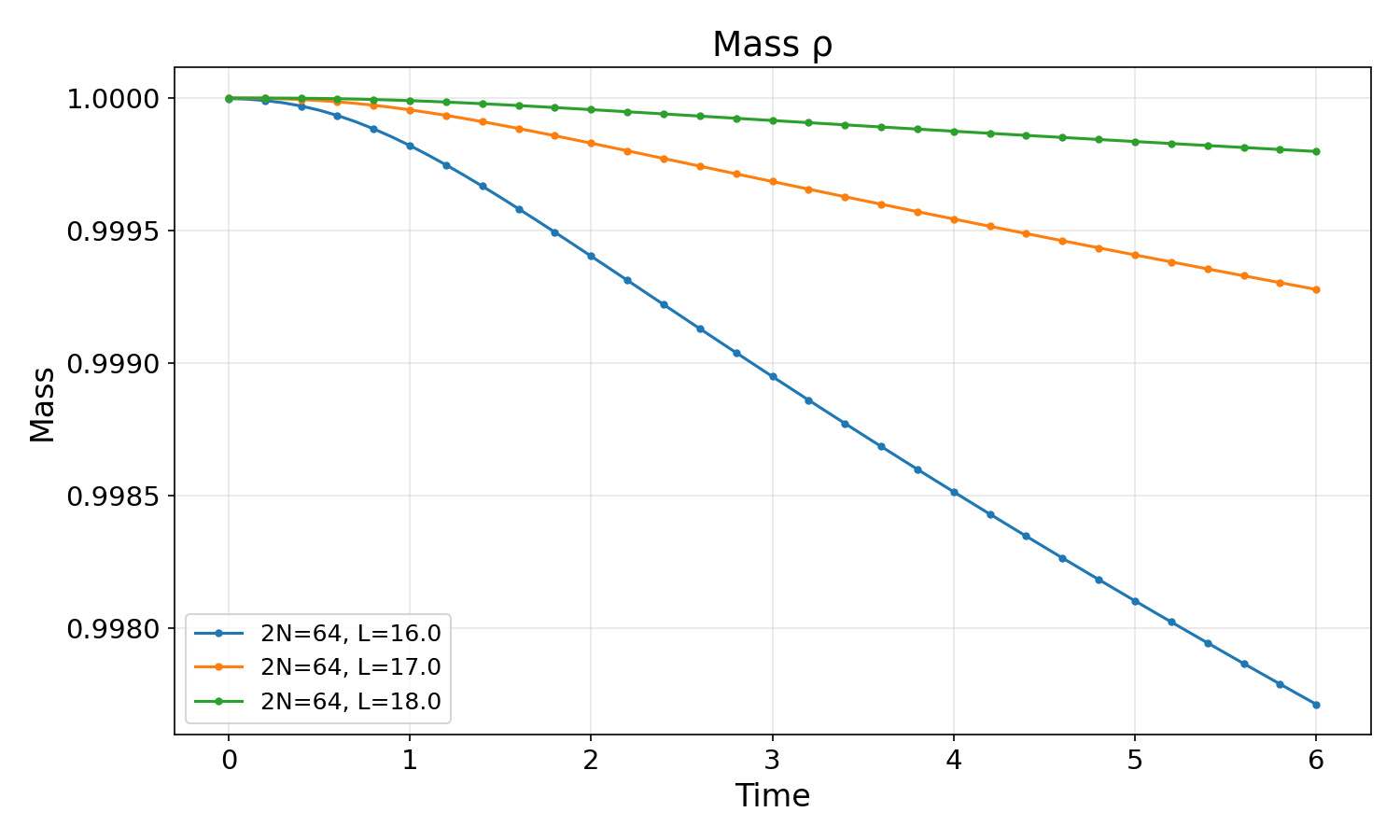}
	\includegraphics[width=0.48\textwidth]{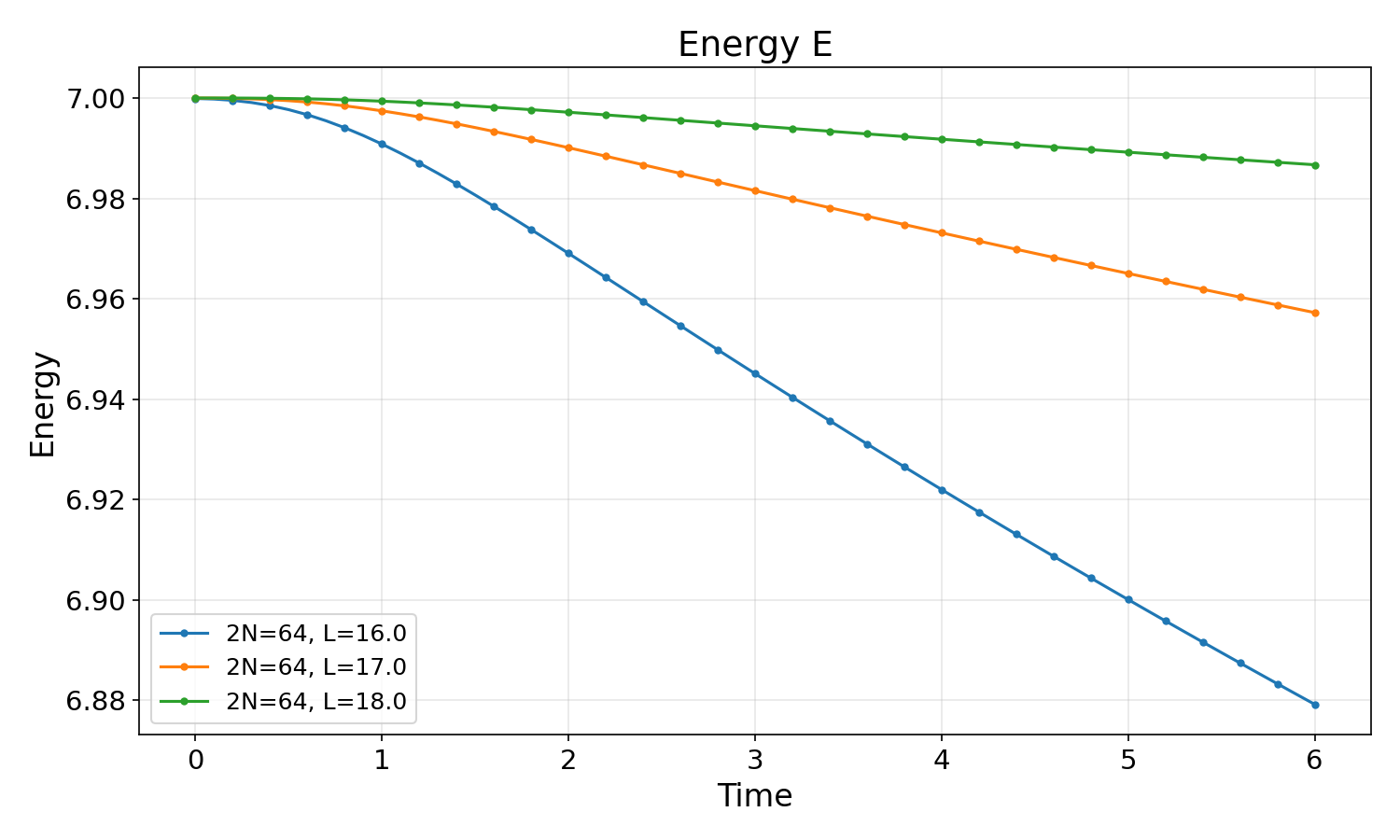}
	\caption{{\bf Hard sphere molecules:} Time evolution of the mass (left) and energy (right) for $2N=64$ with different domain sizes $L=16, 17, 18$.}
	\label{fig:mass_energy}
\end{figure}

We further analyze the numerical accuracy by examining the evolution of several important quantities. According to the Boltzmann H-theorem, for the exact solution of the Boltzmann equation, the entropy $H(f) = \int_{\mathbb{R}^3} f \log f \, dv$ and the relative entropy $H(f|\mu)$ defined in \eqref{DefHt} should be non-increasing functions of time due to the $H$-Theorem, and the Fisher information $I(f) = \int_{\mathbb{R}^3} \frac{|\nabla_v f|^2}{f} \, dv$ should be non-increasing functions of time from the recent result \cite{ISV}. These quantities, together with the $L^2$ distance $\|f-\mu\|_{L^2}$, should eventually approach constant values as the solution converges to the steady-state Maxwellian distribution $\mu$. Figure \ref{fig:entropy_quantities} presents the time evolution of these four quantities for $2N=64$ with $L=16, 17, 18$.

In the numerical postprocessing we use the floor $\max\{f_N,10^{-30}\}$.
In evaluating entropy, the $f_N\log f_N$ contribution is set to zero at grid points where $f_N<10^{-30}$.

\begin{figure}[htbp]
	\centering
	\includegraphics[width=0.48\textwidth]{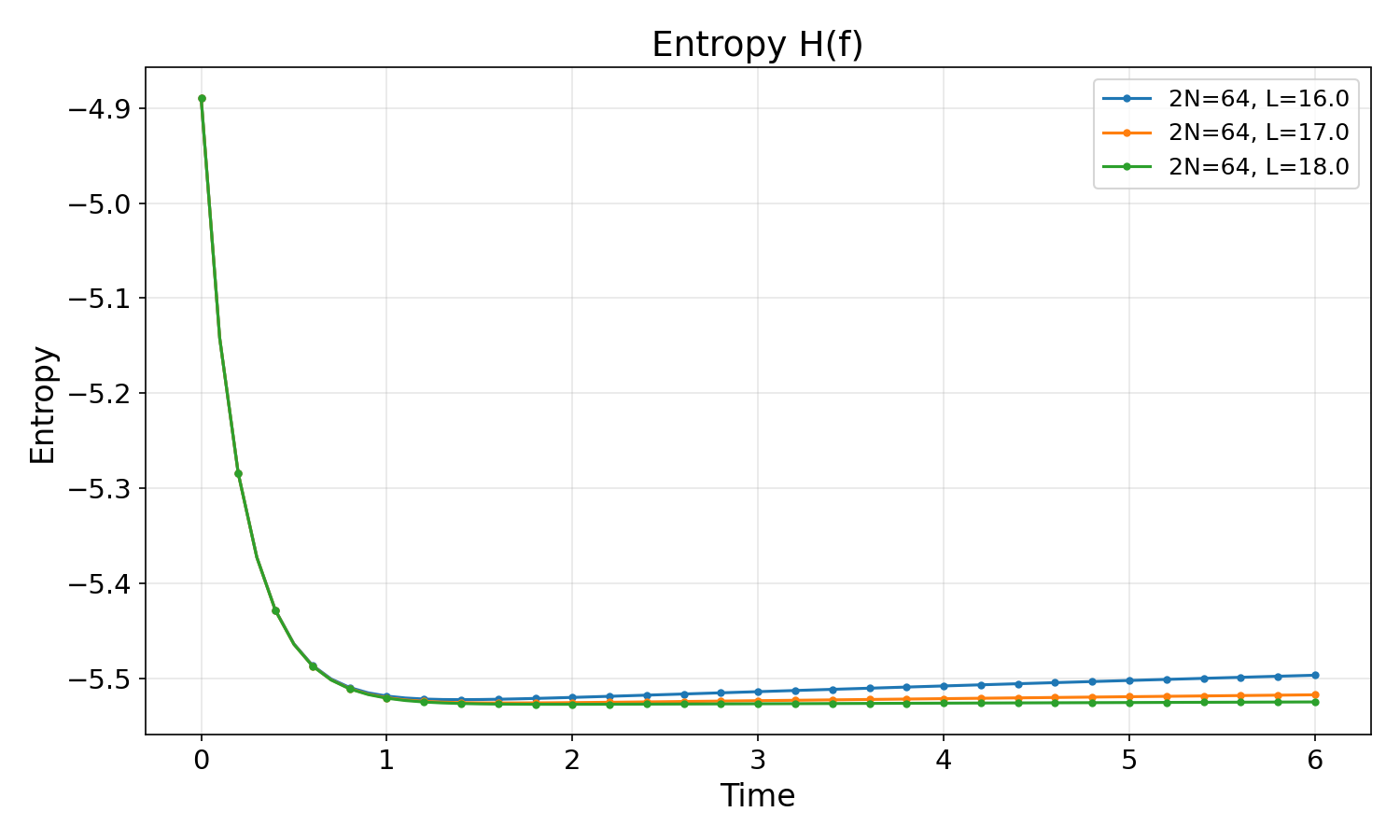}
	\includegraphics[width=0.48\textwidth]{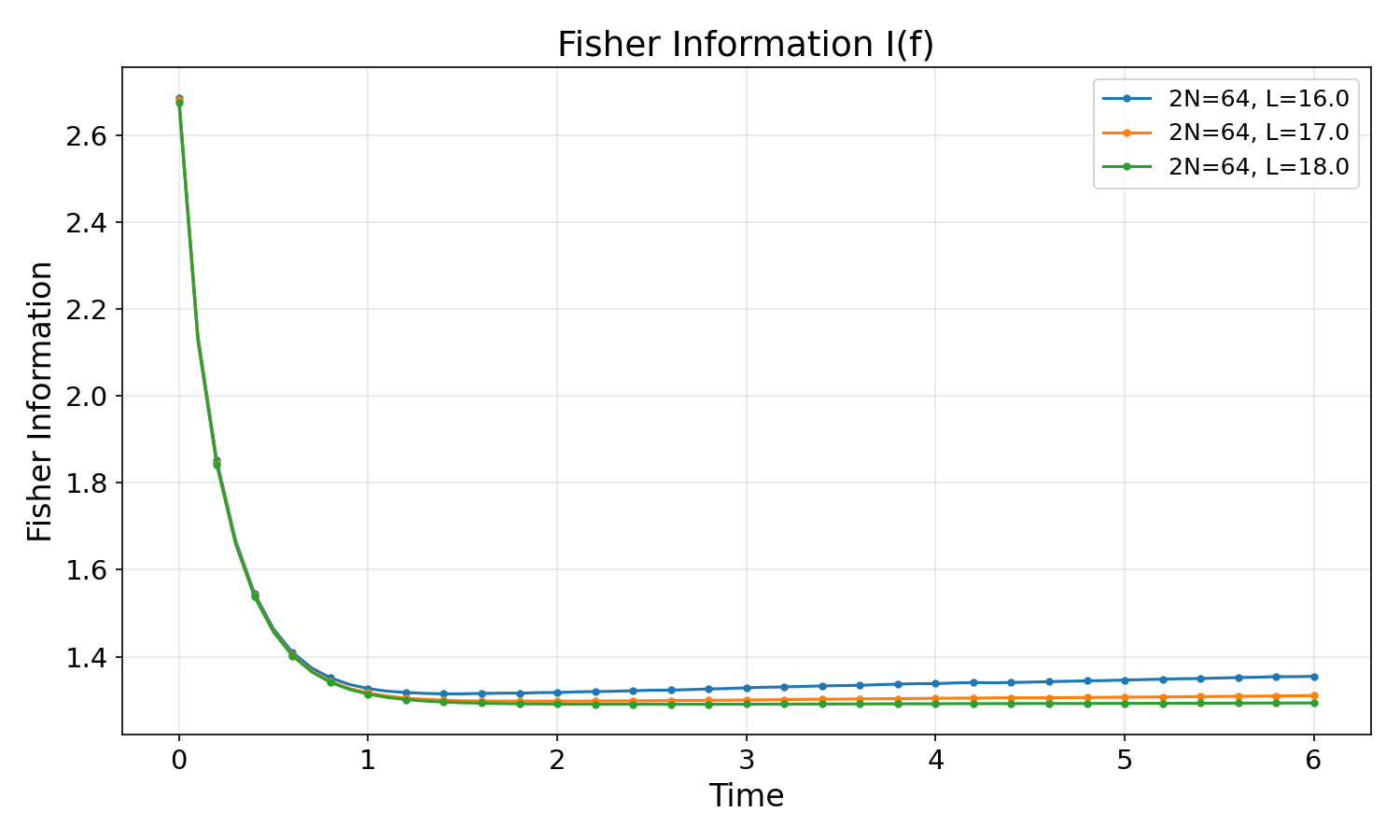}
	\includegraphics[width=0.48\textwidth]{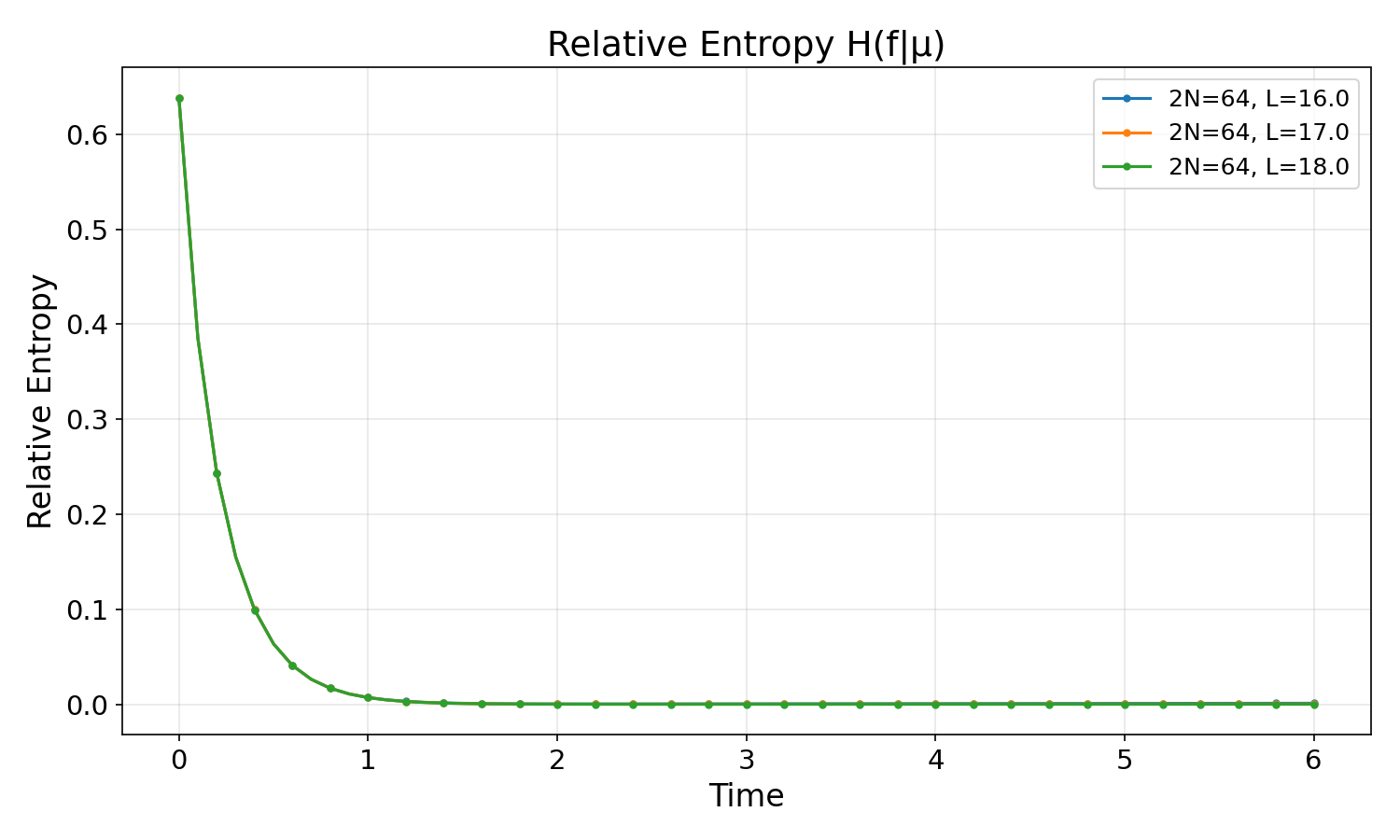}
	\includegraphics[width=0.48\textwidth]{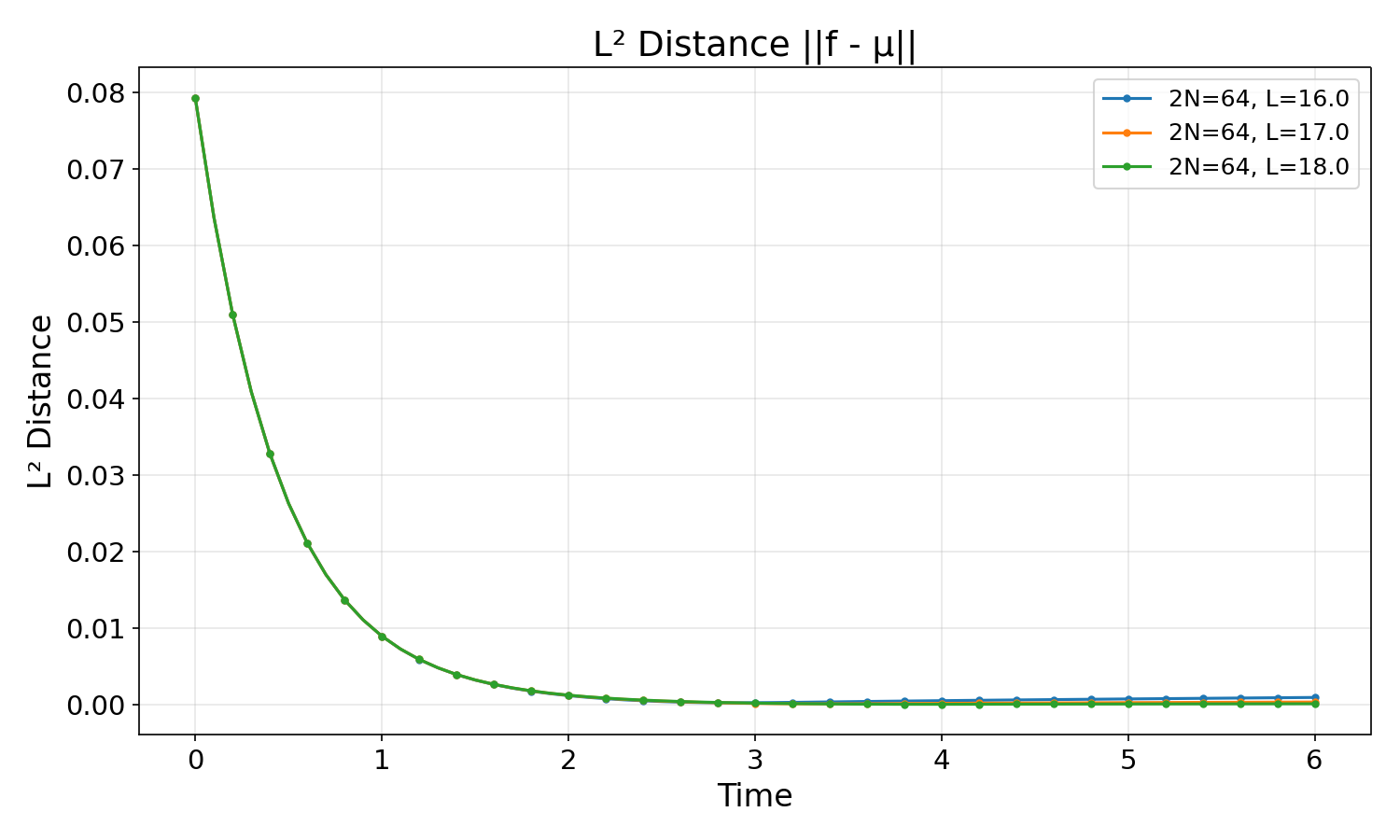}
	\caption{{\bf Hard sphere molecules:} Time evolution of entropy $H(f)$ (top left), Fisher information $I(f)$ (top right), relative entropy $H(f|\mu)$ (bottom left), and relative $L^2$ distance $\|f-\mu\|_{L^2}$ (bottom right) for $2N=64$ with different domain sizes $L=16, 17, 18$.}
	\label{fig:entropy_quantities}
\end{figure}

The numerical results reveal a subtle deviation from the theoretical predictions. After approximately $t=1.5$, the numerical solution appears to have approached the steady state. However, instead of remaining constant as expected for the exact solution, these quantities exhibit a very slight increase over time. This non-monotonic behavior is more pronounced for smaller values of $L$, while it becomes nearly imperceptible for larger $L$. In particular, the relative entropy curves for $L=16, 17, 18$ almost completely overlap, making the upward trend virtually undetectable.

This phenomenon can be explained by the fact that our numerical scheme does not preserve the steady state. As demonstrated in Figure \ref{fig:mass_energy}, the mass and energy gradually decrease over time due to the truncation of the velocity domain. Since the steady-state Maxwellian $\mu$ is determined by the conserved quantities (mass, momentum, and energy), the decrease in mass and energy causes the numerical solution to shift away from $\mu$ continuously, resulting in the observed slight increase in the quantities. Larger values of $L$ lead to slower decay in mass and energy, which in turn causes the steady state to change more slowly, thereby reducing the magnitude of the observed increase. This analysis highlights the importance of using sufficiently large computational domains and number of Fourier modes to obtain accurate long-time behavior.

Since no explicit solution is available for the hard-sphere kernel, we
complement the preceding diagnostics with a direct error measurement against
a high-resolution reference solution: for each domain size $L=16,17,18$, the
reference $f_{\rm ref}$ is the solution computed by the same scheme at the
finest resolution $2N=64$ at the same $L$ and time step, and the errors
$e(t)=\|f_N(t)-f_{\rm ref}(t)\|_{L^2(\mathcal{D}_L)}$ are evaluated for
$2N=32,40,48,56$. Both $f_N$ and $f_{\rm ref}$ are trigonometric polynomials,
so $e(t)$ is evaluated exactly through Parseval's identity from their Fourier
coefficients, without any grid interpolation --- which matters here, since
the grids of the coarse and reference solutions differ. Figure
\ref{fig:ref} shows the time evolution of $e(t)$ on a logarithmic scale, one
panel per domain size: at the final time, for all three domain sizes, the error
decreases monotonically as $N$ increases, exhibiting the mode-truncation
convergence of the scheme in the hard-sphere case and complementing the
convergence study of Section 4.1. We also record how the collision term is evaluated:
the kernel-weighted convolution uses the ordinary integer-index condition $l+m=k$
in \eqref{eq:Qffkcon}, with no reduction modulo $2N$.

\begin{figure}[htbp]
	\centering
	\includegraphics[width=\textwidth]{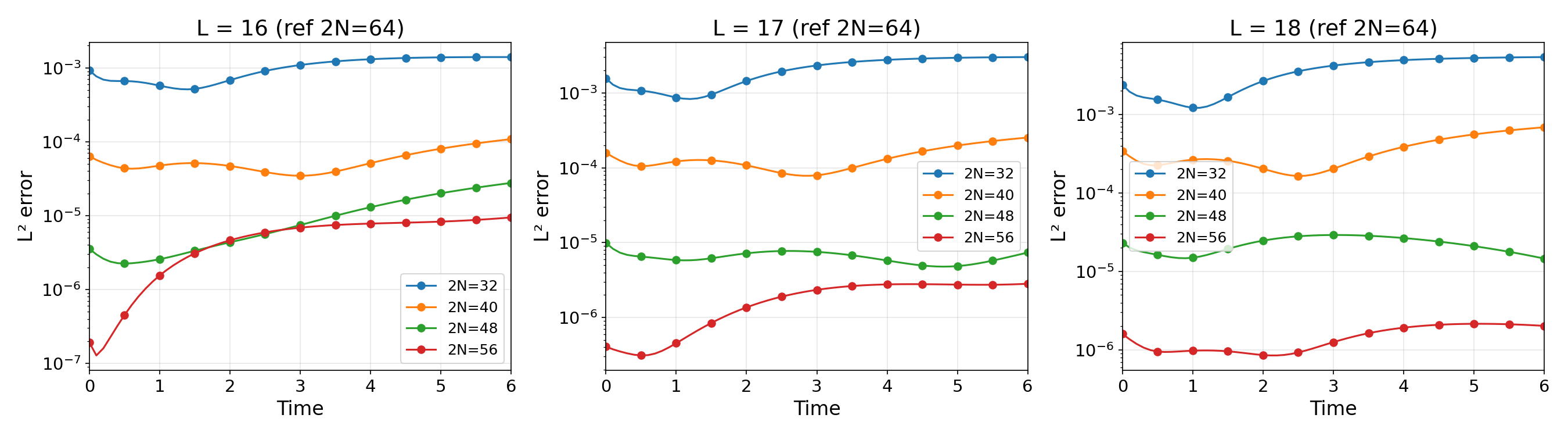}
	\caption{{\bf Hard sphere molecules:} time evolution of the $L^2$ error
	with respect to the reference solution ($2N=64$, same $L$ and time step),
	in log scale, for $2N=32$, $40$, $48$ and $56$; one panel per domain size
	$L=16$ (left), $L=17$ (middle), $L=18$ (right).}
	\label{fig:ref}
\end{figure}

\section*{Acknowledgement}
The research of L.-B. He was supported by NSF of China under Grant No.11771236 and New Cornerstone Investigator Program 100001127. The research is supported by the High Performance Computing Center, Tsinghua University. 
L. Liu acknowledges the support by National Key R\&D Program of China (2021YFA1001200), Ministry of Science and Technology in China, and General Research Fund (14303022 \& 14301423 \& 14307125)  funded by Research Grants Council of Hong Kong.
We thank the student Zhen Hao from Wuhan University for helpful discussions during his visiting at The Chinese University of Hong Kong.  
\appendix
\section{Technical Proofs}\label{sec:appendix}
\subsection{Proof of Theorem \ref{thm:Qfgh} and Theorem \ref{thm:Qfgg}}\label{sec:Boltzmann} We first introduce some lemmas and notation:

\begin{lem}\label{thm:ch}
	(\cite{CHJ} Lemma 4.1) If $b \in L^1([0,1])$ (the angular cutoff condition; the framework of \cite{CHJ} assumes, in addition, a nonnegative angular kernel supported on a hemisphere, together with $-3<\gamma\le 1$, which covers our setting $0\le\gamma\le 1$; the collision operator of the present paper is integrated over the full sphere, and the changes of variables below are used with the angular weights displayed explicitly), then we have

	(1) (Regular change of variables)
	$$
	\int_{\mathbb{R}^3} \int_{\mathbb{S}^2} b(\cos \theta)\left|v-v_*\right|^\gamma f(v^{\prime})\, d \sigma \, d v=\int_{\mathbb{R}^3} \int_{\mathbb{S}^2} b(\cos \theta) \frac{1}{\cos ^{3+\gamma}(\theta / 2)}\left|v-v_*\right|^\gamma f(v)\, d \sigma \,d v .
	$$
	
	(2) (Singular change of variables)
	$$
	\int_{\mathbb{R}^3} \int_{\mathbb{S}^2} b(\cos \theta)\left|v-v_*\right|^\gamma f(v^{\prime})\, d \sigma\, d v_*=\int_{\mathbb{R}^3} \int_{\mathbb{S}^2} b(\cos \theta) \frac{1}{\sin ^{3+\gamma}(\theta / 2)}\left|v-v_*\right|^\gamma f(v_*)\, d \sigma \, d v_* .
	$$
\end{lem}
\begin{lem}\label{thm:can}
	(Cancellation Lemma, \cite{CHJ} Lemma 4.2 ) For any smooth function $f$, it holds that $\int_{\mathbb{R}^3 \times \mathbb{S}^2} B(v-v_*, \sigma)(f^{\prime}-f)\, d v \,d \sigma=(f * S)(v_*)$, where
	$$
	S(z)=\left|\mathbb{S}^1\right| \int_0^{\frac{\pi}{2}} \sin \theta\left[\frac{1}{\cos ^3(\theta / 2)} B\left(\frac{|z|}{\cos (\theta / 2)}, \cos \theta\right)-B(|z|, \cos \theta)\right]\, d \theta.
	$$	
\end{lem}
In the quoted identity, the angular kernel is integrated over the half range $0\le\theta\le\pi/2$ under the convention of \cite{CHJ}. In this paper this auxiliary identity is used only with its stated angular support; no full-sphere bilinear collision term is replaced by its restriction to a hemisphere.
\begin{lem}\label{thm:E}
	(\cite{CHJ} Lemma 2.1) We have the following two decompositions about $\left\langle v^{\prime}\right\rangle^2$ :

	- Let $\mathbf{h}:=\frac{v+v_*}{\left|v+v_*\right|}, \mathbf{n}:=\frac{v-v_*}{\left|v-v_*\right|}, \mathbf{j}:=\frac{\mathbf{h}-(\mathbf{h} \cdot \mathbf{n}) \mathbf{n}}{\sqrt{1-(\mathbf{h} \cdot \mathbf{n})^2}}, E(\theta):=\langle v\rangle^2 \cos ^2(\theta / 2)+\left\langle v_*\right\rangle^2 \sin ^2(\theta / 2)$ and $\sigma=\cos \theta \, \mathbf{n}+\sin \theta\, \hat{\omega}$ with $\hat{\omega} \in \mathbb{S}^1(\mathbf{n}) :=\left\{\hat{\omega} \in \mathbb{S}^2 \mid \hat{\omega} \perp \mathbf{n}\right\}$. Then $\mathbf{h} \cdot \sigma=(\mathbf{h} \cdot \mathbf{n}) \cos \theta+\sqrt{1-(\mathbf{h} \cdot \mathbf{n})^2} \sin \theta\, (\mathbf{j} \cdot \hat{\omega})$ and
	$$
	\left\langle v^{\prime}\right\rangle^2=E(\theta)+\sin \theta\, (\mathbf{j} \cdot  \hat{\omega})\, \tilde{h}
	$$
	
	with $\tilde{h}:=\frac{1}{2} \sqrt{\left|v+v_*\right|^2\left|v-v_*\right|^2-\left(v+v_*\right)^2 \cdot\left(v-v_*\right)^2}=\sqrt{|v|^2\left|v_*\right|^2-\left(v \cdot v_*\right)^2}$.

	- If $\omega=\frac{\sigma-(\sigma \cdot \mathbf{n}) \mathbf{n}}{|\sigma-(\sigma \cdot \mathbf{n}) \mathbf{n}|}$ (which implies $\omega \perp\left(v-v_*\right))$, then
	$$
	\left\langle v^{\prime}\right\rangle^2=E(\theta)+\sin \theta\left|v-v_*\right| v \cdot \omega, \quad \text { which implies } \quad \sin \theta\left|v-v_*\right| v \cdot \omega=\sin \theta \, (\mathbf{j} \cdot \hat{\omega})\, \tilde{h}
	$$
	We also have $\omega=\tilde{\omega} \cos \frac{\theta}{2}+\frac{v^{\prime}-v_*}{\left|v^{\prime}-v_*\right|} \sin \frac{\theta}{2}$, where $\tilde{\omega}=\left(v^{\prime}-v\right) /\left|v^{\prime}-v\right|$.

	Moreover, if $k \geq 2, l_k:=[(k+1) / 2]$, then

\ben\label{eq:Et}
\begin{aligned}
& \sum_{p=1}^{l_{k / 2}-1} \frac{\Gamma(k / 2+1)}{\Gamma(p+1) \Gamma(k / 2+1-p)}\left[\left(\langle v\rangle^2 \cos ^2(\theta / 2)\right)^p\left(\left\langle v_*\right\rangle^2 \sin ^2(\theta / 2)\right)^{k / 2-p}\right. \\
& \left.+\left(\langle v\rangle^2 \cos ^2(\theta / 2)\right)^{k / 2-p}\left(\left\langle v_*\right\rangle^2 \sin ^2(\theta / 2)\right)^p\right]+\left(\cos ^2(\theta / 2)\right)^{k / 2}\langle v\rangle^k+\left(\sin ^2(\theta / 2)\right)^{k / 2}\left\langle v_*\right\rangle^k \\
\leq & (E(\theta))^{k / 2} \leq \sum_{p=1}^{l_{k / 2}} \frac{\Gamma(k / 2+1)}{\Gamma(p+1) \Gamma(k / 2+1-p)}\left[\left(\langle v\rangle^2 \cos ^2(\theta / 2)\right)^p\left(\left\langle v_*\right\rangle^2 \sin ^2(\theta / 2)\right)^{k / 2-p}\right. \\
& \left.+\left(\langle v\rangle^2 \cos ^2(\theta / 2)\right)^{k / 2-p}\left(\left\langle v_*\right\rangle^2 \sin ^2(\theta / 2)\right)^p\right]+\left(\cos ^2(\theta / 2)\right)^{k / 2}\langle v\rangle^k+\left(\sin ^2(\theta / 2)\right)^{k / 2}\left\langle v_*\right\rangle^k
\end{aligned}
\een
where $\Gamma$ denotes the Gamma function.
\end{lem}

\begin{rmk}
We make an observation regarding the term $\left|v-v_*\right| v \cdot \omega= (\mathbf{j} \cdot \hat{\omega})\, \tilde{h}$: since $|\omega|=1$ with $\omega \perp\left(v-v_*\right)$, we have $v \cdot \omega=v_* \cdot \omega$ and $|v \cdot \omega|\le |v|$, $|v_* \cdot \omega|\le |v_*|$, hence $|v \cdot \omega|\le \min\{|v|,|v_*|\}$. Using $|v-v_*|\le 2\max\{|v|,|v_*|\}$, we obtain
\ben\label{eq:vw}
|(\mathbf{j} \cdot \hat{\omega})\, \tilde{h}|=\left|v-v_*\right| |v \cdot \omega|\le 2\min\{|v|,|v_*|\}\max\{|v|,|v_*|\}=2|v||v_*|.
\een
Also, using $\sin\theta=2\sin(\theta/2)\cos(\theta/2)$ and the definition of $E(\theta)$, we have
\ben\label{eq:vwEt}
|\langle v^{\prime}\rangle^2-E(\theta)|=|\sin \theta \, (\mathbf{j} \cdot \hat{\omega})\, \tilde{h}|\le 4\sin(\theta/2)\cos(\theta/2)|v||v_*|\le 2(|v|^2 \cos ^2(\theta / 2)+| v_*|^2 \sin ^2(\theta / 2))\le 2E(\theta).
\een
\end{rmk}
\begin{lem}\label{thm:bv} For smooth functions $f, g, h$, if $\gamma\ge 0$ and $b \in L^1([0,1])$, then
	$$
	\begin{aligned}
	& \left|\int_{\mathbb{R}^6 \times \mathbb{S}^2} b(\cos \theta)\left|v-v_*\right|^\gamma f_* g h^{\prime}\, d v\, d v_*\, d \sigma\right| \ls \int_{\mathbb{S}^2} b(\cos \theta) \sin ^{-3 / 2-\gamma / 2} \frac{\theta}{2}\, d \sigma\,\|g\|_{L^1_\gamma}\|f\|_{L^2_{\gamma/2}}\|h\|_{L^2_{\gamma/2}}, \\
	& \left|\int_{\mathbb{R}^6 \times \mathbb{S}^2} b(\cos \theta)\left|v-v_*\right|^\gamma f_* g h^{\prime} \,d v\, d v_*\, d \sigma\right| \ls \int_{\mathbb{S}^2} b(\cos \theta) \cos ^{-3 / 2-\gamma / 2} \frac{\theta}{2}\, d \sigma\,\|f\|_{L^1_\gamma}\|g\|_{L^2_{\gamma/2}}\|h\|_{L^2_{\gamma/2}}.
	\end{aligned}
	$$
\end{lem}
\begin{proof}
We prove only the first inequality; the second follows by the same argument. 
Using the change of variables in Lemma \ref{thm:ch}, we obtain
\beq
\ba
& \left|\int_{\mathbb{R}^6 \times \mathbb{S}^2} b(\cos \theta)\left|v-v_*\right|^\gamma f_* g h^{\prime} \,d v\, d v_*\, d \sigma\right| \leq\left(\int_{\mathbb{R}^6 \times \mathbb{S}^2} b(\cos \theta) \sin ^{-\frac{3}{2}-\frac{\gamma}{2}} \frac{\theta}{2}\left|v-v_*\right|^\gamma \left|f_*\right|^2|g| \,d v\, d v_*\, d \sigma\right)^{\frac{1}{2}}\\&\times \left(\int_{\mathbb{R}^6 \times \mathbb{S}^2}|g| \, b(\cos \theta) \sin ^{\frac{3}{2}+\frac{\gamma}{2}} \frac{\theta}{2}\left|v-v_*\right|^\gamma | h^{\prime}|^2\, d v\, d v_*\, d \sigma\right)^{\frac{1}{2}} \\ &\ls \left(\int_{\mathbb{S}^2} b(\cos \theta) \sin ^{-\frac{3}{2}-\frac{\gamma}{2}} \frac{\theta}{2} d \sigma\right)\|g\|_{L_{\gamma}^1}\|f\|_{L^2_{\gamma / 2}}\|h\|_{L^2_{\gamma / 2}},
\ea
\eeq
using the fact that $\left|v-v_*\right|^\gamma\ls \wei^\gamma\weis^\gamma$ when $\gamma\ge 0$. This completes the proof of the lemma.
\end{proof}
\begin{rmk}
Using the fact that $\left|v-v_*\right|^\gamma\ls \wei^\gamma\weis^\gamma$ and the lemma above, we also have
\ben\label{eq:Kfgh1}
\left|\int_{\mathbb{R}^6 \times \mathbb{S}^2} b(\cos \theta)\left|v-v_*\right|^\gamma f_* g h^{\prime}\, d v\, d v_*\, d \sigma\right| \ls \int_{\mathbb{S}^2} b(\cos \theta) \sin ^{-3 / 2} \frac{\theta}{2}\, d \sigma\,\|g\|_{L^1_\gamma}\|f\|_{L^2_{\gamma}}\|h\|_{L^2},
\een
\ben\label{eq:Kfgh2}
\left|\int_{\mathbb{R}^6 \times \mathbb{S}^2} b(\cos \theta)\left|v-v_*\right|^\gamma f_* g h^{\prime}\, d v\, d v_*\, d \sigma\right| \ls \int_{\mathbb{S}^2} b(\cos \theta) \cos ^{-3 / 2} \frac{\theta}{2}\, d \sigma\,\|f\|_{L^1_\gamma}\|g\|_{L^2_{\gamma}}\|h\|_{L^2}.
\een
\end{rmk}
\begin{lem}
For $p,q>-2$,
\beq
\int_{\S^2}\cos^{p}(\theta/2)\sin^q(\theta/2)\, d\sigma=2\pi\int_0^\pi\cos^{p}(\theta/2)\sin^q(\theta/2)\sin\theta\,d\theta=4\pi B\rr{\f{p+2}{2},\f{q+2}{2}}.
\eeq
\end{lem}

This yields the following estimates:
\begin{proof}[Proof of Theorem \ref{thm:Qfgh} (1)]
	We use the following decomposition: for any functions $f,g,h$ and $k>0$, we introduce $E(\theta)$ from Lemma \ref{thm:E} and apply the well-known pre/post-collisional change of variables $\left(v, v_*, \sigma\right) \rightarrow\left(v^{\prime}, v_*^{\prime},\left(v-v_*\right) /\left|v-v_*\right|\right)$ for $Q^+$; we obtain
	\beq
	\ba
		&\br{Q(f, g), h\wei^{2 k}}=  \int_{\mathbb{R}^6 \times \mathbb{S}^2} \left|v-v_*\right|^\gamma f_*g \overline{h^{\prime}}\left\langle v^{\prime}\right\rangle^k E(\theta)^{k / 2}\, d v\, d v_*\, d \sigma \\
		& +\int_{\mathbb{R}^6 \times \mathbb{S}^2} \left|v-v_*\right|^\gamma f_* g \overline{h^{\prime}}\left\langle v^{\prime}\right\rangle^k\left(\left\langle v^{\prime}\right\rangle^k-E(\theta)^{k / 2}\right)\, d v\, d v_*\, d \sigma\\
		&-\int_{\mathbb{R}^6 \times \mathbb{S}^2} \left|v-v_*\right|^\gamma f_*g\overline{h}\langle v\rangle^{2 k}\, d v\, d v_*\, d \sigma=: I_1+I_2+I_3.
\ea
\eeq

\underline{\it{Estimate of $I_1$.} } For the first term $I_1$, we decompose $(E(\theta))^{k/2}$ using \eqref{eq:Et} and obtain
\beq
\ba
|I_1|&\le \int_{\mathbb{R}^6 \times \mathbb{S}^2} \left|v-v_*\right|^\gamma |f_*|\, |g|\, |h^{\prime}|\left\langle v^{\prime}\right\rangle^k E(\theta)^{k / 2}\, d v\, d v_*\, d \sigma\\
&\le \int_{\R^6\times \S^2}|v-v_*|^\gamma |f_*|\, |g|\, |h'|\br{v'}^k\wei^k\cos^k(\theta/2)\,dv\, dv_*\,d\sigma\\
&+\sum_{p=1}^{l_{k / 2}} \frac{\Gamma(k / 2+1)}{\Gamma(p+1) \Gamma(k / 2+1-p)} \int_{\mathbb{R}^6 \times \mathbb{S}^2} \left|v-v_*\right|^\gamma |f_*|\, |g|\, |h'|\left\langle v^{\prime}\right\rangle^k\\
&\left[\left(\langle v\rangle^2 \cos ^2(\theta / 2)\right)^p\left(\left\langle v_*\right\rangle^2 \sin ^2(\theta / 2)\right)^{k / 2-p}+\left(\langle v\rangle^2 \cos ^2(\theta / 2)\right)^{k / 2-p}\left(\left\langle v_*\right\rangle^2 \sin ^2(\theta / 2)\right)^p\right]\,dv\, dv_*\,d\sigma\\
&+\int_{\R^6\times \S^2}|v-v_*|^\gamma |f_*|\, |g|\, |h'|\br{v'}^k\weis^k\sin^k(\theta/2)\,dv\, dv_*\,d\sigma
\\&:=I_{11}+I_{12}+I_{13}.
\ea
\eeq

\underline{$\bullet$ Estimate of $I_{11},I_{13}$.} By Lemma \ref{thm:bv}, we have
\beq
\ba
I_{11}&=\int_{\R^6\times\S^2}|v-v_*|^\gamma |f_*|\, |g|\, |h'|\br{v'}^k\wei^k\cos^{k}(\theta/2)\,dv\, dv_*\,d\sigma\\
&=\int_{\R^6\times\S^2}|v-v_*|^\gamma |f_*|\, (|g|\wei^k)( |h'|\br{v'}^k)\cos^{k}(\theta/2)\,dv\, dv_*\,d\sigma\\
&\ls \int_{\S^2}\cos^{k-3/2-\gamma/2}(\theta/2)\,d\sigma\, \|f\|_{L^1_{\gamma}}\|g\|_{L^2_{k+\gamma/2}}\|h\|_{L^2_{k+\gamma/2}}\ls \f{1}{\sqrt{k}}\|f\|_{L^1_{\gamma}}\|g\|_{L^2_{k+\gamma/2}}\|h\|_{L^2_{k+\gamma/2}}.
\ea
\eeq

Here we use the fact that $\int_{\S^2}\cos^{k-3/2-\gamma/2}(\theta/2)\,d\sigma=\f{8\pi}{k+1/2-\gamma/2}\ls \f{1}{\sqrt{k}}$ when $k\ge 8$. Similarly, for $I_{13}$ we have
\beq
I_{13}\ls \int_{\S^2}\sin^{k-3/2-\gamma/2}(\theta/2)\,d\sigma \, \|g\|_{L^1_\gamma}\|f\|_{L^2_{k+\gamma/2}}\|h\|_{L^2_{k+\gamma/2}}\ls \f{1}{\sqrt{k}}\|g\|_{L^1_\gamma}\|f\|_{L^2_{k+\gamma/2}}\|h\|_{L^2_{k+\gamma/2}}.
\eeq

\underline{$\bullet$ Estimate of $I_{12}$.}
For $I_{12}$, for each term, applying \eqref{eq:Kfgh2} yields, for $p=1,2,\cdots, l_{k/2}$,
\beq
\ba
&\int_{\R^6\times\S^2}|v-v_*|^\gamma |f_*|\, |g|\, |h'|\left\langle v^{\prime}\right\rangle^k\wei^{k-2p}\weis^{2p}\cos^{k-2p}(\theta/2)\sin^{2p}(\theta/2)\,d\sigma \,dv \,dv_*
\\& =\int_{\R^6\times\S^2}|v-v_*|^\gamma (|f_*|\weis^{2p})(|g|\wei^{k-2p}) (|h'|\left\langle v^{\prime}\right\rangle^k)\cos^{k-2p}(\theta/2)\sin^{2p}(\theta/2)\,d\sigma \,dv \,dv_*
\\ &\ls \int \cos^{k-2p-3/2}(\theta/2)\sin^{2p}(\theta/2)\, d\sigma\, \|f\|_{L^1_{2p+\gamma}}\|g\|_{L^2_{k-2p+\gamma}}\|h\|_{L^2_{k}}
\\ &\ls  B\left( \f{k}{2}-p+\f14,p+1 \right)\|f\|_{L^1_{2p+\gamma}}\|g\|_{L^2_{k-2p+\gamma}}\|h\|_{L^2_{k}}.
\ea
\eeq
Similarly, applying \eqref{eq:Kfgh1} also gives
\beq
\ba
&\int_{\R^6\times\S^2}|v-v_*|^\gamma |f_*|\, |g|\, |h'|\left\langle v^{\prime}\right\rangle^k\wei^{2p}\weis^{k-2p}\cos^{2p}(\theta/2)\sin^{k-2p}(\theta/2)\,d\sigma\, dv\, dv_*\\
&\ls B\left( p+1,\f{k}{2}-p+\f14 \right)\|g\|_{L^1_{2p+\gamma}}\|f\|_{L^2_{k-2p+\gamma}}\|h\|_{L^2_{k}}.
\ea
\eeq
All Beta parameters in the two previous displays are positive for $1\le p\le l_{k/2}$ and $k\ge8$; since both sums are finite, they are bounded by constants depending only on $k$.
Thus, with $1\le p\le l_{k/2}\le \f{k/2+1}{2}$, we have $2p\le \f{k}{2}+1\le k-3$ when $k\ge 8$, and $k-2p\le k-2$; therefore
\beq
I_{12}\ls \|f\|_{L^1_{k-3+\gamma}}\|g\|_{L^2_{k-2+\gamma}}\|h\|_{L^2_{k}}+\|g\|_{L^1_{k-3+\gamma}}\|f\|_{L^2_{k-2+\gamma}}\|h\|_{L^2_{k}}.
\eeq
Since $0\le \gamma\le 1$, we have $\|f\|_{L^1_{k-3+\gamma}}\ls \|f\|_{L^2_k}$, and hence
\beq
I_{12}\ls \|f\|_{L^2_k}\|g\|_{L^2_k}\|h\|_{L^2_{k}}.
\eeq

\underline{\it{Estimate of $I_2$.} } Using Lemma \ref{thm:E} and Taylor expansion, we obtain
\begin{equation*}
\begin{aligned}
\left\langle v^{\prime}\right\rangle^k-(E(\theta))^{k / 2}&=  \frac{k}{2}\left[\left(\langle v\rangle^2 \cos ^2(\theta / 2)\right)^{k / 2-1}+\left((E(\theta))^{k / 2-1}-\left(\langle v\rangle^2 \cos ^2(\theta / 2)\right)^{k / 2-1}\right)\right]\left|v-v_*\right| \sin \theta\, (v \cdot \omega) \\
& +\frac{k(k-2)}{4} \int_0^1(1-t)(E(\theta)+t \tilde{h} \sin \theta\, (\mathbf{j} \cdot \hat{\omega}))^{k / 2-2} d t\, \tilde{h}^2 \sin ^2 \theta\, (\mathbf{j} \cdot \hat{\omega})^2.
\end{aligned}
\end{equation*}
Thus we can decompose $I_{2}$ as follows:
\beq
\ba
|I_2|\ls &k\left|\int_{\mathbb{R}^6 \times \mathbb{S}^2} \left| v-v_*\right|^{1+\gamma} f_* g\langle v\rangle^{k-2} \overline{h^{\prime}}\left\langle v^{\prime}\right\rangle^k \cos ^{k-1}(\theta / 2) \sin (\theta / 2)(v \cdot \omega) \,d v\, d v_*\, d \sigma \right|\\
	&+k\left|\int_{\mathbb{R}^6 \times \mathbb{S}^2} \left| v-v_*\right|^{1+\gamma} f_* g \overline{h^{\prime}}\left\langle v^{\prime}\right\rangle^k\left[(E(\theta))^{k / 2-1}-\left(\langle v\rangle^2 \cos ^2(\theta / 2)\right)^{k / 2-1}\right] \sin \theta(v \cdot \omega)\, d v \,d v_*\, d \sigma \right|\\
	&+k^2\left|\int_{\mathbb{R}^6 \times \mathbb{S}^2} \int_0^1 \sin ^2 \theta\left|v-v_*\right|^\gamma(1-t)(E(\theta)+t \tilde{h} \sin \theta\,(\mathbf{j} \cdot \hat{\omega}))^{k/2-2} \tilde{h}^2(\mathbf{j} \cdot \hat{\omega})^2  f_*g\overline{h'}\left\langle v^{\prime}\right\rangle^{k} \,d v \,d v_* d \sigma \,d t .
	\right| \\
	&=:I_{21}+I_{22}+I_{23}.
\ea
\eeq

\underline{$\bullet$ Estimate of $I_{21}$.} From \eqref{eq:vw}, we have $\left| v-v_*\right| |v \cdot \omega|\ls |v||v_*|$; applying \eqref{eq:Kfgh2} yields
	\beq
	\ba
	&I_{21}\ls k\int_{\R^6\times\S^2}|v-v_*|^{\gamma}|f_*gh'|\weis\wei^{k-1}\br{v'}^k\cos ^{k-1}(\theta / 2) \sin (\theta / 2)\,d\sigma \,dv\, dv_*\\
	&\ls k B\left( \f{k}{2}-\f14,\f32 \right)\|f\|_{L^1_{1+\gamma}}\|g\|_{L^2_{k-1+\gamma}}\|h\|_{L^2_{k}}\ls \|f\|_{L^1_{1+\gamma}}\|g\|_{L^2_{k-1+\gamma}}\|h\|_{L^2_{k}}.
	\ea
	\eeq
Since $0\le \gamma\le 1$ and $k\ge 8$, we have $\|f\|_{L^1_{1+\gamma}}\ls \|f\|_{L^2_k}$, and therefore
	\beq
	I_{21}\ls \|f\|_{L^2_k}\|g\|_{L^2_k}\|h\|_{L^2_{k}}.
	\eeq
\underline{$\bullet$ Estimate of $I_{23}$.} From \eqref{eq:vw}, we have $|\tilde{h} \sin \theta\, (\mathbf{j} \cdot \hat{\omega})|\ls |v|\,|v_*|\sin(\theta/2)\cos(\theta/2)\ls E(\theta)\ls \max\{\wei^2,\weis^2\}$. Consequently,
	\beq
	\begin{aligned}
	&I_{23} \ls_k \int_{\mathbb{R}^6 \times \mathbb{S}^2} \sin ^2 \theta\left|v-v_*\right|^{\gamma} |v|^2|v_*|^2 \max \left\{\langle v\rangle^2 ,\left\langle v_*\right\rangle^2 \right\}^{k / 2-2}  \left|f_*gh'\right|\langle v^{\prime}\rangle^k\, d v\, d v_*\, d \sigma,
		\end{aligned}
		\eeq
	Since $ \max \left\{\langle v\rangle^2 ,\left\langle v_*\right\rangle^2 \right\}^{k / 2-2}\le (\wei^{k-4}+\weis^{k-4})$, and using $|v|\le \wei$, $|v_*|\le \weis$ together with Lemma \ref{thm:bv}, we obtain
	\beq
	\ba
	& I_{23}\ls \int_{\R^6\times\S^2}\sin^2\theta |v-v_*|^{\gamma}|f_*gh'|\wei^2\weis^2(\wei^{k-4}+\weis^{k-4})\br{v'}^k\, d v\, d v_*\, d \sigma\\
	& \ls \int_{\R^6\times\S^2}\sin^2\theta \sin^{-3/2}(\theta/2)\,d\sigma\, \|g\|_{L^1_{2+\gamma}}\|f\|_{L^2_{k-2+\gamma}}\|h\|_{L^2_{k}}\\ &+\int_{\R^6\times\S^2}\sin^2\theta \cos^{-3/2}(\theta/2)\,d\sigma\, \|f\|_{L^1_{2+\gamma}}\|g\|_{L^2_{k-2+\gamma}}\|h\|_{L^2_{k}}
	\\ &\ls (\|g\|_{L^1_{2+\gamma}}\|f\|_{L^2_{k-2+\gamma}}+\|f\|_{L^1_{2+\gamma}}\|g\|_{L^2_{k-2+\gamma}})\|h\|_{L^2_{k}}.
	\ea
	\eeq
	Since $0\le \gamma\le 1$ and $k\ge 8$, we have $\|f\|_{L^1_{2+\gamma}}\ls \|f\|_{L^2_k}$, and hence
 $I_{23}\ls \|f\|_{L^2_k}\|g\|_{L^2_k}\|h\|_{L^2_{k}}.$
 
\underline{$\bullet$ Estimate of $I_{22}$.} Similarly to the decomposition of $I_{1}$, from \eqref{eq:Et} we have 
	\beq
	\ba
	|(E(\theta))^{k / 2-1}-(\langle v\rangle^2 \cos ^2(\theta / 2))^{k / 2-1}|&\le \sum_{p=1}^{l_{k / 2-1}} \frac{\Gamma(k / 2)}{\Gamma(p+1) \Gamma(k / 2-p)}\left[\left(\langle v\rangle^2 \cos ^2(\theta / 2)\right)^{p}\left(\left\langle v_*\right\rangle^2 \sin ^2(\theta / 2)\right)^{k / 2-1-p}\right. \\ &\left.+\left(\langle v\rangle^2 \cos ^2(\theta / 2)\right)^{k / 2-1-p}\left(\left\langle v_*\right\rangle^2 \sin ^2(\theta / 2)\right)^p\right]
	+\weis^{k-2}\sin^{k-2}(\theta/2).
	\ea
	\eeq
For each term, for $p=1,2,\cdots, l_{k / 2-1}$, applying \eqref{eq:vw} and \eqref{eq:Kfgh2} yields
\beq
\ba
&\left|\int_{\R^6\times\S^2}|v-v_*|^{1+\gamma}|f_*gh'|\weis^{k-2-2p}\wei^{2p}\br{v'}^k(v\cdot \omega)\cos^{2p+1}(\theta/2)\sin^{k-1-2p}(\theta/2)\,d\sigma \, dv\,dv_*\right|\\
& \ls \int_{\R^6\times\S^2}|v-v_*|^{\gamma}|f_*gh'|\weis^{k-1-2p}\wei^{2p+1}\br{v'}^k\cos^{2p+1}(\theta/2)\sin^{k-1-2p}(\theta/2)\, d\sigma\, dv\,dv_*\\
&\ls B\left( p+\f34,\f{k}{2}-p+\f12 \right)\|f\|_{L^1_{k-2p-1+\gamma}}\|g\|_{L^2_{2p+1+\gamma}}\|h\|_{L^2_{k}},
\ea
\eeq
and
\beq
\ba
&\left|\int_{\R^6\times\S^2}|v-v_*|^{1+\gamma}|f_*gh'|\weis^{2p}\wei^{k-2-2p}\br{v'}^k(v\cdot \omega)\cos^{k-1-2p}(\theta/2)\sin^{2p+1}(\theta/2)\,d\sigma \,dv\, dv_*\right|\\
& \ls \int_{\R^6\times\S^2}|v-v_*|^{\gamma}|f_*gh'|\weis^{2p+1}\wei^{k-1-2p}\br{v'}^k\cos^{k-1-2p}(\theta/2)\sin^{2p+1}(\theta/2)\, d\sigma\, dv\,dv_*\\
&\ls B\left( \f{k}{2}-p-\f14,p+\f32\right)\|f\|_{L^1_{2p+1+\gamma}}\|g\|_{L^2_{k-2p-1+\gamma}}\|h\|_{L^2_{k}}.
\ea
\eeq
All Beta parameters in the two previous displays are positive for $1\le p\le l_{k/2-1}$ and $k\ge8$; since both sums are finite, the binomial coefficients and the prefactor $k$ can all be absorbed into constants depending only on $k$.
Since $1\le p\le l_{k / 2-1}=\lfloor k/4\rfloor$, we have $2p+1\le k/2+1\le k-3$ for $k\ge 8$, and $k-2p-1\le k-2$.
For the last term, applying \eqref{eq:Kfgh1} gives
	\beq
	\ba
	&\left|\int_{\R^6\times\S^2}|v-v_*|^{1+\gamma}f_*g\overline{h'}\weis^{k-2}\br{v'}^k(v\cdot \omega)\cos(\theta/2)\sin^{k-1}(\theta/2)\,d\sigma\, dv\,dv_*\right|\\
	&\ls \int_{\R^6\times\S^2}|v-v_*|^{\gamma}|f_*gh'|\weis^{k-1}\wei\br{v'}^k\cos(\theta/2)\sin^{k-1}(\theta/2)\,d\sigma\, dv\,dv_*\\
	&\ls B\left( \f{k}{2}-\f14,\f32 \right)\|g\|_{L^1_{1+\gamma}}\|f\|_{L^2_{k-1+\gamma}}\|h\|_{L^2_{k}}.
	\ea
	\eeq
	Combining these results, we obtain
	\beq
	I_{22}\ls \|f\|_{L^1_{k-3+\gamma}}\|g\|_{L^2_{k-2+\gamma}}\|h\|_{L^2_{k}}+\|g\|_{L^1_{1+\gamma}}\|f\|_{L^2_{k-1+\gamma}}\|h\|_{L^2_{k}},
	\eeq
since $0\le \gamma\le 1$, we have $\|f\|_{L^1_{k-3+\gamma}}\ls \|f\|_{L^2_{k}}$, and therefore
	\beq
	I_{22}\ls \|f\|_{L^2_k}\|g\|_{L^2_k}\|h\|_{L^2_{k}}.
	\eeq
\underline{\it{Estimate of $I_3$.} } Since $I_3=\br{Q^-(f,g),h\wei^{2k}}$, we have 
\beq
|I_3|\le \int_{\mathbb{R}^6 \times \mathbb{S}^2} \left|v-v_*\right|^\gamma |f_*gh|\langle v\rangle^{2 k}\, d v\, d v_*\, d \sigma\ls \|f\|_{L^1_{\gamma}}\|g\|_{L^2_{k+\gamma/2}}\|h\|_{L^2_{k+\gamma/2}}.
\eeq

Combining all the terms, we obtain the bound
\beq
\ba
\left|\br{Q(f, g), h\wei^{2 k}}\right|\le \f{C_1}{\sqrt{k}}\|g\|_{L^1_\gamma}\|f\|_{L^2_{k+\gamma/2}}\|h\|_{L^2_{k+\gamma/2}}+C_2\|f\|_{L^1_{\gamma}}\|g\|_{L^2_{k+\gamma/2}}\|h\|_{L^2_{k+\gamma/2}}+C_k\|f\|_{L^2_k}\|g\|_{L^2_k}\|h\|_{L^2_{k}},
\ea
\eeq
where $C_1,C_2$ are absolute constants and $C_k$ depends only on $k$. This completes the proof.

\noindent\textit{Proof of Theorem \ref{thm:Qfgh} (2).} It suffices to establish that, for any functions $f,g,h$ and any $k\ge 8$, the following inequality holds:
\ben\label{eq:Qfgh2}
\ba
&\left|\br{Q(f,g),h\wei^{2k}}\right|\ls (\|f\|_{L^1_\gamma}\|g\|_{L^2_{k+\gamma}}+\|g\|_{L^1_\gamma}\|f\|_{L^2_{k+\gamma}})\|h\|_{L^2_k}+C_k\|f\|_{L^2_k}\|g\|_{L^2_k}\|h\|_{L^2_{k}}.
\ea
\een
We employ the same decomposition as in the proof of part (1) and need only estimate $I_{11}, I_{13},$ and $I_3$. For $I_{11}$, applying \eqref{eq:Kfgh2} yields
\beq
\ba
I_{11}&=\int_{\R^6\times\S^2}|v-v_*|^\gamma |f_*|\, |g|\, |h'|\br{v'}^k\wei^k\cos^{k}(\theta/2)\,dv\, dv_*\,d\sigma\\
&\ls \int_{\S^2}\cos^{k-3/2}(\theta/2)\,d\sigma\, \|f\|_{L^1_{\gamma}}\|g\|_{L^2_{k+\gamma}}\|h\|_{L^2_{k}}\ls \|f\|_{L^1_{\gamma}}\|g\|_{L^2_{k+\gamma}}\|h\|_{L^2_{k}},
\ea
\eeq
where we have used that $\int_{\S^2}\cos^{k-3/2}(\theta/2)\,d\sigma=\f{8\pi}{k+1/2}\ls 1$ for $k\ge 8$. Similarly, for $I_{13}$, applying \eqref{eq:Kfgh1} with the weight $\weis^{k}$ carried by the star slot and $\br{v'}^k$ by the primed slot, we obtain
\beq
I_{13}\ls \rr{\int_{\S^2}\sin^{k-3/2}(\theta/2)\,d\sigma} \|g\|_{L^1_\gamma}\|f\|_{L^2_{k+\gamma}}\|h\|_{L^2_{k}}=\f{8\pi}{k+1/2}\|g\|_{L^1_\gamma}\|f\|_{L^2_{k+\gamma}}\|h\|_{L^2_{k}}\ls \f{1}{\sqrt{k}}\|g\|_{L^1_\gamma}\|f\|_{L^2_{k+\gamma}}\|h\|_{L^2_{k}}.
\eeq
For $I_3$, we estimate
\beq
|I_3|\le \int_{\mathbb{R}^6 \times \mathbb{S}^2} \left|v-v_*\right|^\gamma |f_*gh|\langle v\rangle^{2 k}\, d v\, d v_*\, d \sigma\ls \|f\|_{L^1_{\gamma}}\|g\|_{L^2_{k+\gamma}}\|h\|_{L^2_{k}}.
\eeq
Combining these estimates with those for $I_{12}$ and $I_2$ from part (1), we conclude that \eqref{eq:Qfgh2} holds. This completes the proof of Theorem \ref{thm:Qfgh} (2).
\end{proof}
\medskip
\begin{proof}[Proof of Theorem \ref{thm:Qfgg}] Following a similar approach to the previous proof, we employ the following decomposition: for functions $f,g$ with $f$ non-negative satisfying $\|f\|_{L^1}>\delta$ and $\|f\|_{L^1_2}+\|f\|_{L\log L}<H$, we have
		\beq
	\ba
		&\br{Q(f, g), g\wei^{2 k}}=  \int_{\mathbb{R}^6 \times \mathbb{S}^2} \left|v-v_*\right|^\gamma f_*\left(g\overline{g^{\prime}}\left\langle v^{\prime}\right\rangle^k E(\theta)^{k / 2}-|g|^2\langle v\rangle^{2 k}\right)\, d v\, d v_*\, d \sigma \\
		& +\int_{\mathbb{R}^6 \times \mathbb{S}^2} \left|v-v_*\right|^\gamma f_* g\overline{g^{\prime}}\left\langle v^{\prime}\right\rangle^k\left(\left\langle v^{\prime}\right\rangle^k-E(\theta)^{k / 2}\right)\, d v\, d v_*\, d \sigma=: I_1+I_2.
\ea
\eeq

\underline{\it{Estimate of $I_1$.} } For the first term $I_1$, we decompose $(E(\theta))^{k/2}$ using \eqref{eq:Et} to obtain

\beq
\ba
&\operatorname{Re}I_1\le  \int_{\R^6\times \S^2}|v-v_*|^\gamma f_*\left(|g|\,|g'|\br{v'}^k\wei^k\cos^k(\theta/2)-|g|^2\wei^{2k}\right)\,dv\, dv_*\,d\sigma\\
	&+\sum_{p=1}^{l_{k / 2}} \frac{\Gamma(k / 2+1)}{\Gamma(p+1) \Gamma(k / 2+1-p)} \int_{\mathbb{R}^6 \times \mathbb{S}^2}\left|v-v_*\right|^\gamma f_*|g|\,|g'|\left\langle v^{\prime}\right\rangle^k\wei^{2p}\weis^{k-2p}\cos^{2p}(\theta/2)\sin^{k-2p}(\theta/2)\,dv\,dv_*\,d\sigma\\
	&+\sum_{p=1}^{l_{k / 2}} \frac{\Gamma(k / 2+1)}{\Gamma(p+1) \Gamma(k / 2+1-p)} \int_{\mathbb{R}^6 \times \mathbb{S}^2}\left|v-v_*\right|^\gamma f_*|g|\,|g'|\left\langle v^{\prime}\right\rangle^k\wei^{k-2p}\weis^{2p}\cos^{k-2p}(\theta/2)\sin^{2p}(\theta/2)\,dv\,dv_*\,d\sigma\\
	&+\int_{\R^6\times \S^2}|v-v_*|^\gamma f_*|g|\,|g'|\br{v'}^k\weis^k\sin^k(\theta/2)\,dv\, dv_*\,d\sigma=:I_{11}+I_{12}+I_{13},
	\ea
	\eeq
where the two mixed sums correspond to the two symmetric branches of the upper bound in \eqref{eq:Et}. For the mixed term, the same two estimates as in the proof of Theorem \ref{thm:Qfgh}~(1) (applied with $h=g$) yield
\beq
I_{12}\ls_k \|f\|_{L^1_{k-3+\gamma}}\|g\|_{L^2_{k-2+\gamma}}\|g\|_{L^2_{k}}+\|g\|_{L^1_{k-3+\gamma}}\|f\|_{L^2_{k-2+\gamma}}\|g\|_{L^2_{k}}\ls_k \|f\|_{L^2_{k}}\|g\|_{L^2_{k}}^2.
\eeq

	\underline{$\bullet$ Estimate of $I_{11}$.}
	 By the regular change of variables in Lemma \ref{thm:ch}~(1) applied to the primed term, together with the inequality $|g|\left|g^{\prime}\right|\left\langle v^{\prime}\right\rangle^k\langle v\rangle^k \cos ^k \frac{\theta}{2}-|g|^2\langle v\rangle^{2 k} \leq \frac{1}{2}\left(\left|g^{\prime}\right|^2\left\langle v^{\prime}\right\rangle^{2 k} \cos ^{2 k} \frac{\theta}{2}-|g|^2\langle v\rangle^{2 k}\right)$, we deduce
	\beq
	\begin{aligned}
		I_{11}& \leq \frac{1}{2} \int_{\mathbb{R}^6 \times \mathbb{S}^2}\left|v-v_*\right|^\gamma f_*|g|^2\langle v\rangle^{2 k}\left(\cos ^{2 k-3-\gamma} \frac{\theta}{2}-1\right) d v d v_* d \sigma \\
		& \leq-\frac{1}{2}\int_{\S^2}\left(1-\cos ^{2 k-3-\gamma} \frac{\theta}{2}\right)\,d\sigma \int_{\mathbb{R}^6}\left|v-v_*\right|^\gamma f_*\left|g(v)\langle v\rangle^k\right|^2 d v d v_*  \le -K\|g\|_{L^2_{k+\gamma/2}}^2.
		\end{aligned}
	\eeq
Here we use the well-known fact that $\int_{\R^3}|v-v_*|^\gamma f_*dv_*\ge C(f)\wei^{\gamma}$ for a non-negative function $f$, where $C(f)$ is a positive constant depending only on $\delta$ and $H$ under the conditions $\|f\|_{L^1}>\delta$ and $\|f\|_{L^1_2}+\|f\|_{ L\log L}<H$. Furthermore, by the spherical formula above,
\beq
\int_{\S^2}\left(1-\cos ^{2 k-3-\gamma} \frac{\theta}{2}\right)d\sigma=4\pi-\f{8\pi}{2k-1-\gamma}\ge4\pi-\f{8\pi}{14}=\f{24\pi}{7}>0, \eeq
since $2k-1-\gamma\ge14$ for $k\ge8$ and $0\le\gamma\le1$.
Thus the constant $K$ depends solely on $\delta$ and $H$.

\underline{$\bullet$ Estimate of $I_{13}$.}
The term $I_{13}$ coincides with the corresponding term in the previous proof with $h=g$. Hence we have
\beq
I_{13}\ls \f{1}{\sqrt{k}}\|g\|_{L^1_{\gamma}}\|f\|_{L^2_{k+\gamma/2}}\|g\|_{L^2_{k+\gamma/2}}.
\eeq \underline{\it{Estimate of $I_2$.} } The term $I_2$ is also identical to the corresponding term in the previous proof. Therefore
 $|I_2|\ls_k \|f\|_{L^2_{k}}\|g\|_{L^2_{k}}^2.$
 Combining these estimates, we obtain
\beq
\ba
\operatorname{Re}\br{Q(f, g), g\wei^{2 k}}+K\|g\|_{L^2_{k+\gamma/2}}^2\le \f{C_1}{\sqrt{k}}\|g\|_{L^1_{\gamma}}\|f\|_{L^2_{k+\gamma/2}}\|g\|_{L^2_{k+\gamma/2}}+C_k\|f\|_{L^2_{k}}\|g\|_{L^2_{k}}^2,
\ea
\eeq
where the constant $K$ depends only on $\delta$ and $H$, $C_1$ is an absolute constant, and $C_k$ depends solely on $k$.
This completes the proof of the theorem.
\end{proof}

\subsection{Proof of Lemma \ref{thm:betalmest}}\label{sec:proofbeta} 
\underline{\it{Point (1).}} For $l,m\in\Z^3$, let $\xi=\pi \lambda\,|l+m|$ and $\eta=\pi \lambda\,|l-m|$. It then suffices to consider the integral
\beq
B(\xi,\eta)=\int_0^1r^{2+\gamma}\sinc( \xi\,r)\,\sinc( \eta\,r)\,dr.
\eeq
Since $|\beta(l,m)|\ls R^{3+\gamma}|B(\xi,\eta)|$, it remains to prove that
\ben\label{eq:Blm}
|B(\xi,\eta)|\ls \br{\xi}^{-1}\br{\eta}^{-1}\br{\xi-\eta}^{-1}.
\een

\smallskip
\underline{Case 1: At least one of $\xi,\eta$ is zero.}
If $\xi=0$ and $\eta=0$, since $\sinc(0)=1$ we have $B(\xi,\eta)=\int_0^1r^{2+\gamma} dr=\f{1}{3+\gamma}\le 1$, which satisfies the bound \eqref{eq:Blm}.

If exactly one of $\xi,\eta$ is zero, by symmetry we may assume $\xi=0$ and $\eta\neq 0$. In this case,
\beq
B(\xi,\eta)=\int_0^1r^{2+\gamma}\sinc( \eta\,r)\, dr=\f{1}{\eta}\int_0^1r^{1+\gamma}\sin( \eta\,r)\,dr.
\eeq
Integrating by parts yields
\beq
\ba
B(\xi,\eta)=\f{1}{\eta^2}\rr{-\cos( \eta)+(1+\gamma)\rr{\int_0^1r^{\gamma}\cos( \eta\,r)\,dr}},
\ea
\eeq
Hence $|B(\xi,\eta)|\ls \f{1}{\eta^2}$. Since $l,m\in\Z^3$ and $\eta\neq 0$, we have $\eta\gs 1$, which implies $|B(\xi,\eta)|\ls \br{\eta}^{-2}$. With $\xi=0$, this establishes the bound \eqref{eq:Blm}.

\smallskip
\underline{Case 2: Both $\xi$ and $\eta$ are positive.} By the definition of $\sinc$,
\beq
B(\xi,\eta)=\f{1}{\xi\eta}\int_0^1r^{\gamma}\sin( \xi\,r)\,\sin( \eta\,r)\,dr.
\eeq
Therefore $|B(\xi,\eta)|\ls \f{1}{\xi\eta}\ls \br{\xi}^{-1}\br{\eta}^{-1}$ since $\xi\gs 1$ and $\eta\gs 1$. Moreover, when $|\xi-\eta|\le 1$, we have $\br{\xi-\eta}\ls 1$, and consequently
\beq
|B(\xi,\eta)|\ls \br{\xi}^{-1}\br{\eta}^{-1}\ls \br{\xi}^{-1}\br{\eta}^{-1}\br{\xi-\eta}^{-1}.
\eeq
Thus it remains to consider the case $|\xi-\eta|>1$. Using the identity $\sin a\sin b=\f{1}{2}(\cos(a-b)-\cos(a+b))$, we obtain
\beq
B(\xi,\eta)=\f{1}{2\xi\eta}\rr{\int_0^1 r^\gamma \cos((\xi-\eta)\,r)\, dr-\int_0^1 r^\gamma \cos((\xi+\eta)\,r)\, dr}.
\eeq
We treat both integrals in the same manner: for $\omega\in\R$, define $I(\omega)=\int_0^1 r^\gamma \cos(\omega\,r)\, dr$. Then
\beq
B(\xi,\eta)=\f{1}{2\xi\eta}(I(\xi-\eta)-I(\xi+\eta)).
\eeq
We now estimate $I(\omega)$ for $|\omega|>1$. If $\gamma=0$, $I(\omega)=\int_0^1 \cos(\omega\,r)\, dr=\f{1}{\omega}\sin(\omega).$
If $\gamma\in (0,1]$, integration by parts gives $I(\omega)=\f{1}{\omega}\rr{\sin(\omega)-\gamma\int_0^1 r^{\gamma-1}\sin(\omega\,r)\,dr}$. 
Hence $|I(\omega)|\ls \f{1}{|\omega|}$ for all $|\omega|>1$.

Returning to the estimate of $B(\xi,\eta)$, since both $\xi$ and $\eta$ are positive, we have $|\xi-\eta|\le \xi+\eta$ and $\xi+\eta>0$. Consequently, when $|\xi-\eta|>1$ and $\xi\gs 1,\ \eta\gs 1$,
\beq
|B(\xi,\eta)|\ls \f{1}{\xi\eta}\rr{\f{1}{|\xi-\eta|}+\f{1}{\xi+\eta}}\ls \f{1}{\xi\eta}\f{1}{|\xi-\eta|}\ls \br{\xi}^{-1}\br{\eta}^{-1}\br{\xi-\eta}^{-1}.
\eeq
This establishes \eqref{eq:Blm} and completes the proof of part (1).

\medskip
\underline{\it{Point (2).}}
The case $l=m=0$ is read directly from \eqref{eq:defbeta}; the equivalences below are used for $(l,m)\ne(0,0)$, and factors $|l|^\delta$ with $\delta=0$ are understood as $1$.
We first claim that for $l,m\in\Z^3$, 
\ben
\big|\,|l+m|-|l-m|\,\big|\sim\f{|l\cdot m|}{\sqrt{|l|^2+|m|^2}}.
\een
The proof is as follows: by direct computation,
\beq
\big|\,|l+m|-|l-m|\,\big|=\f{||l+m|^2-|l-m|^2|}{|l+m|+|l-m|}=\f{4|l\cdot m|}{|l+m|+|l-m|}.
\eeq
Since $(|l+m|+|l-m|)^2\ls |l|^2+|m|^2$ and $(|l+m|+|l-m|)^2\ge \max\{|l+m|^2,|l-m|^2\}=|l|^2+|m|^2+2|l\cdot m|\ge |l|^2+|m|^2$, the claim follows.

Using the estimate from part (1) and the asymptotic relation for $|l+m|-|l-m|$, we obtain
\beq
|\beta(l,m)|^2\ls R^{2(3+\gamma)} \f{1}{|l+m|^2+1}\f{1}{|l-m|^2+1}\f{|l|^2+|m|^2}{|l\cdot m|^2+|l|^2+|m|^2}.
\eeq
We define the sum $S$ by
\beq
S=\sum_{l\in\Z^3}|l|^\delta\,\f{1}{|l+m|^2+1}\f{1}{|l-m|^2+1}\f{|l|^2+|m|^2}{|l\cdot m|^2+|l|^2+|m|^2},
\eeq
Since $\sum_{l\in\Z^3}|l|^\delta|\beta(l,m)|^2\ls R^{2(3+\gamma)}\,S$, it suffices to estimate $S$.
When $m=0$, we have $|\beta(l,m)|^2\ls R^{2(3+\gamma)}\br{l}^{-4}$, so for $0\le\delta<1$,
$S\ls \sum_{l\in\Z^3}\f{|l|^\delta}{(|l|^2+1)^2}\ls 1.$
When $m\neq 0$, we have $|m|\ge1$. To compare the lattice sum with an integral, fix $0<\rho<1/4$ and use the disjoint balls $B(l,\rho)$, $l\in\Z^3$. For $x\in B(l,\rho)$, the factors $1+|x\pm m|^2$ and $|x|^2+|m|^2$ are comparable to their values at $l$, uniformly in $m$. The same is true of $|x\cdot m|^2+|x|^2+|m|^2$, since $|(x-l)\cdot m|\le\rho|m|$ and $|m|^2$ is included in this denominator. For $l\ne0$, $|x|^\delta\sim|l|^\delta$. At $l=0$ the summand vanishes if $\delta>0$; if $\delta=0$, both weight factors are understood as $1$ and the same comparison applies. Integrating over these balls and summing therefore gives
\beq
S\ls \int_{\R^3}\f{1}{|x+m|^2+1}\f{1}{|x-m|^2+1}\f{|x|^2+|m|^2}{|x\cdot m|^2+|x|^2+|m|^2}|x|^\delta\,dx.
\eeq
Applying polar coordinates with $x=r\omega$ where $\omega\in\S^2$, setting $\f{x\cdot m}{|x|\,|m|}=\omega\cdot \f{m}{|m|}=\cos \theta$, and using $\sin\theta\le1$, we obtain the upper bound
\beq
\ba
S&\ls \int_0^\infty \int_0^{\pi}\f{1}{|r\omega+m|^2+1}\f{1}{|r\omega-m|^2+1}\f{r^2+|m|^2}{r^2|m|^2\cos^2\theta+r^2+|m|^2} \,r^{2+\delta}\,d\theta\, dr
\\&=\int_0^\infty \int_0^{\pi}\f{1}{|r|^2+|m|^2+1+2r|m|\cos\theta}\f{1}{|r|^2+|m|^2+1-2r|m|\cos\theta}\f{r^2+|m|^2}{r^2|m|^2\cos^2\theta+r^2+|m|^2} \,r^{2+\delta}\,d\theta\, dr.
\ea
\eeq
We first estimate the integral in $\theta$. Since 
\beq
\f{r^2+|m|^2}{r^2|m|^2\cos^2\theta+r^2+|m|^2}\le \f{r^2+|m|^2+1}{r^2|m|^2\cos^2\theta+r^2+|m|^2+1}\ls \f{r^2+|m|^2+1}{4r^2|m|^2\cos^2\theta+r^2+|m|^2+1},
\eeq
for any $r>0$, let $A=A(r)=r^2+|m|^2+1$ and $B=B(r)=2r|m|$. Then the integral is bounded by $\int_0^\infty r^{2+\delta}\,I(r)\,dr$, where $I(r)$ is defined by
\beq
\ba
I(r):=&\int_0^{\pi}\f{1}{A+B\cos\theta}\f{1}{A-B\cos\theta}\f{A}{A+B^2\cos^2\theta}\,d\theta=\f12\int_0^\pi \rr{\f{1}{A+B\cos\theta}+\f{1}{A-B\cos\theta}}\f{1}{A+B^2\cos^2\theta}\,d\theta
\\=&\int_0^\pi \f{1}{A+B\cos\theta}\f{1}{A+B^2\cos^2\theta}d\theta.
\ea
\eeq
The last equality follows from the symmetry of $\cos\theta$ under the change of variable $\theta\mapsto \pi-\theta$. The integrand can be decomposed into the following partial fractions:
\beq
\f{1}{A+B\cos\theta}\f{1}{A+B^2\cos^2\theta}=\f{1}{A(A+1)}\rr{\f{1}{A+B\cos\theta}+\f{A-B\cos\theta}{A+B^2\cos^2\theta}}.
\eeq
Integrating in $\theta$ yields
\beq
\int_0^\pi \f{1}{A+B\cos\theta}\,d\theta=\f{\pi}{\sqrt{A^2-B^2}},\ \ 
\int_0^\pi \f{A-B\cos\theta}{A+B^2\cos^2\theta}\,d\theta=\f{\pi\sqrt{A}}{\sqrt{A+B^2}}.
\eeq
Therefore
\beq
I(r)=\f{\pi}{A(A+1)}\rr{\f{1}{\sqrt{A^2-B^2}}+\f{\sqrt{A}}{\sqrt{A+B^2}}}.
\eeq

We now consider the integration in $r$. This yields
\beq
S\ls \int_0^\infty r^{2+\delta} I(r)\,dr=\pi\int_0^\infty \f{r^{2+\delta}}{A(A+1)}\rr{\f{1}{\sqrt{A^2-B^2}}+\f{\sqrt{A}}{\sqrt{A+B^2}}}\,dr\ls I_1+I_2,
\eeq
with $I_1,I_2$ defined as
\beq
 I_1:=\int_0^\infty \f{r^{2+\delta}}{A^2\sqrt{A^2-B^2}}\,dr,\quad I_2:=\int_0^\infty \f{r^{2+\delta}}{A^{3/2}\sqrt{A+B^2}}\,dr.	
\eeq

For $I_1$, we make the change of variable $r=\sqrt{|m|^2+1}\tan \phi$ with $\phi\in [0,\pi/2]$. Then $dr=\sqrt{|m|^2+1}\sec^2\phi\,d\phi$ and
\beq
A=r^2+|m|^2+1=(|m|^2+1)(1+\tan^2\phi)=(|m|^2+1)\sec^2\phi,
\eeq
\beq
A^2-B^2=(|m|^2+1)^2\sec^4\phi-4|m|^2(|m|^2+1)\tan^2\phi=(|m|^2+1)^2\sec^4\phi\rr{1-\f{|m|^2}{|m|^2+1}\sin^22\phi}.
\eeq
Let $q=\f{|m|}{\sqrt{|m|^2+1}}$. Then the integral $I_1$ becomes
\beq
\ba
I_1&=\int_0^{\pi/2}\f{(|m|^2+1)^{\f{3+\delta}{2}}\tan^{2+\delta}\phi\,\sec^2\phi}{(|m|^2+1)^2\sec^4\phi \sqrt{(|m|^2+1)^2\sec^4\phi\,(1-q^2\sin^22\phi)}}\,d\phi\\ &=\f{1}{(|m|^2+1)^{\f{3-\delta}{2}}}\int_0^{\pi/2}\f{\sin^{2+\delta}\phi\,\cos^{2-\delta}\phi}{\sqrt{1-q^2\sin^22\phi}}\,d\phi\ls \f{1}{(|m|^2+1)^{\f{3-\delta}{2}}}\int_0^{\pi}\f{\sin^{2-\delta}\phi}{\sqrt{1-q^2\sin^2\phi}}\,d\phi.
\ea
\eeq
Here we use that for $0\le \delta<1$, $\sin^{2+\delta}\phi\,\cos^{2-\delta}\phi=2^{-(2-\delta)}\sin^{2\delta}\phi\,\sin^{2-\delta}2\phi\ls \sin^{2-\delta}2\phi$, together with the change of variable $\phi\mapsto 2\phi$.
For the elliptic integral, by symmetry and the change of variable $\phi\mapsto \pi/2-\phi$,
\beq
\int_0^{\pi}\f{\sin^{2-\delta}\phi}{\sqrt{1-q^2\sin^2\phi}}\,d\phi=2\int_0^{\pi/2}\f{\sin^{2-\delta}\phi}{\sqrt{1-q^2\sin^2\phi}}\,d\phi=2\int_0^{\pi/2}\f{\cos^{2-\delta}\phi}{\sqrt{1-q^2\cos^2\phi}}\,d\phi.
\eeq
To estimate this integral, we use the inequality $\sin\phi\ge \f{2}{\pi}\phi$ for $\phi\in[0,\pi/2]$, which implies $\cos^2\phi\le 1-\rr{\f{2}{\pi}\phi}^2$. Consequently, 
\beq
\ba
&\int_0^{\pi/2}\f{\cos^{2-\delta}\phi}{\sqrt{1-q^2\cos^2\phi}}\,d\phi\le \int_0^{\pi/2}\f{1}{\sqrt{1-q^2(1-\rr{\f{2}{\pi}\phi}^2)}}\,d\phi=\sqrt{|m|^2+1}\int_0^{\pi/2}\f{1}{\sqrt{1+\rr{\f{2}{\pi}\phi}^2|m|^2}}\,d\phi\\ &
=\f{\pi\sqrt{|m|^2+1}}{2|m|}\log\rr{|m|+\sqrt{|m|^2+1}}\ls \log(\br{m}).
\ea
\eeq
The last inequality uses $|m|\ge 1$. Hence 
 $I_1\ls \f{\log(\br{m})}{\br{m}^{3-\delta}}.$

We now consider the integral $I_2$. First, observe that for $A+B^2$:
\beq
A+B^2=r^2+|m|^2+1+4r^2|m|^2\ge r^2(|m|^2+1).
\eeq
Then $I_2$ can be bounded by
\beq
I_2\le \int_0^\infty \f{r^{2+\delta}}{A^{3/2}\sqrt{r^2(|m|^2+1)}}\,dr=\f{1}{\sqrt{|m|^2+1}}\int_0^\infty \f{r^{1+\delta}}{(r^2+|m|^2+1)^{3/2}}\,dr.
\eeq
Applying the change of variable $t=\f{r^2}{|m|^2+1}$, we have $dt=\f{2r}{|m|^2+1}\,dr$. For $0\le\delta<1$,
\beq
I_2\le\f{1}{2\rr{|m|^2+1}^{1-\f{\delta}{2}}}\int_0^\infty \f{t^{\f{\delta}{2}}}{(t+1)^{3/2}}\,dt\ls \f{1}{\rr{|m|^2+1}^{1-\f{\delta}{2}}}.
\eeq

Finally, combining the estimates for $I_1$ and $I_2$ yields
\beq
S\ls I_1+I_2\ls \f{\log(\br{m})}{\br{m}^{3-\delta}}+\f{1}{\rr{|m|^2+1}^{1-\f{\delta}{2}}}\ls \br{m}^{-2+\delta}.
\eeq
This establishes \eqref{eq:betam}.

We observe that this result relies only on \eqref{eq:betalm} from part (1), and the expression is symmetric in $l$ and $m$. Therefore, interchanging $l$ and $m$ gives \eqref{eq:betal}. This completes the proof of part (2).

\subsection{Propagation of regularity for the Boltzmann equation}
Let $f$ be the solution to the Boltzmann equation \eqref{1} with initial data $f_0$. We state the following propagation of regularity results:
\begin{thm}\label{thm:Hs}(\cite{MV} Theorem 5.4)
Let $0 \le f_0 \in L_2^1$ be an initial datum with finite mass and kinetic energy, and let $f$ be the unique solution preserving energy. Then for all $s>0$ and $\eta \ge \beta$, there exists $w(s)>0$ such that
$$
f_0 \in H_{\eta+w}^s \Longrightarrow \sup _{t \ge 0}\|f(t, \cdot)\|_{H_\eta^s}<+\infty.
$$
\end{thm}
Here the framework of \cite{MV} covers hard potentials $0<\gamma<2$, and the admissible auxiliary weight given by Proposition 5.2 there is, in three dimensions, $\beta=\f32+\gamma+\varepsilon$ for any $\varepsilon>0$ (with the choice $\alpha=\gamma$).
For $s\ge0$ and $\eta\ge\eta'$, the norm convention of Section~1.2 gives $\|g\|_{H^s_{\eta'}}\le C_{s,\eta,\eta'}\|g\|_{H^s_\eta}$: multiplication by $\wei^{\eta'-\eta}$ is bounded on $H^s$, by the product rule for integer $s$ and interpolation.
Theorem \ref{thm:Hs} yields, for every target weight $\ell$, the uniform bound $\sup_{t\ge0}\|f(t)\|_{H^s_\ell}<+\infty$ as soon as $f_0\in H^s_{\eta+w(s)}$ with $\eta\ge\max\{\ell,\f32+\gamma+\varepsilon\}$; in particular, the assumption $f_0\in\bigcap_{\ell\ge0}H^1_\ell$ of Proposition \ref{thm:fHnl} suffices.
For Maxwell molecules $\gamma=0$, which are outside the framework of \cite{MV}, we use the polynomial-weight propagation recorded in Section~4 of \cite{AGT} as following by analogous arguments to its exponential-weight analysis. This is applied with all polynomial $H^1$ weights in the standing assumption of Proposition \ref{thm:fHnl}. The separately stated exponential result below retains its exponential initial-data condition. The Maxwell arguments also use finite initial entropy, which is available here.

\begin{thm}\label{thm:AGTsob}(\cite{AGT}, Theorem 4.2, the case $0<s\le1$)
Let $B(z,\sigma)=|z|^\gamma b(\cos\theta)$ with $\gamma\in[0,1]$ and $b\in L^1(\S^2)$, and let $f$ be the conservative solution of \eqref{1} with nonnegative initial datum $f_0\in L^1_2(\R^3)$, with finite initial entropy in the Maxwell case. Suppose that, for some $s\in(0,1]$, $\alpha\in(0,1]$ and $a_0>0$,
\beq
e^{a_0\langle v\rangle^\alpha}\langle D\rangle^s f_0\in L^2(\R^3).
\eeq
Then there exists $a\in(0,a_0]$ such that, for every $r\in(0,a)$,
\beq
\sup_{t\ge0}\|e^{r\langle\cdot\rangle^\alpha}\langle D\rangle^s f(t)\|_{L^2}
\le C(f_0)+\|e^{r\langle\cdot\rangle^\alpha}\langle D\rangle^s f_0\|_{L^2},
\eeq
where the constant may depend on $\gamma,s,r,\alpha$ and the lower-order weighted norms of the initial datum.
\end{thm}
The quoted theorem places the weight outside the Bessel operator. For $0<s\le1$, Lemma~A.4 of \cite{AGT}, together with the $p=2$ cases of its Theorems~3.1 and~3.3, converts the estimate to our convention $\|\langle D\rangle^s(e^{r\langle\cdot\rangle^\alpha}f)\|_{L^2}$ after decreasing the propagated rate if necessary so that $0<r<\min\{a,1/2\}$. Polynomial-weight propagation is recorded separately in Section~4 of \cite{AGT} as following by analogous arguments. It is this polynomial-weight statement that is used for $\gamma=0$ in Proposition \ref{thm:fHnl}, under the standing assumption $f_0\in\bigcap_{\ell\ge0}H^1_\ell$ and finite initial entropy. The exponential theorem above requires its stated exponential initial norm. The hard-potential range is covered by Theorem \ref{thm:Hs}.

\begin{thm}\label{thm:ex}(\cite{AGT} Theorems 3.1 and 3.3) Suppose that $B(z, \sigma)=|z|^\gamma b(\cos \theta)$ with $b\in L^1(\S^2)$ and either $0<\gamma\leq 2$ (hard potentials) or $\gamma=0$ (Maxwell molecules). Given a non-negative initial datum $f_0\in L^1_2(\mathbb{R}^3)$ with $f_0\,e^{a_0\langle v\rangle^{\alpha}}\in L^1(\mathbb{R}^3)$ for some $\alpha\in(0,2]$, $a_0>0$ --- and, in the Maxwell case $\gamma=0$, additionally $\int_{\mathbb{R}^3}f_0|\log f_0|\,dv<\infty$ --- there exist constants $a\in(0,a_0)$ and $C>0$, depending on the initial mass and energy, on $a_0$, on the stated norms of $f_0$, on $\gamma$ and on $b$ (and, for $\gamma=0$, on the initial entropy), such that the solution $f(t,v)$ of Eq. \eqref{1} satisfies
	$$
	\sup_{t\ge0}\int_{\mathbb{R}^3} e^{a\langle v\rangle^\alpha} f(t,v)\,dv \leq C.
	$$
\end{thm}
We record why this propagation result is used in place of the generation-type exponential-moment estimates (e.g.\ Corollary~1.4 of \cite{LM}): generation-type bounds are uniform only on time intervals $[t_0,\infty)$ with $t_0>0$, their constants blowing up as $t\downarrow0$, whereas the proof of Theorem \ref{thm:mainresult}(ii) requires a bound uniform on $[0,\infty)$, which is provided by the propagation above at the cost of replacing the rate $a_0=c$ by some $a<c$. We also record that the two cases in the statement above are not interchangeable: the hard-potential result (Theorem~3.1 of \cite{AGT}) requires $\gamma>0$, while for Maxwell molecules the tails are propagated rather than generated and the corresponding statement (Theorem~3.3 of \cite{AGT}) carries the additional finite-entropy requirement on the initial datum, which is satisfied here since $f_0\in L\log L$. For Maxwell molecules, the $p=2$ case of Theorem~3.3 of \cite{AGT} supplies the zero-order input in the exponential-weight regularity proof of Theorem~4.2 of \cite{AGT}; the polynomial-weight use in Proposition \ref{thm:fHnl} has the separate attribution given above.
Since the statement above covers $\gamma=0$ as well as $\gamma>0$, the hypothesis $\gamma>0$ has been removed from Theorem \ref{thm:mainresult}(ii), which thereby applies to Maxwell molecules as well.

\end{document}